\documentclass[a4paper,10pt]{article}

\usepackage{cmap}
\usepackage[T2A]{fontenc}
\usepackage[utf8]{inputenc}

\usepackage{amsmath}
\usepackage{amsfonts}
\usepackage{amsthm}
\usepackage{amssymb}
\usepackage{mathtools}
\usepackage{dsfont}

\usepackage{icomma}

\usepackage{euscript}
\usepackage{mathrsfs}

\usepackage[left=2cm,right=2cm,top=2cm,bottom=2cm,bindingoffset=0cm]{geometry}

\usepackage{tikz-cd}

\usetikzlibrary{babel}
\usetikzlibrary{patterns}

\usepackage{quiver}

\usepackage[unicode, pdftex]{hyperref}
\hypersetup{
    colorlinks=true,
    allcolors=purple
    }

\usepackage{indentfirst}

\usepackage{autonum}

\usepackage{enumitem}

\newtheorem{theorem}{Theorem}[section]
\newtheorem{proposition}[theorem]{Proposition}
\newtheorem{lemma}[theorem]{Lemma}
\newtheorem{corollary}[theorem]{Corollary}

\theoremstyle{definition}
\newtheorem{definition}[theorem]{Definition}
\newtheorem{remark}[theorem]{Remark}
\newtheorem{example}[theorem]{Example}

\newcommand{\rightarrowdbl}{\rightarrow\mathrel{\mkern-14mu}\rightarrow}

\newcommand{\xrightarrowdbl}[2][]{%
  \xrightarrow[#1]{#2}\mathrel{\mkern-14mu}\rightarrow
}

\usepackage{epigraph}

\usepackage{etoolbox}
\makeatletter
\newlength\epitextskip
\pretocmd{\@epitext}{\em}{}{}
\apptocmd{\@epitext}{\em}{}{}
\patchcmd{\epigraph}{\@epitext{#1}\\}{\@epitext{#1}\\[\epitextskip]}{}{}
\makeatother

\let\cal\mathcal
\def\AA{{\cal A}}
\def\BB{{\cal B}}
\def\CC{{\cal C}}
\def\DD{{\cal D}}

\def\YY{{\cal Y}}

\def\tYY{{\Tilde{ \cal Y}}}

\def\Hom{\operatorname{Hom}}
\def\Ext{\operatorname{Ext}}
\def\Tor{\operatorname{Tor}}
\def\End{\operatorname{End}}
\def\coker{\operatorname{coker}}
\def\im{\operatorname{im}}
\def\Im{\operatorname{Im}}
\def\res{\operatorname{res}}

\def\Id{\operatorname{Id}}

\def\Ind{\operatorname{Ind}}
\def\Mod{\operatorname{Mod}}
\def\mod{\operatorname{mod}}
\def\Proj{\operatorname{Proj}}
\def\proj{\operatorname{proj}}
\def\Fun{\operatorname{Fun}}
\def\Coh{\operatorname{Coh}}
\def\Ker{\operatorname{Ker}}
\def\Coker{\operatorname{Coker}}

\newcommand{\colim}[1][]{\mathop{\operatorname{colim}}\limits_{#1}}

\newcommand{\indlim}[1][]{\mathop{\varinjlim}\limits_{#1}}

\begin{document}

\title{ \bfseries Local coherence and recollement of abelian categories}
\author{ \textsc{mikhail dusman\textsuperscript{1}, alexander  pavlov\textsuperscript{1,2}} \\
\small \textsuperscript{1}Moscow Institute of Physics and Technology (MIPT), Center of Fundamental Mathematics \\
\small \textsuperscript{2}Faculty of Mathematics, HSE University }

\date{}
\maketitle

\begin{abstract}
We study how local coherence of abelian categories behaves under gluing. For AB5 locally finitely presentable abelian categories, we give conditions on the gluing functors under which local coherence is transferred along a gluing $(\AA, \BB, \CC)$: from $\AA$ and $\CC$ to $\BB$, and conversely from $\BB$ to $\AA$ or $\CC$. The coherence theorem for augmented homotopes of Bondal and Zhdanovskiy is recovered as a special case, the coherence theorem for trivial extensions of Fossum, Griffith, and Reiten is a companion result in the same direction. We also introduce quasiprojective pairs of adjoint functors and prove a coherence criterion that generalises the criterion of Gomez Pardo and Martinez Hernandez to arbitrary Grothendieck categories. Finally, we extend Stenstr{\"o}m's characterisation of coherence via FP-injective modules to compactly assembled Grothendieck categories: such a category is coherently assembled if and only if the class of CM-injective objects, defined through compact morphisms, is closed under filtered colimits. Along the way we revisit the various equivalent conditions for the coherence of a ring, collecting them in a single place. The proofs make use of the compactly assembled categories of Efimov.
\end{abstract}



\tableofcontents

\section*{Introduction}
\addcontentsline{toc}{section}{Introduction}

\subsection*{Motivation}

Let $R$ be an associative ring. It is \textit{left coherent} when every finitely generated left ideal admits a finite presentation as a left $R$-module. Categorically, this is the requirement that $R - \Mod$ be \textit{locally coherent}: locally finitely presentable, with its finitely presented objects forming an abelian subcategory. As a finiteness condition, coherence is far more elastic than Noetherianness; it accommodates, for instance, polynomial rings in infinitely many variables $k[x_i \mid i \in S]$, which the Noetherian setting cannot.

The natural categorical setting for coherence is a Grothendieck abelian category, or more generally an AB5 locally finitely presentable abelian category. There, the question of how local coherence behaves under the standard operations on abelian categories acquires content. The operation we study is the formation of a recollement: a localisation sequence
\[
\AA \xrightarrow{i_*} \BB \xrightarrow{j^*} \CC
\]
of abelian categories, together with adjoint data exhibiting $\AA$ as a Serre subcategory of $\BB$ with quotient $\CC$. Recollement categorifies an idempotent decomposition of a ring; it arises for Morita rings $\bigl(\begin{smallmatrix} R & N \\ M & S \end{smallmatrix}\bigr)$, in representation theory, and in the BBD framework for perverse sheaves on singular spaces.

Our question is how local coherence of $\BB$ relates to that of $\AA$ and $\CC$. In one direction, local coherence of $\BB$ restricts to $\AA$ and $\CC$ under mild assumptions; in the other, recovering it for $\BB$ from both ends is more delicate and typically requires additional finiteness assumptions on the gluing data.

Treating the question categorically has two payoffs. First, it places under one roof two results previously stated separately in the language of rings and modules: the coherence theorem for trivial extensions of an abelian category by a right exact endofunctor, due to R. Fossum, P. Griffith, and I. Reiten~\cite{reiten}, and the coherence theorem for the augmented homotope $\hat A_\Delta = A_\Delta \oplus k$ of a well-tempered element $\Delta \in A$, due to A. Bondal and I. Zhdanovskiy~\cite{Bondal_Zhdanovskiy_21}. Second, it suggests a natural criterion for the converse implication, parallel to the quasiprojective-module criterion of J.L. Gomez Pardo and J. Martinez Hernandez~\cite{PARDO_HERNANDEZ}; this is the content of Theorem~\ref{th_quasiproj_criterion} below.

\subsection*{Main results}

We first formulate the direct course, in which local coherence of $\BB$ is deduced from local coherence of $\AA$ and $\CC$. We recall that an accessible abelian category $\AA$ is termed \textit{locally coherent} provided it is locally finitely presentable and the finitely presented objects $\AA^{f.p.}$ assemble into an abelian subcategory; it is called \textit{locally Noetherian} provided each of its $\omega$-presented objects is Noetherian.

\begin{theorem}[Theorem~\ref{th_coh_B_straightforward}]
Let $\AA$, $\BB$, $\CC$ be AB5 locally finitely presentable abelian categories, together with adjoint functors
\begin{equation}
\begin{tikzcd}
	\AA &&& \BB &&& \CC
	\arrow["{i_*}"{description}, from=1-1, to=1-4]
	\arrow["{i^*}"{description}, shift right=5, from=1-4, to=1-1]
	\arrow["{j^*}"{description}, from=1-4, to=1-7]
	\arrow["{j_!}"{description}, shift right=5, from=1-7, to=1-4]
\end{tikzcd}
\end{equation}
such that $i_*$ is a fully faithful embedding of $\Ker(j^*)$, $j_!$ is fully faithful, the functors $i_*$ and $j^*$ are exact, commute with filtered colimits, and preserve finitely presented objects, and $\mathbb{L}_1 i^*$ sends finitely generated subobjects of finitely presented objects to finitely generated objects. If $\AA$ is locally Noetherian and $\CC$ is locally coherent, then $\BB$ is locally coherent.
\end{theorem}

The first specialisation of this theorem concerns the standard recollement attached to an idempotent ideal of a ring.

\begin{corollary}[Corollary~\ref{cor_B_B+_situation}]
Let $R$ be an associative ring, let $I \subset R$ be a finitely generated projective idempotent two-sided ideal, and set $k = R/I$ and $E = \End_R(I)$. If $k$ is left Noetherian and $E$ is left coherent, then $R$ is left coherent.
\end{corollary}

Applying Corollary~\ref{cor_B_B+_situation} to the homotope $\hat A_\Delta = A_\Delta \oplus k$ associated with a well-tempered element $\Delta \in A$ recovers, in the categorical bracket of the standard setup, the theorem of Bondal and Zhdanovskiy. Here a homotope $A_\Delta$ is the $k$-module $A$ with multiplication $x \cdot_\Delta y = x \Delta y$, and $\Delta$ is \textit{well-tempered} if the augmentation ideal $B_+ = A_\Delta$ is projective and idempotent.

\begin{theorem}[\cite{Bondal_Zhdanovskiy_21}, Theorem~\ref{th_bondal_zhdanovskiy}]
Let $\Delta \in A$ be well-tempered and let $B = \hat A_\Delta$. If $A$ is a coherent algebra over a Noetherian commutative ring $k$, then $B$ is coherent.
\end{theorem}

The Noetherianness of $\AA$ in Theorem~\ref{th_coh_B_straightforward} is used only to control the failure of exactness of $i^*$. When $i^*$ is itself exact, this hypothesis can be weakened to local coherence.

\begin{theorem}[Theorem~\ref{th_coh_B_i*exact}]
With the assumptions of Theorem~\ref{th_coh_B_straightforward}, but with $i^*$ exact in place of the hypothesis on $\mathbb{L}_1 i^*$, local coherence of $\AA$ and $\CC$ implies local coherence of $\BB$.
\end{theorem}

The reverse course --- restriction of local coherence from $\BB$ to $\AA$ or $\CC$ --- is treated in Theorems~\ref{th_revcourse_coh_j^*j_*}, \ref{th_revcourse_coh_j_!j^*}; these give one-sided implications. For a localizing Serre subcategory of finite type this restriction is a theorem of Krause \cite[Theorem 2.6]{KRAUSE_spec_loc_coh_cat}, proved there through the finitely presented objects; our proof only uses the functor $\tYY$ and therefore works for compactly assembled categories, see Remark~\ref{rem_krause_comparison}. To obtain a criterion in this direction we strengthen the gluing data into the notion of a quasiprojective pair, motivated by the work of Gomez Pardo and Martinez Hernandez on the coherence of the endomorphism ring of a quasiprojective module. A pair of adjoint functors $(j_! \dashv j^*) \colon \BB \rightleftarrows \CC$ between Grothendieck abelian categories is \textit{quasiprojective} (Definition~\ref{def_qproj_functors}) if $j_!$ is fully faithful on $\CC^{f.p.}$ and $R^1 j^*(X) = 0$ for every subobject $X$ of an object of $\BB_{\CC^{f.p.}} := j_!(\CC^{f.p.})$.

\begin{theorem}[Theorem~\ref{th_quasiproj_criterion}]
Let $(j_! \dashv j^*)$ be a quasiprojective pair as in Definition~\ref{def_qproj_functors}. The following are equivalent:
\begin{enumerate}
    \item $\CC$ is locally coherent.
    \item For every morphism $f \colon B_1 \to B_2$ in $\BB_{\CC^{f.p.}}$, there exists a morphism $g \colon B_3 \to \ker f$ with $B_3 \in \BB_{\CC^{f.p.}}$ such that $j^*(\ker g) = 0$ and $j^*(\coker g) = 0$.
\end{enumerate}
\end{theorem}

When $\BB = R - \Mod$, $\CC = S - \Mod$, $j_! = M \otimes_S -$ and $j^* = \Hom_R(M, -)$ for a quasiprojective left $R$-module $M$ with $S = \End_R(M)$, the criterion above specialises to \cite[Theorem 1.2]{PARDO_HERNANDEZ}.

Our last main result leaves the locally finitely presentable world. Stenstr{\"o}m~\cite{Stenstrom-1970} characterised coherent rings as those over which filtered colimits of FP-injective modules are FP-injective, and the statement transfers verbatim to AB5 locally finitely presentable abelian categories (Theorem~\ref{th_fpinj_ab}). In a compactly assembled category in the sense of Efimov there may be no non-zero compact objects at all, so FP-injectivity has to be redefined through \textit{compact morphisms}: an object $x$ is \textit{CM-injective} (Definition~\ref{def_cminj}) if $\phi^* \colon \Ext^1(z, x) \to \Ext^1(y, x)$ vanishes for every compact morphism $\phi \colon y \to z$. For locally finitely presentable categories this is FP-injectivity, and local coherence is coherent assembly, i.e.\ left exactness of the functor $\tYY \colon \AA \to \Ind(\AA)$ left adjoint to $\colim$.

\begin{theorem}[Theorem~\ref{th_cminj}]
Let $\AA$ be a compactly assembled Grothendieck abelian category. Then $\AA$ is coherently assembled if and only if every filtered colimit of CM-injective objects is CM-injective.
\end{theorem}

The proof of the direct implication replaces the finitely generated free resolutions of the classical argument by the ind-completion $\Ind(\AA)$, in which the analogue of Stenstr{\"o}m's condition on $\varinjlim \Ext^1$ holds unconditionally; the converse is a categorical form of Stenstr{\"o}m's original argument, carried out at the level of $\tYY$.

\subsection*{Generality of the framework}

The result of Reiten on the coherence of trivial extensions, which we recall from~\cite{reiten}, is the closest forerunner of Theorem~\ref{th_coh_B_straightforward} in the literature. Given an abelian category $\AA$ with a small additive full subcategory $\mathcal{P}$ of projective objects, an object $A$ of $\AA$ is \textit{$\mathcal{P}$-coherent} if it is of finite $\mathcal{P}$-type (i.e.\ a quotient of an object of $\mathcal{P}$) and every subobject of finite $\mathcal{P}$-type is of finite $\mathcal{P}$-presentation. The trivial extension by a right exact endofunctor $F \colon \AA \to \AA$ is the comma category $(F \downarrow \Id_\AA)$; on objects of $\mathcal{P}$ it produces a subcategory $T(\mathcal{P})$.

\begin{theorem}[\cite{reiten}, Theorem~\ref{th_reiten_TP_coherent}]
$T(\mathcal{P})$ is coherent in $(F \downarrow \Id_\AA)$ if and only if:
\begin{enumerate}
\item[(a)] $\mathcal{P}$ is coherent in $\AA$;
\item[(b)] $L_i F(A)$ is $\mathcal{P}$-coherent for every $\mathcal{P}$-coherent $A$ and every $i > 0$;
\item[(c)] every subobject $B$ of $F(A)$ of finite $\mathcal{P}$-type is coherent, and $F(B)$ is of finite $\mathcal{P}$-type.
\end{enumerate}
\end{theorem}

Specialising this result to the trivial extension of a Morita ring yields the following.

\begin{corollary}[\cite{reiten}, Corollary~\ref{cor_reiten_matrix}]
The ring $\bigl(\begin{smallmatrix} R & N \\ M & S \end{smallmatrix}\bigr)$ is left coherent exactly when both $R$ and $S$ are left coherent and, for all finitely presented left $R$-modules $A$ and finitely presented left $S$-modules $B$, the following hold: in each positive degree $i > 0$ the modules $\Tor_i^R({}_S M, A)$ and $\Tor_i^S({}_R N, B)$ are again finitely presented, and any finitely generated $S$-submodule $C$ of $M \otimes_R A$ is itself finitely presented while $N \otimes_S C$ stays finitely generated --- the analogous requirements with $R$ and $S$ interchanged being imposed as well.
\end{corollary}

The corollary admits the following sharp example due to Reiten.

\begin{example}[\cite{reiten}, Example~\ref{ex_reiten_kxi}]
Take $k$ to be a field and set $A = k[x_i \mid i \in S]$ with $S$ an infinite index set, regarded as a $k$-algebra by means of the standard augmentation $A \to k$ that annihilates every $x_i$. The resulting ring $T = \bigl(\begin{smallmatrix} A & 0 \\ k & k \end{smallmatrix}\bigr)$ is then left coherent while failing to be right coherent.
\end{example}

Two observations about this example are in order. The ring $A$ is coherent, although it is not Noetherian; the failure of right coherence of $T$ is therefore a strictly recollement-theoretic phenomenon, located in the gluing data. The recollement attached to the idempotent $\bigl(\begin{smallmatrix} 0 & 0 \\ 0 & 1 \end{smallmatrix}\bigr) \in T$ is, in the notation of Example~(\ref{ex_psa_matrix_ring}), $(\Mod - A, \Mod - T, \Mod - k)$, and it satisfies every hypothesis of Theorem~\ref{th_coh_B_i*exact} apart from the preservation of finitely presented objects. The next remark, which marks the precise point at which Theorem~\ref{th_coh_B_i*exact} is sharp, captures this.

\begin{remark}[Remark~\ref{rem_generality}]
In any recollement of abelian categories, the assumptions on $i_*$ and $j^*$ in Theorem~\ref{th_coh_B_i*exact} --- exactness, commutation with filtered colimits, preservation of finitely presented objects --- are automatic with the exception of the preservation of finitely presented objects. Example~\ref{ex_reiten_kxi} shows that this last assumption cannot be removed.
\end{remark}

The generality of Theorem~\ref{th_coh_B_straightforward} and Theorem~\ref{th_coh_B_i*exact} is, in this sense, not formal: the conclusion fails as soon as the finite-presentation hypothesis on $j^*$ is dropped, and the relevant counterexamples come from the very ring-theoretic situation that motivates the question.

\subsection*{Related work}

The systematic theory of coherent rings goes back to S.U. Chase~\cite{Chase1960DirectPO}, who isolated the notion to characterise the rings over which products of flat modules stay flat; the standard module-theoretic descriptions are collected in the textbook of Stenstr{\"o}m~\cite{stenstrom1975rings}. The interplay with injectivity is due to B. Stenstr{\"o}m~\cite{Stenstrom-1970}, who proved that coherence is equivalent to the closure of FP-injective modules under filtered colimits (Theorem~\ref{th_cminj} extends this equivalence to compactly assembled Grothendieck categories); the character-module form of these equivalences is due to T. Cheatham and D. Stone~\cite{CheathamStone}, J. Lambek~\cite{Lambek_1964}, T. W\"urfel~\cite{Wurfel1973}, and E. Matlis~\cite{Matlis_1982}. M. Hovey~\cite{Hovey2001} recast coherence through wide subcategories of $R - \Mod$, and H. Krause through the spectrum of a locally coherent Grothendieck category~\cite{KRAUSE_spec_loc_coh_cat, Krause2022}; the same paper of Krause contains the theorem that finite type localizing subcategories and quotients of locally coherent categories are locally coherent, which is the closest forerunner of our reverse course (Remark~\ref{rem_krause_comparison}). We use the locally finitely presentable formalism of Ad\'amek--Rosick\'y~\cite{Adamek_Rosicky_1994} and Popescu~\cite{popescu1973}, and the \textit{compactly assembled} categories of Efimov~\cite{efimov_Kth_large_cat, efimov_Th_of_heart}, which place coherence in the broader accessible-category context that we use throughout.

The hierarchy of finiteness conditions $\mathcal{FP}_n$ has its modern formulation in the work of D. Bravo and M. P\'erez~\cite{FPn_first}; the $n$-coherent rings of E. Ellis and R. Parra~\cite{k-th_of_n-coherent_rings}, whose $K$-theory extends that of Schlichting~\cite{Schlichting04}, and the categorical version of D. Bravo, J. Gillespie, and M. P\'erez~\cite{Bravo_Becker_loctypeFPn} are the natural generalisations. Locally coherent Grothendieck categories were studied by I. Herzog~\cite{Herzog_Ziegler} through the Ziegler spectrum; the local coherence of hearts of $t$-structures by M. Saor{\'\i}n~\cite{Saorin_loccoh_hearts} and C. Parra, M. Saor{\'\i}n, and S. Virili~\cite{Parra_Saorin_Virili}; and L. Positselski~\cite{Positselski_loccoh_exact} extends the theory to exact categories.

Trivial extensions of abelian categories are surveyed in the expository account of Fossum, Griffith, and Reiten~\cite{reiten}. The recollement formalism, set up by Beilinson, Bernstein, and Deligne for perverse sheaves and transported to the abelian framework by V. Franjou and T. Pirashvili~\cite{Pirashvili_recollement_ab}, was developed homologically by C. Psaroudakis~\cite{PSAROUDAKIS_homolog_theor}, C. Parra and J. Vit\'oria~\cite{Psaroudakis_props_via_recoll}, and C. Psaroudakis and J. Vit\'oria~\cite{Psaroudakis_modul}, with the theory of \textit{ladders} due to N. Gao, S. K\"onig, and C. Psaroudakis~\cite{ladders}. Our ring-theoretic motivation comes from A. Bondal and I. Zhdanovskiy~\cite{Bondal_Zhdanovskiy_21} on homotopes and from Gomez Pardo and Martinez Hernandez~\cite{PARDO_HERNANDEZ} on endomorphism rings of quasiprojective modules; the $K$-theory of trivial extensions and gluings was studied recently by Q. Chen and M. Zheng~\cite{ChenXi_Kgroups_trivial_ext}.

\subsection*{Overview of the structure of the paper}

We give a brief overview of the structure of the paper.

Section~\ref{sec:loccoh} is devoted to the theory of locally coherent abelian categories. In Subsection~\ref{ssec:fg_fp} we recall the notions of finitely generated and finitely presented objects, the equivalent characterisations of these notions in an AB5 category (Theorem~\ref{th_fg_fg_fp}), and the closure of these properties under extensions, kernels, and cokernels (Lemma~\ref{lem_ab_fp_fg_via_ker_coker}). Subsection~\ref{ssec:ind} recalls the ind-completion of an abelian category, the functors $\tYY$ associated with a compactly assembled category in the sense of Efimov, and proves the commutation lemmata (Lemma~\ref{lem_commutativity_of_ind} and Lemma~\ref{lem_commutativity_of_ind(spec)}) that we use throughout the paper. Subsection~\ref{ssec:loccoh_def} introduces local coherence, summarises its equivalent characterisations, and contains Theorem~\ref{th_MEGA} --- seventeen equivalent formulations of left coherence of a ring --- including the Chase, Stenstr{\"o}m, Cheatham--Stone, Matlis, Krause, and Hovey criteria, and the categorical formulation via coherent assembly due to Efimov. The proof of Theorem~\ref{th_MEGA} is given in this subsection. Subsection~\ref{ssec:ext} introduces comma categories and recollement situations of abelian categories, following \cite{Pirashvili_recollement_ab, PSAROUDAKIS_homolog_theor, ladders}, and lists the structure theorems we use in Section~\ref{sec:thms}.

Section~\ref{sec:thms} contains the main coherence results. Subsection~\ref{ssec:direct} is the direct course: from local coherence of $\AA$ and $\CC$ we deduce local coherence of $\BB$. The main theorems are Theorem~\ref{th_coh_B_straightforward} (when $\mathbb{L}_1 i^*$ controls finitely generated subobjects of finitely presented objects) and Theorem~\ref{th_coh_B_i*exact} (when $i^*$ is exact); their ring-theoretic specialisations include Corollary~\ref{cor_B_B+_situation} and, as a further specialisation, the theorem of Bondal and Zhdanovskiy on homotopes (Theorem~\ref{th_bondal_zhdanovskiy}). The same subsection contains the coherence theorem of Reiten for trivial extensions of abelian categories (Theorem~\ref{th_reiten_TP_coherent}), its matrix-ring corollary (Corollary~\ref{cor_reiten_matrix}), and Example~\ref{ex_reiten_kxi} together with Remark~\ref{rem_generality}, which together pin down the sharpness of Theorem~\ref{th_coh_B_i*exact}. Subsection~\ref{ssec:reverse} is the reverse course: from local coherence of $\BB$ we deduce local coherence of $\AA$ or of $\CC$, under finiteness assumptions on the relevant adjoint pair (Theorems~\ref{th_revcourse_coh_j^*j_*}, \ref{th_revcourse_coh_j_!j^*}); Remark~\ref{rem_krause_comparison} compares these results with the theorem of Krause~\cite{KRAUSE_spec_loc_coh_cat} on finite type localizing subcategories. Subsection~\ref{ssec:qproj} introduces the notion of a quasiprojective pair of adjoint functors (Definition~\ref{def_qproj_functors}) and proves the criterion for local coherence stated in Theorem~\ref{th_quasiproj_criterion}, which generalises the criterion of~\cite{PARDO_HERNANDEZ} to arbitrary Grothendieck abelian categories. Subsection~\ref{ssec:fpinj} treats the FP-injective viewpoint: we recall the Stenstr{\"o}m equivalences in their categorical form (Theorem~\ref{th_fpinj_ab}) and use them to obtain a coherence transfer result under hypotheses on the right adjoint pairs $(i_* \dashv i^!)$ and $(j^* \dashv j_*)$ together with the preservation of FP-injective objects by all four functors. Subsection~\ref{ssec:cminj} extends the Stenstr{\"o}m criterion to compactly assembled Grothendieck abelian categories: we introduce CM-injective objects (Definition~\ref{def_cminj}), show that they coincide with FP-injective objects in the locally finitely presentable case (Proposition~\ref{prop_cminj_vs_fpinj}), and prove that coherent assembly is equivalent to the closure of CM-injective objects under filtered colimits (Theorem~\ref{th_cminj}). The subsection ends with a discussion of the condition that kernels of morphisms between compact morphisms be compact: it implies coherent assembly, but in the locally finitely presentable case it is equivalent to local Noetherianity (Propositions~\ref{prop_NA_cohass}, \ref{prop_NA_lfp}). The basic facts on $\tYY$ and compact morphisms needed there are collected in Remarks~\ref{rem_tYY_basic} and \ref{rem_compact_lifts} of Subsection~\ref{ssec:ind}.

\subsection*{Acknowledgements}
The authors are grateful to Alexander Efimov and Alexey Bondal for fruitful discussions and comments.

The research is supported by the Ministry of Science and Higher Education
of the Russian Federation (Goszadaniye) 075-00337-20-03, project No.
075-00141-26-00.

\subsection*{Conventions and notation}

\begin{itemize}
    \item We use the word ``category'' to mean category.\footnote{or more precisely, $1$-category.}

    \item Unless there is a specific need for it, we do not address set-theoretic issues.

    \item All rings are assumed to be associative and unital, unless otherwise stated.
\end{itemize}

\section{Locally coherent abelian categories} \label{sec:loccoh}

\subsection{Finitely generated and finitely presented objects} \label{ssec:fg_fp}

We want to recall the definitions of finitely generated and finitely presented objects and 
locally finitely presentable categories.

Let $\CC$ be a cocomplete category, and let $X,\, \{Y_i\}_{i \in I} \in \CC$. There is a canonical map
\begin{equation}
\colim[i] \Hom(X,\, Y_i) \xrightarrow{t_X} \Hom(X,\, \colim[i] Y_i).
\end{equation}
It is obtained by taking the colimit of the maps $\Hom(X,\, Y_i) \to \Hom(X,\, \colim[i] Y_i)$ induced by the canonical morphisms $Y_i \to \colim[i] Y_i$.

\begin{definition}
Let $\CC$ be a cocomplete category and let $\kappa$ be a regular cardinal. An object $X \in \CC$ is said to be {\it $\kappa$-type} (resp.\ {\it $\kappa$-compact}) if, for any $\kappa$-filtered system $\{Y_i\}_{i \in I}$, the map $t_X$ is injective (resp.\ bijective).

The full subcategory of $\kappa$-compact objects of $\CC$ will be denoted $\CC^{\kappa}$.
\end{definition}

\begin{remark} \label{rem_filtered_convention}
We call the object $X$ {\it finitely presented} if it is $\omega$-compact, and {\it finitely generated} or {\it of finite type} if it is $\omega$-type.

We denote the full subcategory of $\CC$ spanned by the finitely presented (resp.\ finitely generated) objects by $\CC^{f.p.}$ or $\CC^{\omega}$ (resp.\ $\CC^{f.g.}$).

Note that, by \cite[Theorem~1.5]{Adamek_Rosicky_1994}, for every small $\omega$-filtered category $\DD$ there is an $\omega$-directed cofinal subposet $\DD_0$.
Thus, without loss of generality, we use $\omega$-filtered diagrams rather than $\omega$-directed ones, and we refer to them simply as {\it filtered}.
\end{remark}

We are now interested in equivalent characterisations of finitely generated and finitely presented objects.
Let us first define the class of categories in which such characterisations can be established.

\begin{definition}\cite{Adamek_Rosicky_1994}
Let $\kappa$ be a regular cardinal.

A category $\CC$ is called {\it $\kappa$-accessible} if $\CC$ has $\kappa$-filtered colimits and there is a set $A$ of $\kappa$-compact objects of $\CC$ such that every object of $\CC$ is a $\kappa$-filtered colimit of objects from $A$.

A category $\CC$ is called {\it locally $\kappa$-presentable} if it is cocomplete and $\kappa$-accessible.

A category $\CC$ is called {\it locally $\kappa$-type} if it is cocomplete and has a set $A$ of $\kappa$-type objects such that every object is a $\kappa$-filtered colimit of objects from $A$.
\end{definition}

\begin{remark}
We say that $\CC$ is {\it locally finitely presentable} if it is locally $\omega$-presentable, that $\CC$ is {\it locally of finite type} if it is locally $\omega$-type, and that $\CC$ is {\it accessible} if it is $\kappa$-accessible for some regular cardinal $\kappa$.
\end{remark}

The following theorem is due to \cite{Adamek_Rosicky_1994}; the formulation given here is taken from \cite{kato}.
\begin{theorem}\cite[Proposition 1.69]{Adamek_Rosicky_1994} \label{th_locpres_fg}
Let $\CC$ be a locally finitely presentable pointed category, $X \in \CC$. The following are equivalent:
\begin{enumerate}
\item $X$ is finitely generated.
\item For any filtered  system $\{Y_i\}$ of objects of $\CC$, all of whose transition morphisms are monomorphisms, $t_X$ is bijective.
\item There is an epimorphism $F \rightarrowdbl X$ from a finitely presented object $F$.
\end{enumerate}
\end{theorem}

We now describe equivalent characterisations of finitely generated and finitely presented objects that exist in abelian categories.

\begin{theorem}\cite[Proposition 5.6.]{popescu1973}
Let $X$ be an object in an AB5-category. The following are equivalent:
\begin{enumerate}
    \item $X$ is finitely generated
    \item Any complete filtered set $\{X_i\}_i$ of subobjects of $X$ is stationary (that is, $X = X_{i_0}$ for some index $i_0$).
\end{enumerate}
\end{theorem}

It is a folklore result in algebra that a module $M$ is finitely presented if and only if it is finitely generated and, for every epimorphism $N \xrightarrowdbl{\vartheta} M$ from a finitely generated module $N$, the kernel $\ker(\vartheta)$ is finitely generated.

It turns out that this characterisation remains true in any locally small, locally finite-type AB5 category.
\begin{theorem}\cite[Theorem~5.10.]{popescu1973} \label{th_fg_fg_fp}
Let $\CC$ be a locally small, locally finite-type AB5 category, and let $X \in \CC$. The following are equivalent:
\begin{enumerate}
    \item $X$ is finitely presented.
    \item \label{item_almostalmost_fp} $X$ is finitely generated and, for every epimorphism $Z \xrightarrowdbl{\vartheta} X$ from a finitely generated object $Z$, the kernel $\ker(\vartheta)$ is finitely generated.
\end{enumerate}
\end{theorem}

Continuing the analogy with algebra we point out some more statements investigating the behaviour of finitely presented and finitely generated objects under kernels, cokernels and extensions.

\begin{lemma} \label{lem_ab_fp_fg_via_ker_coker}
Let $\AA$ be an AB5 category, and let

\[
0 \to M_1 \to M_2 \to M_3 \to 0
\]

\[
N_1 \to N_2 \to N_3 \to 0
\]

be exact sequences. Then

\begin{enumerate}
    \item If $N_1$ and $N_3$ are finitely generated objects, then $N_2$ is also finitely generated.
    \item If $M_2$ is a finitely generated object, then $M_3$ is also finitely generated.
    \item If $M_2$ is a finitely presented object and $M_1$ is a finitely generated object, then $M_3$ is a finitely presented object.
    \item If $M_1$ and $M_3$ are finitely presented objects, then $M_2$ is also finitely presented.
\end{enumerate}
\end{lemma}

\begin{proof}
Items (1), (2), and (3) follow easily from the $5$-lemma.
To deduce item (4) i.e. isomorphism $\indlim \Hom (M_2, X_i) \to \Hom(M_2, \indlim X_i)$ from $5$-lemma it is sufficient to show that the natural map $\indlim \Ext^1 (M_2, X_i) \to \Ext^1(M_2, \indlim X_i)$ is a monomorphism. This holds by the same argument as in \cite[Proposition 2.1.]{Stenstrom-1970}.
\end{proof}

Note that the form of the previous lemma holds if $\kappa$-compact and $\kappa$-type objects are considered instead, for some regular cardinal $\kappa$, and $\AA$ is an exact category possessing $\kappa$-filtered colimits which are exact. In this form item (4) of the previous lemma appears in \cite{efimov_Th_of_heart}.

\begin{proposition}\cite[Proposition 1.14.]{efimov_Th_of_heart}
Let $\mathcal{E}$ be an accessibly exact category and let $\kappa$ be a regular cardinal.  
Suppose that $\mathcal{E}$ has exact $\kappa$-filtered colimits. Then the full subcategory $\mathcal{E}^\kappa \subset \mathcal{E}$ is closed under extensions.
\end{proposition}

We will need the following common notation for modules over an additive category.

\begin{definition}
Let $\CC$ be an additive category. Denote by $\Mod \CC$ the category of
additive functors $\CC^{\mathrm{op}} \to \mathbf{Ab}$, morphisms between such functors are natural transformations.

Fix a regular cardinal $\kappa$. We write  
\[
\mod_\kappa \CC := (\Mod \CC)^\kappa
\]  
and  
\[
\proj_\kappa \CC := \Proj \CC \cap \mod_\kappa \CC,
\]  

where $\Proj \CC$ denotes the full subcategory of projective objects in $\Mod \CC$.

As above we use $\mod \CC$ for $\mod_\omega \CC$ which coincides with {\it finitely presented} functors, and $\proj \CC$ for $\proj_\omega \CC$. 
\end{definition}

\subsection{Ind-completion of abelian categories} \label{ssec:ind}
We recall the definition and the basic properties of the Ind-completion.

\begin{definition}
The category $\Ind(\mathcal{A})$ of \textit{ind-objects} of $\mathcal{A}$ is defined as follows:
\begin{itemize}
\item Objects are filtered systems $\{X_i\}_{i \in I}$ in $\mathcal{A}$.
\item Morphisms between two such systems $\{X_i\}_{i \in I}$ and $\{Y_j\}_{j \in J}$ are given by
\[
\Hom_{\Ind(\mathcal{A})}\bigl(\{X_i\}_{i \in I}, \{Y_j\}_{j \in J}\bigr) = \lim_{i \in I} \colim[j \in J] \Hom_{\mathcal{A}}(X_i, Y_j).
\]
\item Composition is defined in the natural way.
\end{itemize}
\end{definition}

\begin{remark} \label{rem_ind_basic}
If $\mathcal{A}$ is abelian, then $\Ind(\mathcal{A})$ is abelian.

Filtered colimits in $\Ind(\mathcal{A})$ are exact. In particular, filtered colimits commute
with finite limits and finite colimits.

We write $\YY_{\AA}\colon\AA\to\Ind(\AA)$ for the canonical embedding, also called the
\emph{Yoneda embedding}. It sends an object $X\in\AA$ to the constant ind-object $\{X\}$
(indexed by the one-point category). It is fully faithful and exact, and identifies $\AA$
with a full subcategory of $\Ind(\AA)$. When the category is clear from the context we omit
the subscript and write simply $\YY$.

Whenever $\AA$ admits (small) filtered colimits, the Yoneda embedding has a left adjoint
\begin{equation}
\begin{tikzcd}
	\AA &&& {\Ind(\AA)}
	\arrow[""{name=0, anchor=center, inner sep=0}, "\YY"{description}, shift right=5, from=1-1, to=1-4]
	\arrow[""{name=1, anchor=center, inner sep=0}, "\colim"{description}, from=1-4, to=1-1]
	\arrow["\dashv"{anchor=center, rotate=-90}, shift left=5, draw=none, from=1, to=0]
\end{tikzcd}
\end{equation}
namely the functor $\colim$ that sends a formal filtered colimit $\{X_i\}_{i\in I}$ to its
actual value $\indlim[i \in I] X_i$ in $\AA$. Since $\YY$ is fully faithful, the counit
$\colim\circ\YY\cong\Id_{\AA}$ is an isomorphism. In all the situations below $\AA$ has
filtered colimits, so $\colim$ is defined. We write $\indlim$ for colimits computed in $\AA$ and $\colim$ both for this functor and for colimits computed inside $\Ind(\AA)$.

Every ind-object $M = \{m_j\}_{j \in J}$ is the filtered colimit $M = \colim[j \in J] \YY m_j$ of its constituents in $\Ind(\AA)$; we write $\mathrm{in}_j \colon \YY m_j \to M$ for the structure morphisms and $\mathrm{can}_j \colon m_j \to \colim M$ for the canonical morphisms into the colimit in $\AA$. The objects $\YY t$, $t \in \AA$, are compact (i.e.\ finitely presented) in $\Ind(\AA)$, and they generate it: a morphism $f$ of $\Ind(\AA)$ is zero (resp.\ a monomorphism) as soon as $\Hom(\YY t, f)$ is zero (resp.\ injective) for every $t \in \AA$.
\end{remark}

\begin{remark}
We use the ind-completion tacitly at the level of functors. Any additive functor $F\colon\AA\to\BB$ extends, uniquely up to natural isomorphism, to an
additive functor
\[
    IF\colon \Ind(\AA)\to\Ind(\BB),\qquad
    IF\bigl(\{X_i\}_{i\in I}\bigr)=\{F(X_i)\}_{i\in I},
\]
characterised by the fact that it preserves filtered colimits and restricts to $F$ along the
canonical embeddings $\YY_\AA: \AA\hookrightarrow\Ind(\AA)$ and $\YY_\BB: \BB\hookrightarrow\Ind(\BB)$. The
assignment $F\mapsto IF$ is functorial: $I\Id_{\AA}=\Id_{\Ind(\AA)}$ and
$I(G\circ F)\cong IG\circ IF$. We shall use the following without further comment.
\begin{enumerate}
    \item If $(F\dashv G)\colon\AA\rightleftarrows\BB$ is an adjoint
    pair, then so is $(IF\dashv IG)\colon\Ind(\AA)\rightleftarrows\Ind(\BB)$.
    \item If $\AA$ and $\BB$ are abelian and if $F$ is left exact (resp.\ right exact, resp.\ exact), then so is $IF$.
    Indeed, finite limits commute with filtered colimits in $\Ind(\BB)$, and every object of
    $\Ind(\AA)$ is a filtered colimit of objects of $\AA$.
\end{enumerate}
\end{remark}

The theory of compactly assembled categories rests on a relative version of compactness, in which the role of a compact object is played by a morphism.

\begin{definition}[{\cite[Definition 21.1.2.1]{Lurie_SAG}, \cite[Definition 1.38]{efimov_Kth_large_cat}}] \label{def_compact_morphism}
Let $\CC$ be a category with filtered colimits. A morphism $f \colon x \to y$ in $\CC$ is called \emph{compact} if for every filtered system $\{z_i\}_{i \in I}$ in $\CC$ with colimit $z$ and every morphism $g \colon y \to z$ the composite $g \circ f$ factors through some $z_i$, i.e.\ lies in the image of the canonical map $\varinjlim_i \Hom_{\CC}(x, z_i) \to \Hom_{\CC}(x, z)$.

An object $x \in \CC$ is called \emph{compactly exhaustible} if it is the colimit of a sequence $x_1 \to x_2 \to \dots$ of compact morphisms.
\end{definition}

\begin{remark} \label{rem_compact_morphism_basic}
An object $x$ is compact if and only if $\Id_x$ is a compact morphism; more generally, every morphism that factors through a compact object is compact, so compact objects are compactly exhaustible. Compact morphisms form a two-sided ideal: if $f$ is compact, then so are $f \circ a$ and $b \circ f$ for arbitrary $a$, $b$. In an $\omega$-accessible category a morphism is compact if and only if it factors through a compact object (Proposition~\ref{prop_cminj_vs_fpinj} below), so that nothing new happens there; the notion becomes essential when compact objects are scarce, for instance in the category of sheaves on $\mathbb{R}$, which has no non-zero compact objects at all \cite[\S 8]{efimov_Th_of_heart}.
\end{remark}

\begin{definition} \label{def_compass_cohass}
An accessible category with filtered colimits $\CC$ is called {\it compactly assembled} if the functor $\colim\colon \Ind(\CC) \to \CC$ admits a left adjoint $\tYY = \tYY_{\CC} \colon \CC \to \Ind(\CC)$.

A compactly assembled {\bf abelian} category $\AA$ is called {\it coherently assembled} if $\tYY$ is left exact.
\end{definition}

Compact assembly can be phrased entirely in terms of compact morphisms, without mentioning $\Ind(\CC)$; the same circle of ideas provides a cardinal bound for compactly assembled categories.

\begin{proposition}[{\cite[\S 21.1.2]{Lurie_SAG}, \cite[Theorem 2.39]{Ramzi_dualizable}, \cite[Proposition 1.24]{efimov_Kth_large_cat}}] \label{prop_compass_via_exhaustible} \label{prop_compass_cardbound}
Let $\CC$ be an accessible category with filtered colimits. The following are equivalent:
\begin{enumerate}
    \item[(i)] $\CC$ is compactly assembled;
    \item[(ii)] every object of $\CC$ is the colimit of a filtered system in which all transition morphisms are compact;
    \item[(iii)] $\CC$ is generated under filtered colimits by compactly exhaustible objects, i.e.\ every object of $\CC$ is a filtered colimit of compactly exhaustible objects.
\end{enumerate}
Moreover, a compactly assembled category (not necessarily cocomplete) is $\omega_1$-accessible, and its $\omega_1$-compact objects are exactly the compactly exhaustible objects.

If $\CC$ is presentable and stable, conditions (i)--(iii) are also equivalent to the dualizability of $\CC$ in $\mathrm{Pr}^L_{\mathrm{st}}$ and to $\CC$ being a retract, in $\mathrm{Pr}^L$, of a compactly generated category \cite[Proposition D.7.3.1]{Lurie_SAG}, \cite[Proposition 1.11]{efimov_Kth_large_cat}.
\end{proposition}

The implication (i)$\Rightarrow$(ii) is made explicit in Lemma~\ref{lem_tYY_compact_transitions} below: the ind-object $\tYY x$ itself can be presented by a filtered system with compact transition morphisms. The passage (ii)$\Rightarrow$(iii) is the ``ladder'' argument of \cite[proof of Proposition 1.24]{efimov_Kth_large_cat}: if $x = \indlim[j \in J] x_j$ along compact morphisms, then $x$ is the filtered colimit, over the order-preserving maps $f \colon \mathbb{N} \to J$, of the compactly exhaustible objects $\indlim[n] x_{f(n)}$; the same argument gives the $\omega_1$-accessibility statement, and the description of the $\omega_1$-compact objects is \cite[Proposition 1.24]{efimov_Kth_large_cat}. The implication (iii)$\Rightarrow$(i) is the substantial one; see \cite[Theorem 2.39]{Ramzi_dualizable}. Thus a compactly assembled category is one in which every object can be ``exhausted'' by compact morphisms, exactly as an object of a compactly generated category is exhausted by compact subobjects, and the functor $\tYY$ records these exhaustions; condition (iii) exhibits compact assembly as a strengthening of $\omega_1$-accessibility in which the generating $\omega_1$-compact objects are required to be of a specific form.

The simplest examples are the compactly generated categories.

\begin{proposition} \label{prop_omega_acc_compass}
If $\AA$ is an $\omega$-accessible category with filtered colimits, then it is automatically compactly assembled.
\end{proposition}

\begin{proof}
In the $\omega$-accessible case we can define $\tYY$ explicitly as follows: set $\tYY(X) = \{X_i\}_{i \in I}$, where $X_i \in \CC^{f.p.}$ and $\indlim X_i \cong X$ in $\CC$. Such a filtered system exists because $\AA$ is $\omega$-accessible.

We claim that $\tYY$ is well defined. Let $\{X_i\}_{i \in I}$ and $\{Y_j\}_{j \in J}$ be two filtered systems of finitely presented objects such that

\begin{equation}
\colim[i \in I] X_i \cong \colim[j \in J] Y_j \cong X
\end{equation}
Then $\{X_i\}_{i \in I}$ is isomorphic to $\{Y_j\}_{j \in J}$ in $\Ind(\AA)$.

Indeed,
\begin{equation}
\Hom_{\Ind(\mathcal{A})}\bigl(\{X_i\}_{i \in I}, \{Y_j\}_{j \in J}\bigr) = 
\lim_{i \in I} \colim[j \in J] \Hom_{\mathcal{A}}(X_i, Y_j) \cong 
\Hom_{\mathcal{A}}(\colim[i \in I] X_i, \colim[j \in J]Y_j)
\end{equation}
where the last isomorphism follows from the fact that each $X_i$ is finitely presented.
\end{proof}

So for compactly assembled $\AA$ we have
\begin{equation}
\begin{tikzcd}
	\AA &&& {\Ind(\AA)}
	\arrow[""{name=0, anchor=center, inner sep=0}, "\YY"{description}, shift right=5, from=1-1, to=1-4]
	\arrow[""{name=1, anchor=center, inner sep=0}, "\tYY"{description}, shift left=5, from=1-1, to=1-4]
	\arrow[""{name=2, anchor=center, inner sep=0}, "\colim"{description}, from=1-4, to=1-1]
	\arrow["\dashv"{anchor=center, rotate=-90}, shift left=5, draw=none, from=1, to=2]
	\arrow["\dashv"{anchor=center, rotate=-90}, shift left=5, draw=none, from=2, to=0]
\end{tikzcd}.
\end{equation}

\begin{remark} \label{rem_tYY_basic}
We record the basic formal properties of $\tYY$ for a compactly assembled category $\AA$ with filtered colimits.
\begin{enumerate}
    \item The unit $u \colon \Id_{\AA} \to \colim \circ\, \tYY$ is an isomorphism. Indeed, for $x, a \in \AA$ the composite bijection
    \[
    \Hom_{\AA}(\colim \tYY x, a) \cong \Hom_{\Ind(\AA)}(\tYY x, \YY a) \cong \Hom_{\AA}(x, \colim \YY a) \cong \Hom_{\AA}(x, a)
    \]
    is induced by $u_x$. We use $u$ and the counit $\colim \YY \cong \Id$ to identify $\colim \tYY x$ and $\colim \YY x$ with $x$. With these identifications the adjunction $\tYY \dashv \colim$ reads
    \begin{equation}
    \Hom_{\Ind(\AA)}(\tYY x, M) \xrightarrow{\ \sim\ } \Hom_{\AA}(x, \colim M), \qquad f \mapsto \colim f .
    \end{equation}
    In particular a morphism $f \colon \tYY x \to M$ is determined by $\colim f$.

    \item We denote by $\rho \colon \tYY \Rightarrow \YY$ the natural transformation $\rho_z \colon \tYY z \xrightarrow{\eta} \YY \colim \tYY z \cong \YY z$, where $\eta$ is the unit of $\colim \dashv \YY$; thus $\colim \rho_z = \Id_z$. Naturality of $\eta$ gives $\eta_M \circ \mathrm{in}_j = \YY(\mathrm{can}_j)$ for every ind-object $M$ (notation of Remark~\ref{rem_ind_basic}). For $\tYY z = \{z_j\}_{j \in J}$ we put $\sigma_j := \mathrm{can}_j \colon z_j \to z$; then $\rho_z \circ \mathrm{in}_j = \YY \sigma_j$ and $z = \indlim[j] z_j$.
\end{enumerate}
\end{remark}

Once $\tYY$ is available, compactness of a morphism can be read off from the ind-completion; this is the form in which the notion will be used throughout the paper.

\begin{proposition}[{\cite[Proposition 1.40]{efimov_Kth_large_cat}}] \label{prop_compact_via_tYY}
Let $\CC$ be a compactly assembled category. A morphism $f \colon x \to y$ in $\CC$ is compact in the sense of Definition~\ref{def_compact_morphism} if and only if $\YY f \colon \YY x \to \YY y$ factors through $\rho_y \colon \tYY y \to \YY y$.
\end{proposition}

\begin{proof}
Write $\tYY y = \{y_j\}_{j \in J}$, so $y = \indlim[j] y_j$ with canonical morphisms $\sigma_j \colon y_j \to y$ (Remark~\ref{rem_tYY_basic}). If $f$ is compact, apply Definition~\ref{def_compact_morphism} to $g = \Id_y$: $f = \sigma_{j_0} \circ a$ for some $j_0 \in J$ and $a \colon x \to y_{j_0}$, and then $\rho_y \circ \mathrm{in}_{j_0} \circ \YY a = \YY \sigma_{j_0} \circ \YY a = \YY f$. Conversely, let $\hat{f} \colon \YY x \to \tYY y$ satisfy $\rho_y \hat{f} = \YY f$, and let $g \colon y \to z = \indlim[i] z_i$ be given. By the $\tYY \dashv \colim$ adjunction, $g$ corresponds to a morphism $\hat{g} \colon \tYY y \to \{z_i\}_{i \in I}$ with $\colim \hat{g} = g$. Since $\YY x$ is compact in $\Ind(\CC)$, the composite $\hat{g} \hat{f} \colon \YY x \to \{z_i\} = \colim[i] \YY z_i$ factors as $\mathrm{in}_{i_0} \circ \YY a$ for some $i_0 \in I$ and $a \colon x \to z_{i_0}$. Applying $\colim$ and using $\colim \hat{f} = f$ we get $g \circ f = \mathrm{can}_{i_0} \circ a$, as required.
\end{proof}

\begin{remark} \label{rem_compact_lifts}
In the situation of Proposition~\ref{prop_compact_via_tYY} we call any $\hat{\phi} \colon \YY y \to \tYY z$ with $\rho_z \hat{\phi} = \YY \phi$ a \emph{lift} of the compact morphism $\phi \colon y \to z$; then $\colim \hat{\phi} = \phi$. A lift is in general not unique. For $\tYY z = \{z_j\}_{j \in J}$ the morphisms $\sigma_j \colon z_j \to z$ are compact, with lifts $\mathrm{in}_j$. The ideal property of Remark~\ref{rem_compact_morphism_basic} is visible on lifts: if $\hat{\phi}$ is a lift of $\phi$, then $\hat{\phi} \circ \YY a$ and $\tYY b \circ \hat{\phi}$ are lifts of $\phi \circ a$ and $b \circ \phi$, by naturality of $\rho$. Finally, for a compact morphism $\phi$ with lift $\hat{\phi}$ one has $\tYY \phi = \hat{\phi} \circ \rho_y$: both sides are morphisms $\tYY y \to \tYY z$ with $\colim$ equal to $\phi$.
\end{remark}

The following lemma makes the structure of $\tYY x$ explicit from the point of view of Proposition~\ref{prop_compass_via_exhaustible}(ii); the argument is the one of \cite[proof of Proposition 1.24]{efimov_Kth_large_cat}.

\begin{lemma} \label{lem_tYY_compact_transitions}
Let $\CC$ be compactly assembled and $x \in \CC$. Then $\tYY x$ is isomorphic in $\Ind(\CC)$ to a filtered system $\{x_j\}_{j \in J}$ all of whose transition morphisms $x_j \to x_{j'}$, $j \leq j'$, are compact. In particular $x = \indlim[j] x_j$ along compact morphisms.
\end{lemma}

\begin{proof}
Write $\tYY x = \{x_i\}_{i \in I}$; by the convention of Remark~\ref{rem_filtered_convention} we may assume that $I$ is a directed poset. Denote the transition morphisms by $f_{ik} \colon x_i \to x_k$ and the canonical morphisms by $\sigma_i \colon x_i \to x$.

\emph{Step 1: for every $i \in I$ there is $k \geq i$ such that $f_{ik}$ is compact.} Since $\tYY$ preserves filtered colimits, $\tYY x \cong \colim[k] \tYY x_k$ with structure morphisms $\tYY \sigma_k$. The object $\YY x_i$ is compact in $\Ind(\CC)$, so $\mathrm{in}_i = \tYY(\sigma_k) \circ \lambda$ for some $k \geq i$ and $\lambda \colon \YY x_i \to \tYY x_k$. Let $c \colon x_i \to x_k$ be the morphism with $\YY c = \rho_{x_k} \lambda$; it is compact, with lift $\lambda$. By Remark~\ref{rem_compact_lifts}, $\tYY \sigma_k = \mathrm{in}_k \circ \rho_{x_k}$, hence $\mathrm{in}_i = \mathrm{in}_k \circ \YY c$. On the other hand $\mathrm{in}_i = \mathrm{in}_k \circ \YY f_{ik}$. Thus $c$ and $f_{ik}$ have the same image in $\Hom_{\Ind(\CC)}(\YY x_i, \tYY x) = \varinjlim_{k'} \Hom(x_i, x_{k'})$, i.e.\ $f_{kk'} \circ c = f_{kk'} \circ f_{ik} = f_{ik'}$ for some $k' \geq k$. Being a composite of the compact morphism $c$ with $f_{kk'}$, the transition morphism $f_{ik'}$ is compact.

\emph{Step 2: reindexing.} Let $J$ be the set of pairs $(i, k) \in I \times I$ with $i \leq k$ and $f_{ik}$ compact, ordered by $(i, k) \leq (i', k')$ if either $(i, k) = (i', k')$ or $k \leq i'$. This is a partial order, and it is directed (so that $\{x_i\}_{(i,k) \in J}$ is a filtered system): given $(i, k)$ and $(i', k')$, choose $m \geq k, k'$ and then, by Step 1, $n \geq m$ with $f_{mn}$ compact; then $(m, n)$ dominates both. The map $J \to I$, $(i, k) \mapsto i$, is order preserving and cofinal by Step 1, so the restricted system $\{x_i\}_{(i, k) \in J}$ is isomorphic to $\tYY x$ in $\Ind(\CC)$ and has colimit $x$. For $(i, k) < (i', k')$ its transition morphism is $f_{ii'} = f_{ki'} \circ f_{ik}$, which is compact since $f_{ik}$ is.
\end{proof}

\begin{lemma} \label{lem_commutativity_of_ind}
Consider a functor between $\omega$-accessible categories together with the induced functor on the ind-completion:
\begin{equation}
\begin{tikzcd}
	\AA &&& \BB \\
	\\
	\\
	{\Ind(\AA)} &&& {\Ind(\BB)}
	\arrow["F"{description}, from=1-1, to=1-4]
	\arrow["{\tYY_\AA}"{description}, from=1-1, to=4-1]
	\arrow["{\tYY_\BB}"{description}, from=1-4, to=4-4]
	\arrow["IF"{description}, from=4-1, to=4-4]
\end{tikzcd}
\end{equation}
If $F$ preserves compact objects and filtered colimits, then the diagram is commutative, i.e.,
\[
\tYY_\BB(F(a)) \cong IF \tYY_\AA(a) \,\,\, \text{for} \,\, a \in \AA. 
\]
\end{lemma}

\begin{proof}
\begin{equation} \label{diag_ind_commutativity}
\begin{tikzcd} [sep = 9]
	\AA &&&&& \BB \\
	& a &&&&& {F(a)} \\
	\\
	\\
	\\
	{\Ind(\AA)} &&&&& {\Ind(\BB)} & {\{b_i\}} \\
	& {\{a_i\}} &&&& {\{F(a_i)\}}
	\arrow["F"{description}, from=1-1, to=1-6]
	\arrow["{{\tYY_\AA}}"{description}, from=1-1, to=6-1]
	\arrow["{{\tYY_\BB}}"{description}, from=1-6, to=6-6]
	\arrow[maps to, from=2-2, to=2-7]
	\arrow[maps to, from=2-2, to=7-2]
	\arrow[maps to, from=2-7, to=6-7]
	\arrow["IF"{description}, from=6-1, to=6-6]
	\arrow[maps to, from=7-2, to=7-6]
	\arrow["\cong"{marking, allow upside down}, draw=none, from=7-6, to=6-7]
\end{tikzcd}
\end{equation}

Set $\{a_i\} := \tYY_\AA(a)$ and $\{b_i\} := \tYY_\BB(F(a))$. As shown in Diagram~(\ref{diag_ind_commutativity}), $IF(\{a_i\}) = \{F(a_i)\}$, so it suffices to show that $\{F(a_i)\} \cong \{b_i\}$.

Note that two filtered systems of compact objects are isomorphic if and only if their colimits are isomorphic; hence it suffices to show that
\[
\indlim F(a_i) \cong \varinjlim b_i
\]
since $\{F(a_i)\}$ is a system of compact objects ($F$ preserves compact objects) and $\{a_i\}$ is a system of compact objects by definition.
\end{proof}

The previous lemma will be used frequently to analyse a functor $\tYY_\AA$ by means of $\tYY_\BB$, or vice versa, for various $\AA$ and $\BB$.

The requirements on the functor $F$ in the previous lemma can be weakened if commutativity is required only on a subcategory of $\AA$, for example, only on the image of the fully faithful left adjoint of $F$. We make this precise as follows.

\begin{lemma} \label{lem_commutativity_of_ind(spec)}
Consider a pair of adjoint functors $(H \dashv F)$ between $\omega$-accessible categories such that $H$ is fully faithful and $F$ preserves filtered colimits:
\begin{equation}
\begin{tikzcd}
	\AA &&& \BB \\
	\\
	\\
	{\Ind(\AA)} &&& {\Ind(\BB)}
	\arrow[""{name=0, anchor=center, inner sep=0}, "F"{description}, shift right=2, from=1-1, to=1-4]
	\arrow["{{\tYY_\AA}}"{description}, from=1-1, to=4-1]
	\arrow[""{name=1, anchor=center, inner sep=0}, "H"{description}, shift right=2, from=1-4, to=1-1]
	\arrow["{{\tYY_\BB}}"{description}, from=1-4, to=4-4]
	\arrow[""{name=2, anchor=center, inner sep=0}, "IF"{description}, shift right=2, from=4-1, to=4-4]
	\arrow[""{name=3, anchor=center, inner sep=0}, "IH"{description}, shift right=2, from=4-4, to=4-1]
	\arrow["\dashv"{anchor=center, rotate=-90}, shift left=5, draw=none, from=1, to=0]
	\arrow["\dashv"{anchor=center, rotate=-90}, shift left=5, draw=none, from=3, to=2]
\end{tikzcd}
\end{equation}

Then
\[
\tYY_\BB(FH(b)) \cong IF\, \tYY_\AA H(b),
\]
without the additional assumption that $F$ preserves compact objects.
\end{lemma}

\begin{proof}
\begin{equation} \label{diag_big_ind_commutativity}
\begin{tikzcd} [sep = 7]
	\BB &&&&& \AA &&&&& \BB \\
	& b &&&&& {H(b)} &&&&& b \\
	\\
	\\
	\\
	{\Ind(\BB)} &&&&& {\Ind(\AA)} &&&&& {\Ind(\BB)} \\
	& {\{b_i\}} &&&&& {\{H(b_i)\}} &&&&& {\{b_i\}}
	\arrow["H"{description}, from=1-1, to=1-6]
	\arrow["{{{\tYY_\BB}}}"{description}, from=1-1, to=6-1]
	\arrow["F"{description}, from=1-6, to=1-11]
	\arrow["{\tYY_\AA}"{description}, from=1-6, to=6-6]
	\arrow["{{{\tYY_\BB}}}"{description}, from=1-11, to=6-11]
	\arrow[maps to, from=2-2, to=2-7]
	\arrow[maps to, from=2-2, to=7-2]
	\arrow[maps to, from=2-7, to=2-12]
	\arrow[maps to, from=2-7, to=7-7]
	\arrow[maps to, from=2-12, to=7-12]
	\arrow["IH"{description}, from=6-1, to=6-6]
	\arrow["IF"{description}, from=6-6, to=6-11]
	\arrow[maps to, from=7-2, to=7-7]
	\arrow[maps to, from=7-7, to=7-12]
\end{tikzcd}
\end{equation}

Note that Lemma~(\ref{lem_commutativity_of_ind}) is applicable to the functor $H$. Indeed, $H$ preserves filtered colimits as a left adjoint, and it preserves compact objects since its right adjoint preserves filtered colimits.

Hence the left square in Diagram~(\ref{diag_big_ind_commutativity}) commutes, where $b$ is an object of $\BB$ and $\{b_i\}$ is a filtered system of compact objects in $\BB$ such that $\indlim b_i \cong b$.

It follows that for objects of $\AA$ of the form $H(b)$ with $b \in \BB$, the filtered system of compact objects of $\AA$ approximating $H(b)$ can be chosen to be $\{H(b_i)\}$.

The desired statement, that is, the commutativity of the right square in~(\ref{diag_big_ind_commutativity}), now follows immediately from fully faithfulness of $H$. 
\end{proof}

\subsection{Local coherency conditions} \label{ssec:loccoh_def}

\begin{definition}
Let $\AA$ be an accessible abelian category, and let $\kappa$ be a regular cardinal. The category $\AA$ is called {\it locally $\kappa$-coherent} if $\AA$ is locally $\kappa$-presentable and $\AA^\kappa$ is abelian.

If $\kappa = \omega$, we simply say {\it locally coherent} instead of {\it locally $\omega$-coherent}.
\end{definition}

There is a direct analogy with coherent rings in algebra.

\begin{proposition} \label{prop_loccoh=cohring}
Let $\AA = R - \Mod$ (resp.\ $\Mod - R$). Then $\AA$ is locally coherent if and only if $R$ is left (resp.\ right) coherent.
\end{proposition}

\begin{proposition}\label{prop_enough_to_check_on_epi}
To verify the local coherence condition it suffices to show that $\ker(\phi) \in \AA^{f.p.}$ for every $X_1, X_2 \in \AA^{f.p.}$ and every $\phi \in \Hom_\AA(X_1, X_2)$, since $\coker(\phi) \in \AA^{f.p.}$ is automatic by Lemma~(\ref{lem_ab_fp_fg_via_ker_coker}).

Moreover, it is enough to check that $\ker(\phi) \in \AA^{f.p.}$ when $\phi$ is an epimorphism.
\end{proposition} 

\begin{proof}
Indeed:
\begin{equation}
\begin{tikzcd}[sep=scriptsize]
	{\ker(\phi)} && {X_1} && {X_2} && {\coker(\phi)} \\
	&&& {im(\phi)}
	\arrow[from=1-1, to=1-3]
	\arrow["\phi", from=1-3, to=1-5]
	\arrow[from=1-3, to=2-4]
	\arrow[from=1-5, to=1-7]
	\arrow[from=2-4, to=1-5]
\end{tikzcd}
\end{equation}
Finite presentation of $\ker(\phi)$ follows from finite presentation of $\im(\phi)$ since $X_1 \to \im(\phi)$ is an epimorphism; finite presentation of $\im(\phi)$ follows from finite presentation of $\coker(\phi)$ since $X_2 \to \coker(\phi)$ is an epimorphism.
\end{proof}

It turns out that if a category $\AA$ is coherently assembled in the sense of Definition~(\ref{def_compass_cohass}), then the cardinal $\kappa$ for which $\AA$ is locally $\kappa$-coherent is controllable, as in Proposition~(\ref{prop_compass_cardbound}). Namely:

\begin{proposition}\cite[Corollary 1.45.]{efimov_Th_of_heart}
Let $\AA$ be a coherently assembled abelian category.
Then $\AA$ is locally $\omega_1$-coherent.
\end{proposition}

Note that for an $\omega$-accessible abelian category $\AA$ satisfying mild assumptions, the coherently assembled condition implies not only local $\omega_1$-coherence but also local coherence. Namely:

\begin{proposition}\cite{efimov_Kth_large_cat} \label{prop_loccoh=cohass}
Let $\AA$ be locally finitely presentable AB5-category. The following are equivalent:
\begin{enumerate}
    \item $\AA$ is locally coherent
    \item $\AA$ is coherently assembled
\end{enumerate}
\end{proposition}

\subsubsection{Characterisation of the coherency of a ring}
The main goal of this subsection is to prove the following equivalences characterising coherent rings.

\begin{theorem}\cite{Stenstrom-1970, CheathamStone, Chase1960DirectPO, Matlis_1982, Krause2022}   \label{th_MEGA}
Let $R$ be a ring. 
The following are equivalent: 
\begin{enumerate}
    \item \label{item_mega_1} The ring $R$ is left coherent.

    \item The category $R - \Mod$ is locally coherent.

    \item The category $R - \Mod$ is coherently assembled.

    \item The category $R-proj$ has weak kernels.

    \item The wide subcategory $\CC_0 \subset R-\Mod$ generated by $R$ coincide with $R-\Mod^{f.p.}$

    \item For any $n \geq 1$ the category $\mathcal{FP}_n(R)$ is abelian.

    \item The map $\xi_1 : \varinjlim_i \Ext^1(F, M_i) \to \Ext^1(F, \varinjlim_i M_i)$ is an isomorphism for every finitely presented left  module $F$ and direct system $(M_i)_i$.
    \item The map $\xi_n : \varinjlim_i \Ext^n(F, M_i) \to \Ext^n(F, \varinjlim_i M_i)$ are isomorphisms for all $n \ge 0$, for every finitely presented left module $F$ and direct system $(M_i)_i$.
    \item Every direct limit of FP-injective left modules is FP-injective.
    \item Every direct limit of left modules of FP-injective dimension $\le n$ has FP-injective dimension $\le n$.

    \item A left module $M$ is FP-injective exactly when its character module $M^+$ is a flat right module.
    \item A left module $M$ is FP-injective exactly when $M^{++}$ is an injective left module.
    \item A right module $M$ is flat exactly when $M^{++}$ is a flat right module.

    \item For every bimodule $N$, the induced morphisms $\Tor_R^n(\Hom_R(N, E), M) \to \Hom_R(\Ext_R^n(M, N), E)$ are isomorphisms in all degrees $n \ge 0$ whenever $E$ is an injective left module and $M$ a finitely presented left module.
    \item {\bf (for $R$ - commutative ring)} For all injective $R$-modules $N$ and $E$.
    $\Hom_R(N, E)$ is a flat $R$-module.

    \item Any direct product of flat right modules is a flat right module.
    \item Any direct product of copies of $R$ is a flat right module.
\end{enumerate}
\end{theorem}

Before the proof we recall some notation.

\begin{definition} \label{def_MEGA}
A module $M$ is {\it FP-injective} \cite{Stenstrom-1970} if $\Ext^{1}(F, M) = 0$
for every finitely presented module $F$.

For any module $M$, let {\it FP-injective dimension} \cite{Stenstrom-1970} $\operatorname{FP - inj.\dim} M$ denote the smallest integer $n \ge 0$ such that $\Ext^{n+1}(F, M) = 0$ for every finitely presented module $F$.

Given a left (resp.\ right) $R$-module $M$, its {\it character module} is the right (resp.\ left) $R$-module $M^+ = \Hom_\mathbb{Z}(M, \mathbb{Q}/\mathbb{Z})$.

For rings $\Gamma$, $\Lambda$  define the homomorphism
\begin{equation}
\sigma_0: \Hom_{\Gamma}(B, C) \otimes_{\Lambda} A \to \Hom_{\Gamma}(\Hom_{\Lambda}(A, B), C)
\end{equation}
by setting
\begin{equation}
\sigma_0(f \otimes a)(g) = f(g(a)) \qquad f \in \Hom_{\Gamma}(B, C), \; g \in \Hom_{\Lambda}(A, B).
\end{equation}

Thus passing to the projective resolution of $A$ one obtains a homomorphism
\begin{equation} \label{torhom_to_homext}
\sigma_*: \Tor^{\Lambda}(\Hom_{\Gamma}(B, C), A) \to \Hom_{\Gamma}(\Ext_{\Lambda}(A, B), C)
\end{equation}
which reduces to the previous one in degree zero.

A submodule $N \hookrightarrow M$ of a module $M$ is said to be a {\it pure submodule} if the induced map $N \otimes X \to M \otimes X$ stays a monomorphism for every $R$-module $X$. The snake lemma makes it plain that a pure submodule of a flat module is itself flat.

A morphism $X \to Y$ in $\CC$ is a \textbf{weak kernel} of a morphism $Y \to Z$ if the induced sequence
\[
\Hom_{\CC}(-, X) \to \Hom_{\CC}(-, Y) \to \Hom_{\CC}(-, Z)
\]
is exact, i.e. the universal property with non-unique lifting holds. 

Let $n \geq 0$ be an integer. Recall that an $R$-module $M$ is {\it finitely $n$-presented} if there exists an exact sequence
\[
F_n \to F_{n-1} \to F_{n-2} \to \cdots \to F_0 \to M \to 0
\]
where $F_i$ is a finitely generated and free (or projective) module, for every $0 \leq i \leq n$. Following \cite{FPn_first}, we denote by $\mathcal{FP}_n(R)$ the full subcategory of finitely $n$-presented modules. 
Note that $\mathcal{FP}_0(R)$ is the category of finitely generated modules and $\mathcal{FP}_1(R)$ is the category of finitely presented modules.

A full subcategory $\CC$ of $R-\Mod$ is called {\it wide} if it is abelian and closed under extensions. 
We denote by $\CC_0$ the wide subcategory generated by $R$.
\end{definition}

The following folklore result regarding character modules is due to Lambek \cite{Lambek_1964}.

\begin{proposition} \cite{Lambek_1964} \label{prop_Lambek}
An $R$-module is flat exactly when its character module is injective.
\end{proposition}

In what follows we will also use the observation of Cartan, Eilenberg \cite{cartan1956homological}
with the remark of Matlis \cite{Matlis_1982}.

\begin{proposition} \cite{cartan1956homological, Matlis_1982} \label{prop_CartanEilenberg+Matlis}
If $A$ is a finitely presented $\Lambda$-module, $C$ is an injective $\Gamma$-module 
and $A$ admits a resolution by finitely generated projectives (when $\Lambda$ is a coherent ring this is automatic for any finitely presented $\Lambda$-module)
then $\sigma_*$ from (\ref{torhom_to_homext}) is an isomorphism.
\end{proposition}

Note that it is enough to check purity of a submodule by tensoring with finitely presented modules.

\begin{proposition} \label{prop_pure_submodule_via_fp}
Let $N \xhookrightarrow{\psi} M$ be an injection of a submodule in an $R$-module $M$, such that $N \otimes Y \to M \otimes Y$ is a monomorphism for any finitely presented $R$-module $Y$. Then $N$ is pure submodule.
\end{proposition}

\begin{proof}
Let $N \otimes X \xrightarrow{\psi'} M \otimes X$ be the map obtained from $\psi$ by tensoring with arbitrary $R$-module $X$, and 
let $\{X_i\}_{i \in I}$ be a filtered system of finitely presented subobjects of $X$ such that $\indlim X_i = X$.
Since the tensor product commutes with colimits $\psi'$ is a filtered colimit of maps
$N \otimes X_i \to M \otimes X_i$ which are monomorphisms by the conditions, the claim now follows since filtered colimits are exact in $R-\Mod$.
\end{proof}

We also want to recall the result of \cite{Wurfel1973}.

\begin{proposition}\cite{Wurfel1973} \label{prop_Wurfel}
The canonical morphism $M \xrightarrow{\phi} M^{++}$ given by $\phi(m)(\alpha) = \alpha(m)$ for 
$ \alpha \in \Hom_{\mathbb{Z}}(M, \mathbb{Q}/\mathbb{Z})$ is an injection of a pure submodule. 
\end{proposition}

\begin{proof}
Clearly $M \xhookrightarrow{\phi} M^{++}$ is a submodule since $\mathbb{Q}/\mathbb{Z}$ is an injective cogenerator in $\mathbb{Z}-\Mod$.
Due to (\ref{prop_pure_submodule_via_fp}) it is enough to show that $M \otimes X \xrightarrow{\phi'} M^{++} \otimes X$ is a monomorphism for any finitely presented module $X$.
Note that $ M^{++} \otimes X \cong \Hom(\Hom(X, M^+), \mathbb{Q}/\mathbb{Z}) \cong (M \otimes X)^{++}$
where the first isomorphism is from (\ref{prop_CartanEilenberg+Matlis}). 
Now the inclusion  $M \otimes X \xrightarrow{\phi'} (M \otimes X)^{++}$ implies the claim.
\end{proof}

\begin{lemma} \cite{Matlis_1982} \label{lem_Matlis}
Suppose that each $A_\alpha$ is a flat $R$-module, and that $B_\alpha$ is a submodule of $A_\alpha$ such that $A_\alpha/B_\alpha$ is flat. Then
\[
\prod_\alpha A_\alpha \text{ is flat } \iff \prod_\alpha B_\alpha \text{ and } \prod_\alpha (A_\alpha/B_\alpha) \text{ are flat}.
\]
\end{lemma}

The following lemmata of Stenstr{\"o}m~\cite{stenstrom1975rings} characterise finite generation and finite presentation of modules through the way direct products interact with tensor products.
Concretely, for a right $A$-module $M$ and a family $(L_i)_I$ of left $A$-modules there is a homomorphism
\[
\varphi: M \otimes_A \bigl( \prod_I L_i \bigr) \to \prod_I \bigl( M \otimes_A L_i \bigr),
\]
given by $\varphi(x \otimes (z_i)) = (x \otimes z_i)$. In particular, there is a homomorphism $\varphi: M \otimes A^I \to M^I$ given by $\varphi(x \otimes (a_i)) = (x a_i)$.

\begin{lemma} \cite{stenstrom1975rings} \label{lem_stenstrom_fg}
The following properties of the module $M$ are equivalent:
\begin{enumerate}
\item[(a)] $M$ is finitely generated.
\item[(b)] $\varphi: M \otimes_A (\prod_I L_i) \to \prod_I (M \otimes_A L_i)$ is an epimorphism for every family $(L_i)_I$ of left $A$-modules.
\item[(c)] $\varphi: M \otimes_A A^I \to M^I$ is an epimorphism for every set $I$.
\end{enumerate}
\end{lemma}

\begin{lemma} \cite{stenstrom1975rings} \label{lem_stenstrom_fp}
    The following properties of the module $M$ are equivalent:
\begin{enumerate}
\item[(a)] $M$ is finitely presented.
\item[(b)] $\varphi: M \otimes_A (\prod_I L_i) \to \prod_I (M \otimes_A L_i)$ is an isomorphism for every family $(L_i)_I$ of left $A$-modules.
\item[(c)] $\varphi: M \otimes_A A^I \to M^I$ is an isomorphism for every set $I$.
\end{enumerate}
\end{lemma}

We will need a general categorical lemma due to Krause \cite{Krause2022} about category
$\mod \CC$.

\begin{lemma} \cite{Krause2022} \label{lem_krause_wker}
Let $\CC$ be an essentially small additive category.
The category $\mod \CC$ is additive and all morphisms in 
$\mod \CC$ have cokernels. The category is abelian if and only if all 
morphisms in $\CC$ have weak kernels.
\end{lemma}

The following lemma characterizing subcategory $C_0 \subset R-\Mod$ (in the notation of Definition (\ref{def_MEGA})) is due to Hovey \cite{Hovey2001}.

\begin{lemma}\cite[Lemma 1.6.]{Hovey2001}
A ring $R$ is coherent if and only if $C_0 = R-\Mod^{f.p.}$.
\end{lemma}

\begin{proof}[Proof of Theorem \ref{th_MEGA}]
The proof will go as follows
\begin{equation}
\begin{tikzcd}
	& 14 &&& 6 & 5 &&&&& \\
	&&&& 3 & 2 & 4 \\
	11 && 15 &&&&&&& 8 \\
	12 && 17 &&& 1 &&& 10 && 7 \\
	13 && 16 &&&&&&& 9
	\arrow[Rightarrow, from=1-2, to=3-1]
	\arrow[Rightarrow, from=1-2, to=3-3]
	\arrow[Rightarrow, from=1-5, to=1-6]
	\arrow[Rightarrow, from=2-6, to=1-5]
	\arrow[Rightarrow, 2tail reversed, from=2-6, to=1-6]
	\arrow[Rightarrow, 2tail reversed, from=2-6, to=2-5]
	\arrow[Rightarrow, 2tail reversed, from=2-6, to=2-7]
	\arrow[Rightarrow, 2tail reversed, from=3-1, to=4-1]
	\arrow[Rightarrow, from=3-3, to=4-3]
	\arrow[Rightarrow, from=3-10, to=4-9]
	\arrow[Rightarrow, from=3-10, to=4-11]
	\arrow[Rightarrow, from=4-1, to=5-1]
	\arrow[Rightarrow, from=4-3, to=4-6]
	\arrow[Rightarrow, from=4-6, to=1-2]
	\arrow[Rightarrow, 2tail reversed, from=4-6, to=2-6]
	\arrow[Rightarrow, from=4-6, to=3-10]
	\arrow[Rightarrow, from=4-6, to=5-3]
	\arrow[Rightarrow, from=4-9, to=5-10]
	\arrow[Rightarrow, from=4-11, to=5-10]
	\arrow[Rightarrow, from=5-1, to=5-3]
	\arrow[Rightarrow, from=5-3, to=4-3]
	\arrow[Rightarrow, from=5-10, to=4-6]
\end{tikzcd}.
\end{equation}

\textbf{1 $\iff$ 2} is the content of Proposition (\ref{prop_loccoh=cohring}).

\textbf{2 $\iff$ 3} is the content of Proposition (\ref{prop_loccoh=cohass}).

\textbf{2 $\Rightarrow$ 4}.
A finitely generated projective is in particular finitely presented, so \textbf{2} implies that the kernel of any map between finitely generated projectives is finitely presented.
One then only has to recall that every finitely presented module is a quotient of a finitely generated projective.

\textbf{4 $\Rightarrow$ 2}.
Every finitely presented module arises as the cokernel of a map between finitely generated projectives, which gives
\[
R - \mod \xrightarrow{\sim} \mod(R - \proj).
\]
Hence $R - \mod$ is abelian by Lemma (\ref{lem_krause_wker})

\textbf{2 $\iff$ 5}. This implication also appears in \cite[Lemma 1.6. ]{Hovey2001}.
$\CC_0$ is abelian, so \textbf{5 $\Rightarrow$ 2} is obvious.
Conversely one could observe that $\CC_0$ is obtained by sequentially taking kernels cokernels and extensions. 
Clearly $R \in R-\Mod^{f.p.}$ and $R-\Mod^{f.p.}$ is closed under extensions, cokernels.
So from \textbf{2} $R-\Mod^{f.p.}$ also closed under kernels implying $\CC_0 = R-\Mod^{f.p.}$

\textbf{6 $\Rightarrow$ 5}. This implication is from \cite[Proposition 5.1.]{k-th_of_n-coherent_rings}.
As in the above argument if $\mathcal{FP}_n(R)$ is abelian then $C_0 \subseteq \mathcal{FP}_n(R) \subseteq \mathcal{FP}_1(R)$, so $C_0 = \mathcal{FP}_1(R)$.

\textbf{2 $\Rightarrow$ 6}. The similar implication appears in \cite[Theorem 2.4.]{FPn_first}.
Since $R-\Mod^{f.p.}$ is abelian any presentation 
\[F_1 \to F_0 \to M \to 0\] of module $M$ could be extended on the left to 
\[F_2 \to F_1 \to F_0 \to M \to 0\] and thus $R-\Mod^{f.p.} = \mathcal{FP}_2(R)$.
By the same argument $\mathcal{FP}_n(R) = \mathcal{FP}_{n+1}(R)$ for any positive $n$ and
thereby $\mathcal{FP}_n(R) = \mathcal{FP}_{\infty}(R)$ is abelian.

\textbf{1 $\Rightarrow$ 8.} 
Left coherence of $R$ guarantees that each finitely presented left module $F$ has a projective resolution whose terms are finitely presented.  Realising $\Ext^n(F,-)$ as the $n$‑th cohomology of the Hom‑complex attached to such a resolution, one may invoke that direct limits are exact, hence commute with cohomology, and that they commute with $\Hom$ out of finitely presented modules; combining the two, $\xi_n$ turns out to be an isomorphism.

\textbf{8 $\Rightarrow$ 7} is obvious.

\textbf{7 $\Rightarrow$ 9.} 
Suppose $(M_i)$ is a direct system of FP‑injective left modules, so that $\Ext^1(F,M_i)=0$ for each finitely presented $F$ and therefore $\varinjlim_i\Ext^1(F,M_i)=0$.  Condition 7 then forces $\Ext^1(F,\varinjlim_i M_i)=0$, i.e.\ $\varinjlim_i M_i$ is FP‑injective.

\textbf{8 $\Rightarrow$ 10.} 
Assume each $M_i$ has FP‑injective dimension $\le n$, equivalently $\Ext^{n+1}(F,M_i)=0$ for all finitely presented $F$.  Taking the direct limit and applying 8 in degree $n+1$ yields $\Ext^{n+1}(F,\varinjlim_i M_i)=0$, whence $\varinjlim_i M_i$ too has FP‑injective dimension $\le n$.

\textbf{10 $\Rightarrow$ 9.} 
Specialise 10 to $n=0$.  Since FP‑injective dimension $\le 0$ is precisely FP‑injectivity, the case $n=0$ of 10 states exactly that a direct limit of FP‑injective modules remains FP‑injective, which is 9.

\textbf{9 $\Rightarrow$ 1.} 
Let $L$ be a finitely generated submodule of a finitely presented module $F$, and fix an arbitrary homomorphism $\alpha : L \to \varinjlim_i M_i$.

Form the direct system $\{E(M_i)\}$ of injective envelopes of the $M_i$.
Injectives are FP‑injective, so by \textbf{9} the colimit $\varinjlim_i E(M_i)$ is FP‑injective, and consequently $\alpha$ prolongs to a map $\beta : F \to \varinjlim_i E(M_i)$.  Finite presentation of $F$ lets $\beta$ factor through some $E(M_j)$.
The resulting composite $L \to F \to E(M_j) \to E(M_j)/M_j$ vanishes in the colimit, so after enlarging $j$ we may take it to be zero outright; then $L \to F \to E(M_j)$ runs through $M_j$ and provides a factorisation of $\alpha$ via $M_j$.  Hence $L$ is finitely presented, and $R$ is coherent.

\textbf{11 $\iff$ 12} is obvious due to (\ref{prop_Lambek}).

\textbf{12 $\Rightarrow$ 13}. 
 Let $M$ be a flat $R$-module. Its character module $M^{+}$ is then injective, so $M^{+++}$ is injective as well and therefore $M^{++}$ is flat.
 For the converse, if $M^{++}$ is flat then $M$, being a pure submodule of $M^{++}$ by (\ref{prop_Wurfel}), is flat too.

\textbf{13 $\Rightarrow$ 16}. Let $(M_i)_{i \in I}$ be a family of flat $R$-modules; we prove that $\prod_{i \in I} M_i$ is flat. The sum $\bigoplus_i M_i$ is flat, hence $(\bigoplus_i M_i)^{++} \cong (\prod_i M_i)^{++}$ is flat by (13). As $\bigoplus_i M_i^{++}$ embeds as a pure submodule of $\prod_i M_i^{++}$, the sequence $(\prod_i M_i^{++})^{+} \to (\bigoplus_i M_i^{++})^{+} \to 0$ splits, so $\prod_i M_i^{++} \cong (\bigoplus_i M_i^{++})^{+}$ is flat. Finally $\prod_i M_i$, being a pure submodule of $\prod_i M_i^{++}$, is flat as well.

\textbf{1 $\Rightarrow$ 14} Is immediate from (\ref{prop_CartanEilenberg+Matlis}).

\textbf{14 $\Rightarrow$ 15}.
For $N, E$ injective $R$-modules, from 14 one gets
$\Tor_R^1(\Hom_R(N, E), M) \cong \Hom_R(\Ext_R^1(M, N), E) \cong 0$. 

\textbf{15 $\Rightarrow$ 17}.
Let $E$ be a universal injective $R$-module and $H = \Hom_R(E, E)$. Since $R \subset H$ and, by assumption, $H$ is flat, we have $H/R$ flat. Let $\mathcal{U}$ be an index set, $H_\alpha = H$, $R_\alpha = R$ and $E_\alpha = E$ for all $\alpha \in \mathcal{U}$. Now
\[
\prod_{\alpha} H_\alpha \cong \Hom_R\Bigl(E, \prod_{\alpha} E_\alpha\Bigr)
\]
is flat by assumption. Hence it follows from Lemma (\ref{lem_Matlis}) that $\prod_{\alpha} R_\alpha$ is flat.

\textbf{1 $\Rightarrow$ 16}.
Suppose $(L_i)_I$ is a family of flat right $R$-modules. Lemma (\ref{lem_stenstrom_fp}) supplies, for every finitely generated left ideal $\mathfrak{a}$, monomorphisms
\[
\mathfrak{a} \otimes \bigl( \prod L_i \bigr) \to \prod \bigl( \mathfrak{a} \otimes L_i \bigr) \to \prod L_i
\]
from which $\prod L_i$ is seen to be flat.

\textbf{16 $\Rightarrow$ 17} is obvious.

\textbf{17 $\Rightarrow$ 1}.
Let $L$ be a finitely generated submodule of a finitely presented module $M$. For any index set $I$ there is a commutative square
\begin{equation}
\begin{tikzcd}
L \otimes A^I \arrow[r] \arrow[d, "\varphi^L"] & M \otimes A^I \arrow[d, "\varphi^M"] \\
L^I \arrow[r] & M^I
\end{tikzcd}
\end{equation}
whose horizontal maps are monomorphisms. As $\varphi^M$ is a monomorphism by Lemma (\ref{lem_stenstrom_fp}), so is $\varphi^L$, and therefore $L$ is finitely presented.
\end{proof}

Consider the following generalization of Chase criterion.

\begin{theorem}\cite{efimov_Th_of_heart}
Let $\CC$ be a compactly assembled additive category, and let $\mathcal{A} = \Fun^{\mathrm{cont},\mathrm{add}}(C, \mathbf{Ab})$ be the category of colimit preserving additive functors. The following are equivalent.
\begin{enumerate}
\item[(i)] $\mathcal{A}$ is coherently assembled.
\item[(ii)] $C$ has infinite products.
\end{enumerate}
\end{theorem}

Here the Chase criterion corresponds to the case $\CC = \mathrm{Flat} - R$, the (compactly generated) category of flat right $R$-modules, and $\mathcal{A} = R - \Mod$, the abelian category of left $R$-modules.

\subsection{Extensions of abelian categories} \label{ssec:ext}

We begin by recalling the notion of a \textit{comma category}.

\begin{definition} \label{def_comacat}
Let $\mathcal{A}$, $\mathcal{B}$, $\CC$ be categories and let 
$F: \mathcal{B} \to \mathcal{A}$ and $G: \CC \to \mathcal{A}$ be functors.

The \textit{comma category} $(F \downarrow G)$ is defined as follows:

\begin{itemize}
\item \textit{Objects}: triples $(B, C, f)$ where $B \in \mathcal{B}$, $C \in \CC$, and $f: F(B) \to G(C)$ is a morphism in $\mathcal{A}$.
\item \textit{Morphisms}: a morphism from $(B, C, f)$ to $(B', C', f')$ is a pair $(\beta: B \to B', \gamma: C \to C')$ such that the diagram

\begin{equation}
\begin{tikzcd}
F(B) \arrow[d, "F(\beta)"'] \arrow[r, "f"] & G(C) \arrow[d, "G(\gamma)"] \\
F(B') \arrow[r, "f'"] & G(C')
\end{tikzcd}
\end{equation}

commutes in $\mathcal{A}$.
\end{itemize}

Composition is defined componentwise: $(\beta',\gamma') \circ (\beta,\gamma) = (\beta'\circ\beta,\; \gamma'\circ\gamma)$.
\end{definition}

\begin{remark}\label{rem_triv_ext}
When $\mathcal{A} = \mathcal{B} = \CC$ and $G = \Id_{\mathcal{A}}$, the comma categories $(\Id_{\mathcal{A}} \downarrow F)$ and $(F \downarrow \Id_{\mathcal{A}})$ coincide with the left and right trivial extensions as developed by R. Fossum, P. Griffith and I. Reiten \cite{reiten}, where the notations are $F \rtimes \mathcal{A}$ and $\mathcal{A} \ltimes F$ respectively.

Note that due to \cite{reiten} when $\AA$ is abelian and $F$ is an additive endofunctor then 
$(\Id_{\mathcal{A}} \downarrow F)$ and $(F \downarrow \Id_{\mathcal{A}})$ are again additive.
\end{remark}

When $\AA$ is abelian and $F$ is an additive endofunctor it is immediately clear that both $(F \downarrow \Id_{\mathcal{A}})$ and $(\Id_{\mathcal{A}} \downarrow F)$ are additive categories. Furthermore, the category $(F \downarrow \Id_{\mathcal{A}})$ has kernels and $(\Id_{\mathcal{A}} \downarrow F)$ has cokernels.

\begin{proposition} \cite{reiten} \label{prop_comacat_ab}
\begin{enumerate}
\item[(a)] If $F$ is right exact, then $(F \downarrow \Id_{\mathcal{A}})$ is abelian.
\item[(b)] If $F$ is left exact, then $(\Id_{\mathcal{A}} \downarrow F)$ is abelian.
\end{enumerate}
\end{proposition}

\begin{definition} \cite{Pirashvili_recollement_ab} \label{def_recoll_situat}
A \textit{recollement situation} consists of three abelian categories $\AA,\, \BB,\, \CC$ together with six additive functors:

\begin{equation}\label{recoll_verydef}
\begin{tikzcd}
	\AA &&& \BB &&& \CC
	\arrow[""{name=0, anchor=center, inner sep=0}, "{i_*}"{description}, from=1-1, to=1-4]
	\arrow[""{name=1, anchor=center, inner sep=0}, "{i^*}"{description}, shift right=5, from=1-4, to=1-1]
	\arrow[""{name=2, anchor=center, inner sep=0}, "{i^!}"{description}, shift left=5, from=1-4, to=1-1]
	\arrow[""{name=3, anchor=center, inner sep=0}, "{j^*}"{description}, from=1-4, to=1-7]
	\arrow[""{name=4, anchor=center, inner sep=0}, "{j_*}"{description}, shift left=5, from=1-7, to=1-4]
	\arrow[""{name=5, anchor=center, inner sep=0}, "{j_!}"{description}, shift right=5, from=1-7, to=1-4]
	\arrow["\dashv"{anchor=center, rotate=-90}, shift left=5, draw=none, from=0, to=2]
	\arrow["\dashv"{anchor=center, rotate=-90}, shift left=5, draw=none, from=1, to=0]
	\arrow["\dashv"{anchor=center, rotate=-90}, shift left=5, draw=none, from=3, to=4]
	\arrow["\dashv"{anchor=center, rotate=-90}, shift left=5, draw=none, from=5, to=3]
\end{tikzcd}  
\end{equation}
that are adjoint, as shown and additionally
satisfy the following conditions:

\begin{enumerate}
\item \(i_*\) is an embedding onto the full subcategory of \(\mathcal{B}\) consisting of objects \(X\) such that \(j^*(X) = 0\).

(i.e. $i_*$ is embedding of the kernel of $j^*$)
\item The unit \(\Id_A \to i^! i_*\) and the counit \(i^* i_* \to \Id_A\) are isomorphisms.

(i.e. $i_*$ is fully faithful)
\item The unit \(\Id_C \to j^* j_!\) and the counit \(j^* j_* \to \Id_C\) are isomorphisms.

(i.e. $j_*$ and $j_!$ are fully faithful)
\end{enumerate}
\end{definition}

\begin{definition}\cite{Psaroudakis_modul} \label{def_equiv_recollements}
Two recollements $R(\mathcal{A}, \mathcal{B}, \CC)$ and $R(\mathcal{A}', \mathcal{B}', \CC')$ are \textbf{equivalent} if there are equivalence functors $\Phi : \mathcal{B} \to \mathcal{B}'$ and $\Theta : \CC \to \CC'$ such that the diagram below commutes up to natural equivalence, i.e. there is a natural equivalence of functors between $\Theta j^*$ and $j^*\Phi$.
\begin{equation}
\begin{tikzcd}
\mathcal{B} \arrow[r, "j^*"] \arrow[d, "\cong"'] & \CC \arrow[d, "\cong"] \\
\mathcal{B}' \arrow[r, "j'^*"] & \CC'
\end{tikzcd}
\end{equation}
\end{definition}

\begin{lemma} \cite{Psaroudakis_modul}
Two recollements $(\mathcal{A}, \mathcal{B}, \CC)$ and $(\mathcal{A}', \mathcal{B}', \CC')$ are equivalent if and only if there are exact equivalences $\Phi : \mathcal{B} \to \mathcal{B}'$, $\Psi : \mathcal{A} \to \mathcal{A}'$ and $\Theta : \CC \to \CC'$ such that the six diagrams associated to the six functors of the recollements commute up to natural equivalences.
\end{lemma}

Unless otherwise stated, we denote by $\mu : j_! j^* \longrightarrow \Id_{\mathcal{B}}$, resp. $\kappa : i_* i^! \longrightarrow \Id_{\mathcal{B}}$, the counit of the adjoint pair $(j_!, j^*)$, resp. $(i_*, i^!)$, and by $\lambda : \Id_{\mathcal{B}} \longrightarrow i_* i^*$, resp. $\nu : \Id_{\mathcal{B}} \longrightarrow j_* j^*$, the unit of the adjoint pair $(i^*, i_*)$, resp. $(j^*, j_*)$.

We want to mention numerous remarks due to \cite{PSAROUDAKIS_homolog_theor} about some properties of a recollement situation.

\begin{remark} \cite[Remark 2.2.]{PSAROUDAKIS_homolog_theor} \label{rem_trivia-recollement}
Every recollement situation enjoys the following properties.

\begin{enumerate}
\item Both $j^* : \mathcal{B} \to \CC$ and $i_* : \mathcal{A} \to \mathcal{B}$ are exact functors.
\item One has the vanishing $i^* j_! = i^! j_* = 0$ of the composites.
\item Through $i_*$ the category $\mathcal{A}$ becomes equivalent to the Serre subcategory $\Ker j^* = \Im i_*$ of $\mathcal{B}$; from now on we treat this equivalence as an identification.
\item As the exact functor $j^*$ carries a fully faithful left adjoint $j_!$ and a fully faithful right adjoint $j_*$, the subcategory $\mathcal{A}$ is at once localizing and colocalizing in $\mathcal{B}$, and the quotient satisfies $\mathcal{B}/\mathcal{A} \simeq \CC$. Consequently each recollement $(\mathcal{A}, \mathcal{B}, \CC)$ gives rise to a short exact sequence of abelian categories $0 \to \mathcal{A} \to \mathcal{B} \to \CC \to 0$.
\end{enumerate}
\end{remark}

The next remark records that a recollement situation $(\AA,\, \BB,\, \CC)$ is already encoded by its right-hand data, with $\AA$ recovered as $\Ker j^*$.

\begin{remark} \cite{PSAROUDAKIS_homolog_theor, ladders} 
Suppose $j^* : \mathcal{B} \to \CC$ is an exact functor of abelian categories sitting inside an adjoint triple $(j_!, j^*, j_*)$ in which at least one of $j_!$, $j_*$ is fully faithful; recall that $j_!$ is fully faithful precisely when $j_*$ is. Write $\mu : j_! j^* \to \Id_{\mathcal{B}}$ for the counit of $(j_!, j^*)$ and $\nu : \Id_{\mathcal{B}} \to j_* j^*$ for the unit of $(j^*, j_*)$. For an arbitrary $B \in \mathcal{B}$ the prescriptions
\[
i^*(B) := \Coker \mu_B, \qquad i^!(B) := \Ker \nu_B
\]
are readily seen to define functors $i^* : \mathcal{B} \to \Ker j^*$ and $i^! : \mathcal{B} \to \Ker j^*$. A direct verification shows that $(i^*, i_*, i^!)$ forms an adjoint triple, with $i_* : \Ker j^* \to \mathcal{B}$ the inclusion; thus $(\Ker j^*, \mathcal{B}, \CC)$ is itself a recollement of abelian categories.

It then remains to match this recollement against an arbitrary one. Given a recollement $(\mathcal{A}, \mathcal{B}, \CC)$ and the recollement $(\Ker j^*, \mathcal{B}, \CC)$ produced from the right-hand adjoint triple as above, the two are equivalent in the sense of Definition (\ref{def_equiv_recollements}).

In other words, up to equivalence one can rebuild the original recollement $(\mathcal{A}, \mathcal{B}, \CC)$ out of the adjoint triple $(j_!, j^*, j_*)$ alone, which supplies a second, equivalent route to the definition of a recollement of abelian categories.
\end{remark}

\begin{remark} \cite[Remark 2.4.]{PSAROUDAKIS_homolog_theor} 
Let $\mathcal{B}$ be an abelian category equipped with a Serre subcategory $\mathcal{A}$, so that there is an exact sequence of abelian categories $0 \to \mathcal{A} \to \mathcal{B} \to \mathcal{B}/\mathcal{A} \to 0$. Suppose in addition that $\mathcal{B}$ is a Grothendieck category admitting projective covers and that $\mathcal{A}$ is closed under products and coproducts. Under these hypotheses \cite{KRAUSE_spec_loc_coh_cat} ensures that the quotient functor $j^* : \mathcal{B} \to \mathcal{B}/\mathcal{A}$ possesses both a left adjoint $j_!$ and a right adjoint $j_*$. Remark (\ref{rem_trivia-recollement}) then yields the recollement of abelian categories $(\mathcal{A}, \mathcal{B}, \mathcal{B}/\mathcal{A})$.
\end{remark}

The proposition below goes back to Franjou and Pirashvili \cite{Pirashvili_recollement_ab}; we reproduce the argument from \cite{PSAROUDAKIS_homolog_theor}.

\begin{proposition} \cite{Pirashvili_recollement_ab} \label{prop_sess_of_recollement}
Retaining the recollement situation $(\AA, \BB, \CC)$ as above, fix $B \in \mathcal{B}$. Then there are exact sequences
\[
0 \to \Ker \mu_B \to j_! j^*(B) \xrightarrow{\mu_B} B \xrightarrow{\lambda_B} i_* i^*(B) \to 0
\]
and
\[
0 \to i_* i^!(B) \xrightarrow{\kappa_B} B \xrightarrow{\nu_B} j_* j^*(B) \to \Coker \nu_B \to 0,
\]
in which the objects $\Ker \mu_B$ and $\Coker \nu_B$ belong to $i_*(\mathcal{A})$.
\end{proposition}

\begin{proof} \cite[Proposition 2.6.]{PSAROUDAKIS_homolog_theor}
The counit $\mu_B : j_! j^*(B) \to B$ of the pair $(j_!, j^*)$ yields the exact sequence
\[
0 \to \Ker \mu_B \to j_! j^*(B) \to B \to \Coker \mu_B \to 0
\]
in $\mathcal{B}$. Pushing it through the exact functor $j^* : \mathcal{B} \to \CC$ gives $j^*(\Ker \mu_B) = 0 = j^*(\Coker \mu_B)$, so that both $\Ker \mu_B$ and $\Coker \mu_B$ lie in $i_*(\mathcal{A})$.
Writing $\Coker \mu_B = i_*(A)$ with $A \in \mathcal{A}$, property (i) supplies the isomorphism $i_* i^*(\Coker \mu_B) \simeq \Coker \mu_B$.
Right exactness of $i_* i^*$ combined with $i^* j_! = 0$ now forces the claim through the commutative diagram
\begin{equation}
\begin{tikzcd}
j_! j^*(B) \arrow[r, "\mu_B"] & B \arrow[r] \arrow[d, "\lambda_B"] & \Coker \mu_B \arrow[r] \arrow[d, "\lambda_{\Coker \mu_B}"] & 0 \\
0 \arrow[r] & i_* i^*(B) \arrow[r, "\sim"] & i_* i^*(\Coker \mu_B) \arrow[r] & 0
\end{tikzcd}
\end{equation}
Exactness of the remaining sequence $0 \to i_* i^!(B) \xrightarrow{\kappa_B} B \xrightarrow{\nu_B} j_* j^*(B) \to \Coker \nu_B \to 0$ is obtained dually.
\end{proof}

In certain cases the terms $\Ker \mu_B$ and $\Coker \nu_B$ occurring in the sequences of (\ref{prop_sess_of_recollement}) admit an explicit description.

\begin{proposition} \cite{Pirashvili_recollement_ab} \label{prop_pirash_recol_sleva_iLi}
Take $A \in \Ker i^*$. Then the counit $\epsilon_A : j_! j^* A \to A$ is an epimorphism whose kernel sits in $i_* \mathcal{A}$. In fact, whenever $A$ has enough projectives one obtains a short exact sequence
\[
0 \to i_* (L_1 i^*) A \to j_! j^* A \xrightarrow{\epsilon_A} A \to 0.
\]
\end{proposition}

\begin{proposition} \cite{Pirashvili_recollement_ab} \label{prop_pirash_recol_sprava_iRi}
Take $A \in \Ker i^!$. Then the unit $\eta_A : A \to j_* j^* A$ is a monomorphism whose cokernel sits in $i_* \mathcal{A}$. Indeed, provided $A$ has enough injectives, there is a short exact sequence
\[
0 \to A \xrightarrow{\eta_A} j_* j^* A \to i_* (R^1 i^!) A \to 0.
\]
\end{proposition}

\begin{proof}
If $i^! A = 0$, the unit sequence collapses to the short exact sequence
\[
0 \to A \xrightarrow{\eta_A} j_* j^* A \to \Coker \eta_A \to 0.
\]
Applying first the exact functor $j^*$ and recalling that $j^* \eta$ is an isomorphism, we obtain $j^*(\Coker \eta_A) = 0$, so $\Coker \eta_A$ belongs to $i_* \mathcal{A}$. Assume now that $A$ has enough injectives. Applying the left exact functor $i^!$, the long exact sequence of right derived functors furnishes the exact sequence
\[
0 \to i^! A \to i^! j_* j^* A \to i^! \Coker \eta_A \to (R^1 i^!) A \to (R^1 i^!) j_* j^* A.
\]
Combined with the vanishing $i^! j_* = 0$, this produces an isomorphism $i^! \Coker(\eta_A) \cong R^1 i^!(A)$.
\end{proof}

We next revisit a few standard examples of recollements of abelian categories.

\begin{example}\cite{PSAROUDAKIS_homolog_theor} \label{ex_idemp_ring_recoll}
Let $R$ be a ring and $e^2 = e \in R$ an idempotent. Then we have a recollement

\begin{equation}
\begin{tikzcd}
	{\Mod\text{-}R/ReR} &&& {\Mod\text{-}R} &&& {\Mod\text{-}eRe}
	\arrow[""{name=0, anchor=center, inner sep=0}, "{i_*}"{description}, from=1-1, to=1-4]
	\arrow[""{name=1, anchor=center, inner sep=0}, "{i^* = R/ReR \otimes_R \_}"', shift right=5, from=1-4, to=1-1]
	\arrow[""{name=2, anchor=center, inner sep=0}, "{i^! = \Hom_R(R/ReR, \_)}", shift left=5, from=1-4, to=1-1]
	\arrow[""{name=3, anchor=center, inner sep=0}, "{j^* = e(-)}"{description}, from=1-4, to=1-7]
	\arrow[""{name=4, anchor=center, inner sep=0}, "{j_* = \Hom_{eRe}(eR, \_)}", shift left=5, from=1-7, to=1-4]
	\arrow[""{name=5, anchor=center, inner sep=0}, "{j_! = Re \otimes_{eRe} \_}"', shift right=5, from=1-7, to=1-4]
	\arrow["\dashv"{anchor=center, rotate=-90}, shift left=5, draw=none, from=0, to=2]
	\arrow["\dashv"{anchor=center, rotate=-90}, shift left=5, draw=none, from=3, to=4]
	\arrow["\dashv"{anchor=center, rotate=-90}, shift left=5, draw=none, from=1, to=0]
	\arrow["\dashv"{anchor=center, rotate=-90}, shift left=5, draw=none, from=5, to=3]
\end{tikzcd}.
\end{equation}
Here the category $\Mod\text{-}R/ReR$ of (left) $R/ReR$-modules is identified with the kernel of $e(-) : \Mod\text{-}R \to \Mod\text{-}eRe$. Recall that over a semiprimary ring $R$ every idempotent ideal $I$ takes the shape $ReR$ for some idempotent $e \in R$; hence in such rings idempotent ideals invariably give rise to recollements of the associated module categories.
\end{example}

\begin{example}\cite{PSAROUDAKIS_homolog_theor} \label{ex_psa_matrix_ring}
Let $R, S$ be rings, $M$ an $S$-$R$-bimodule and $N$ an $R$-$S$-bimodule. Let  
\[
\phi : M \otimes_R N \longrightarrow S
\]
be an $S$-$S$-bimodule homomorphism and let  
\[
\psi : N \otimes_S M \longrightarrow R
\]
be an $R$-$R$-bimodule homomorphism. Then the above data allow us to define the \textbf{Morita ring}:
\[
\Lambda_{(\phi, \psi)} = 
\begin{pmatrix}
R & {}_R N_S \\
{}_S M_R & S
\end{pmatrix}
\]
where the addition of elements of $\Lambda_{(\phi, \psi)}$ is componentwise and multiplication is given by
\[
\begin{pmatrix}
r & n \\
m & s
\end{pmatrix}
\cdot
\begin{pmatrix}
r' & n' \\
m' & s'
\end{pmatrix}
=
\begin{pmatrix}
rr' + \psi(n \otimes m') & rn' + ns' \\
mr' + sm' & ss' + \phi(m \otimes n')
\end{pmatrix}.
\]

The $\Lambda_{(\phi, \psi)}$-modules are tuples $(X, Y, f, g)$ where $X \in \Mod\text{-}R$, $Y \in \Mod\text{-}S$, $f : M \otimes_R X \to Y$ and $g : N \otimes_S Y \to X$ such that the following diagrams are commutative:

\begin{equation}
\begin{tikzcd}
N \otimes_S M \otimes_R X \arrow[r, "1_N \otimes f"] \arrow[dr, "\psi \otimes 1_X"'] & N \otimes_S Y \arrow[d, "g"] \\
& X
\end{tikzcd}
\qquad
\begin{tikzcd}
M \otimes_R N \otimes_S Y \arrow[r, "1_M \otimes g"] \arrow[dr, "\phi \otimes 1_Y"'] & M \otimes_R X \arrow[d, "f"] \\
& Y
\end{tikzcd}
\end{equation}

If $(X, Y, f, g)$ and $(X', Y', f', g')$ are $\Lambda_{(\phi, \psi)}$-modules, then a morphism $(X, Y, f, g) \to (X', Y', f', g')$ in $\Mod\text{-}\Lambda_{(\phi, \psi)}$ is a pair of homomorphisms $(a, b)$ where $a : X \to X'$ is a morphism in $\Mod\text{-}R$ and $b : Y \to Y'$ is a morphism in $\Mod\text{-}S$ such that the following diagrams are commutative:

\begin{equation}
\begin{tikzcd}
M \otimes_R X \arrow[r, "f"] \arrow[d, "1_M \otimes a"'] & Y \arrow[d, "b"] \\
M \otimes_R X' \arrow[r, "f'"] & Y'
\end{tikzcd}
\qquad
\begin{tikzcd}
N \otimes_S Y \arrow[r, "g"] \arrow[d, "1_N \otimes b"'] & X \arrow[d, "a"] \\
N \otimes_S Y' \arrow[r, "g'"] & X'
\end{tikzcd}
\end{equation}

Then from Example (\ref{ex_idemp_ring_recoll}) using the idempotent elements \( e_1 = \begin{pmatrix} 1_R & 0 \\ 0 & 0 \end{pmatrix} \) and \( e_2 = \begin{pmatrix} 0 & 0 \\ 0 & 1_S \end{pmatrix} \) of \( \Lambda(\phi, \psi) \), we derive the following recollements of abelian categories:
\begin{equation}
\begin{tikzcd}
	{\Mod-S/\im\phi} &&& {\Mod-\Lambda_{(\phi,\psi)}} &&& {\Mod-R}
	\arrow[""{name=0, anchor=center, inner sep=0}, "{i_*=\mathrm{inc}}"{description}, from=1-1, to=1-4]
	\arrow[""{name=1, anchor=center, inner sep=0}, "{i^* = \Lambda/\Lambda e_1\Lambda \otimes_\Lambda \_}"', shift right=5, from=1-4, to=1-1]
	\arrow[""{name=2, anchor=center, inner sep=0}, "{i^! = \Hom_\Lambda(\Lambda/\Lambda e_1\Lambda, \_)}", shift left=5, from=1-4, to=1-1]
	\arrow[""{name=3, anchor=center, inner sep=0}, "{j^* = e_1(-)}"{description}, from=1-4, to=1-7]
	\arrow[""{name=4, anchor=center, inner sep=0}, "{j_* = \Hom_{e_1\Lambda e_1}(e_1\Lambda, \_)}", shift left=5, from=1-7, to=1-4]
	\arrow[""{name=5, anchor=center, inner sep=0}, "{j_! = \Lambda e_1 \otimes_{e_1\Lambda e_1} \_}"', shift right=5, from=1-7, to=1-4]
	\arrow["\dashv"{anchor=center, rotate=-90}, shift left=5, draw=none, from=0, to=2]
	\arrow["\dashv"{anchor=center, rotate=-90}, shift left=5, draw=none, from=3, to=4]
	\arrow["\dashv"{anchor=center, rotate=-90}, shift left=5, draw=none, from=1, to=0]
	\arrow["\dashv"{anchor=center, rotate=-90}, shift left=5, draw=none, from=5, to=3]
\end{tikzcd},
\end{equation} \\
\begin{equation}
\begin{tikzcd}
	{\Mod-R/\im\psi} &&& {\Mod-\Lambda_{(\phi,\psi)}} &&& {\Mod-S}
	\arrow[""{name=0, anchor=center, inner sep=0}, "{i_*=\mathrm{inc}}"{description}, from=1-1, to=1-4]
	\arrow[""{name=1, anchor=center, inner sep=0}, "{i^* = \Lambda/\Lambda e_2\Lambda \otimes_\Lambda -}"', shift right=5, from=1-4, to=1-1]
	\arrow[""{name=2, anchor=center, inner sep=0}, "{i^! = \Hom_\Lambda(\Lambda/\Lambda e_2\Lambda, -)}", shift left=5, from=1-4, to=1-1]
	\arrow[""{name=3, anchor=center, inner sep=0}, "{j^* = e_2(-)}"{description}, from=1-4, to=1-7]
	\arrow[""{name=4, anchor=center, inner sep=0}, "{j_* = \Hom_{e_2\Lambda e_2}(e_2\Lambda, -)}", shift left=5, from=1-7, to=1-4]
	\arrow[""{name=5, anchor=center, inner sep=0}, "{j_! = \Lambda e_2 \otimes_{e_2\Lambda e_2} -}"', shift right=5, from=1-7, to=1-4]
	\arrow["\dashv"{anchor=center, rotate=-90}, shift left=5, draw=none, from=0, to=2]
	\arrow["\dashv"{anchor=center, rotate=-90}, shift left=5, draw=none, from=3, to=4]
	\arrow["\dashv"{anchor=center, rotate=-90}, shift left=5, draw=none, from=1, to=0]
	\arrow["\dashv"{anchor=center, rotate=-90}, shift left=5, draw=none, from=5, to=3]
\end{tikzcd}.
\end{equation}

\end{example}

\begin{example} \label{ex_comacat_recoll}
Let $G : \mathcal{B} \to \mathcal{A}$ be a right exact functor of abelian categories and form the comma category $\CC = (G \downarrow \Id_{\mathcal{A}})$ as in Definition (\ref{def_comacat}). Right exactness of $G$ makes $\CC$ abelian by Proposition (\ref{prop_comacat_ab}).

Suppose furthermore that $G$ admits a right adjoint $G' : \mathcal{A} \to \mathcal{B}$, and denote by $\varepsilon: G G' \to \Id_{\mathcal{A}}$ and $\eta: \Id_{\mathcal{B}} \to G' G$ the counit and the unit of the adjoint pair $(G, G')$.
The relevant functors are the following:

\begin{enumerate}
\item The functor $T_{\mathcal{B}} : \mathcal{B} \to \CC$ is defined on objects $Y \in \mathcal{B}$ by $T_{\mathcal{B}}(Y) = (G(Y), Y, \Id_{GY})$ and given a morphism $\beta : Y \to Y'$ in $\mathcal{B}$ then $T_{\mathcal{B}}(\beta) = (G(\beta), \beta)$ is a morphism in $\CC$.

\item The functor $U_{\mathcal{B}} : \CC \to \mathcal{B}$ is defined on objects $(A, B, f) \in \CC$ by $U_{\mathcal{B}}(A, B, f) = B$ and given a morphism $(\alpha, \beta) : (A, B, f) \to (A', B', f')$ in $\CC$ then $U_{\mathcal{B}}(\alpha, \beta) = \beta$ is a morphism in $\mathcal{B}$. Similarly we define the functor $U_{\mathcal{A}} : \CC \to \mathcal{A}$.

\item The functor $Z_{\mathcal{B}} : \mathcal{B} \to \CC$ is defined on objects $Y \in \mathcal{B}$ by $Z_{\mathcal{B}}(Y) = (0, Y, 0)$ and given a morphism $\beta : Y \to Y'$ in $\mathcal{B}$ then $Z_{\mathcal{B}}(\beta) = (0, \beta)$ is a morphism in $\CC$. Similarly we define the functor $Z_{\mathcal{A}} : \mathcal{A} \to \CC$.

\item The functor $q : \CC \to \mathcal{A}$ is defined on objects $(A, B, f) \in \CC$ by $q(A, B, f) = \Coker f$ and if $(\alpha, \beta) : (A, B, f) \to (A', B', f')$ is a morphism in $\CC$ then we have the induced morphism $q(\alpha, \beta) : \Coker f \to \Coker f'$.

\item The functor $H_{\mathcal{A}} : \mathcal{A} \to \CC$ is defined on objects $X \in \mathcal{A}$ by $H_{\mathcal{A}}(X) = (X, G'(X), \varepsilon_X)$ and given a morphism $\alpha : X \to X'$ in $\mathcal{A}$ then $H_{\mathcal{A}}(\alpha) = (\alpha, G'(\alpha))$ is a morphism in $\CC$.

\item The functor $p : \CC \to \mathcal{B}$ is defined on objects $(A, B, f) \in \CC$ by $p(A, B, f) = \Ker (\eta_B \circ G'(f))$ and if $(\alpha, \beta) : (A, B, f) \to (A', B', f')$ is a morphism in $\CC$ then we have the induced morphism $p(\alpha, \beta) : \Ker (\eta_B \circ G'(f)) \to \Ker (\eta_{B'} \circ G'(f'))$.
\end{enumerate}

It is easy to check that the above data define the following recollements:

\begin{equation} \label{recoll_coma_acb}
\begin{tikzcd}
	{\mathcal{A}} &&& {\CC} &&& {\mathcal{B}}
	\arrow[""{name=0, anchor=center, inner sep=0}, "{Z_{\mathcal{A}}}"{description}, from=1-1, to=1-4]
	\arrow[""{name=1, anchor=center, inner sep=0}, "{q}"'{description}, shift right=5, from=1-4, to=1-1]
	\arrow[""{name=2, anchor=center, inner sep=0}, "{U_{\mathcal{A}}}"{description}, shift left=5, from=1-4, to=1-1]
	\arrow[""{name=3, anchor=center, inner sep=0}, "{U_{\mathcal{B}}}"{description}, from=1-4, to=1-7]
	\arrow[""{name=4, anchor=center, inner sep=0}, "{Z_{\mathcal{B}}}"{description}, shift left=5, from=1-7, to=1-4]
	\arrow[""{name=5, anchor=center, inner sep=0}, "{T_{\mathcal{B}}}"'{description}, shift right=5, from=1-7, to=1-4]
	\arrow["\dashv"{anchor=center, rotate=-90}, shift left=5, draw=none, from=0, to=2]
	\arrow["\dashv"{anchor=center, rotate=-90}, shift left=5, draw=none, from=1, to=0]
	\arrow["\dashv"{anchor=center, rotate=-90}, shift left=5, draw=none, from=3, to=4]
	\arrow["\dashv"{anchor=center, rotate=-90}, shift left=5, draw=none, from=5, to=3]
\end{tikzcd},
\end{equation} \\

\begin{equation}\label{recoll_coma_bca}
\begin{tikzcd}
	{\mathcal{B}} &&& {\CC} &&& {\mathcal{A}}
	\arrow[""{name=0, anchor=center, inner sep=0}, "{Z_{\mathcal{B}}}"{description}, from=1-1, to=1-4]
	\arrow[""{name=1, anchor=center, inner sep=0}, "{U_{\mathcal{B}}}"{description}, shift right=5, from=1-4, to=1-1]
	\arrow[""{name=2, anchor=center, inner sep=0}, "{p}"{description}, shift left=5, from=1-4, to=1-1]
	\arrow[""{name=3, anchor=center, inner sep=0}, "{U_{\mathcal{A}}}"{description}, from=1-4, to=1-7]
	\arrow[""{name=4, anchor=center, inner sep=0}, "{H_{\mathcal{A}}}"{description}, shift left=5, from=1-7, to=1-4]
	\arrow[""{name=5, anchor=center, inner sep=0}, "{Z_{\mathcal{A}}}"{description}, shift right=5, from=1-7, to=1-4]
	\arrow["\dashv"{anchor=center, rotate=-90}, shift left=5, draw=none, from=0, to=2]
	\arrow["\dashv"{anchor=center, rotate=-90}, shift left=5, draw=none, from=1, to=0]
	\arrow["\dashv"{anchor=center, rotate=-90}, shift left=5, draw=none, from=3, to=4]
	\arrow["\dashv"{anchor=center, rotate=-90}, shift left=5, draw=none, from=5, to=3]
\end{tikzcd}.
\end{equation}

Observe that the composite $U_\AA \circ T_\BB$ recovers exactly the original functor
$G: \BB \to \AA$.
\end{example}

In the example above, whenever $\CC$ is a comma category, the functors $U_\AA, Z_\BB$ (respectively $U_\BB, Z_\AA$) occurring in the recollements (\ref{recoll_coma_acb}) (resp.\ (\ref{recoll_coma_bca})) are right (resp.\ left) exact, being left (resp.\ right) adjoints. No such requirement is placed on a general recollement situation (\ref{def_recoll_situat}).

As it happens, granting the mild hypotheses of enough projectives (resp.\ injectives) in $\BB$ and $\CC$, exactness of $i^!$ (resp.\ $i^*$) in the recollement (\ref{recoll_verydef}) turns out to be a {\bf criterion} for that recollement to be equivalent, in the sense of Definition (\ref{def_equiv_recollements}), to one arising from a comma category as in Example (\ref{ex_comacat_recoll}).
Precisely, we have the following.

\begin{proposition}\cite{ladders}
Let $(\mathcal{A}, \mathcal{B}, \CC)$ be a recollement of abelian categories with $i^!$ exact and with $\mathcal{B}$ and $\CC$ possessing enough projective objects. Then, as recollements, $(\mathcal{A}, \mathcal{B}, \CC)$ and $(\mathcal{A}, (i^! j_! \downarrow \Id_{\AA}), \CC)$ are equivalent.
\end{proposition}

\begin{proposition}\cite{ladders}
Let $(\mathcal{A}, \mathcal{B}, \CC)$ be a recollement of abelian categories with $i^*$ exact and with $\mathcal{B}$ and $\CC$ possessing enough injective objects. Then, as recollements, $(\mathcal{A}, \mathcal{B}, \CC)$ and $(\AA, (\Id_{\AA} \downarrow i^*j_*), \CC)$ are equivalent.
\end{proposition}

\section{Local coherency theorems} \label{sec:thms}

As shown in Proposition (\ref{prop_sess_of_recollement}), while {\it decomposing} $\Id_\BB$ into parts, coming 
from $\AA$ and from $\CC$, one is really interested in some partial data from a recollement situation, 
namely either two pairs 
$(i^* \dashv i_*),\,(j_! \dashv j^*)$ of adjoint functors or $(i_* \dashv i^!),\,(j^* \dashv j_*)$.
So from now on we are going to impose only the specified conditions rather than general six-functor data of recollement.

\subsection{Direct course} \label{ssec:direct}

\begin{theorem} \label{th_coh_B_straightforward}
Let $\AA$, $\BB$, $\CC$ be AB5 locally finitely presentable abelian categories and the following adjoint 
functors between them
\begin{equation}
\begin{tikzcd}
	\AA &&& \BB &&& \CC
	\arrow[""{name=0, anchor=center, inner sep=0}, "{{i_*}}"{description}, from=1-1, to=1-4]
	\arrow[""{name=1, anchor=center, inner sep=0}, "{{i^*}}"{description}, shift right=5, from=1-4, to=1-1]
	\arrow[""{name=2, anchor=center, inner sep=0}, "{{j^*}}"{description}, from=1-4, to=1-7]
	\arrow[""{name=3, anchor=center, inner sep=0}, "{{j_!}}"{description}, shift right=5, from=1-7, to=1-4]
	\arrow["\dashv"{anchor=center, rotate=-90}, shift left=5, draw=none, from=1, to=0]
	\arrow["\dashv"{anchor=center, rotate=-90}, shift left=5, draw=none, from=3, to=2]
\end{tikzcd}.
\end{equation}

Where \(i_*\) is a fully faithful embedding of the $\Ker(j^*)$ and $j_!$ is fully faithful.
Moreover let $i_*$, $j^*$ be exact, commute with filtered colimits and preserve finitely presented objects.
Suppose, in addition, that $\mathbb{L}_1 i^*$ sends finitely generated subobjects of finitely presented objects to finitely generated objects.

Then, if $\AA$ is locally Noetherian and $\CC$ is locally coherent, $\BB$ is locally coherent.
\end{theorem}

As a special case of the previous result, one obtains a statement about the coherence of an algebra over a Noetherian ring, deduced from the coherence of the endomorphism algebra of a special augmentation ideal.

\begin{corollary}\label{cor_B_B+_situation}
Let $R$ be an associative ring, let $I \subset R$ be a finitely generated, projective, idempotent two-sided ideal and let $k := R/I$, $E:= \End_R(I)$.

Then, if $k$ is left (right) Noetherian and $E$ is left (right) coherent, $R$ is left (right) coherent.
\end{corollary}

From now on, we say {\it module} for {\it left module} and prove the result in the left-module case; the right-module case is dual.

Consider the following functors between categories $k - \Mod$, $R- \Mod$, $E - \Mod$, adjoint in the obvious way. 
\begin{equation}\label{def_k_R_E_func}
\begin{tikzcd}
	{k-\Mod} &&&& {R-\Mod} &&&& {E-\Mod}
	\arrow["{\res_{\psi}}"{description}, from=1-1, to=1-5]
	\arrow["{{{k \otimes_R\, \_}}}"', shift right=4, from=1-5, to=1-1]
	\arrow["{{\Hom_R(k,\,\_)}}", shift left=4, from=1-5, to=1-1]
	\arrow["{\Hom_R(I,\,\_) \cong I^\vee \otimes_R\, \_}"{description}, from=1-5, to=1-9]
	\arrow["{{{I\otimes_E\,\_}}}"', shift right=4, from=1-9, to=1-5]
	\arrow["{\Hom_E(I^\vee,\, \_)}", shift left=4, from=1-9, to=1-5]
\end{tikzcd}.
\end{equation}

Where $\res_\psi$ is a restriction of scalars along $\psi$ for the short exact sequence 
\begin{equation} \label{ses_I_R_k}
0 \to I \xrightarrow{} R \xrightarrow{\psi} k \to 0
\end{equation}
in $R-\Mod$ and $I^\vee := \Hom_R(I,\, R)$ is an $(E, R)$ - bimodule. Since $I$ is finitely generated and projective as a left $R$-module, the functor $\Hom_R(I,\,\_)$ preserves all colimits, so by the Eilenberg--Watts theorem~\cite{Watts60,Eilenberg60} $\Hom_R(I,\,\_) \cong I^\vee \otimes_R\, \_$ holds.

We now recall some simple remarks on idempotent modules that are needed in the proof of Corollary~(\ref{cor_B_B+_situation}).

\begin{lemma} \label{lem_idempotent}
Let $R$ be a unital ring, $I$ an ideal, $M \in R$-$\Mod$.
\begin{enumerate}
\item $\Hom_R(I, M) = 0$ implies $IM = 0$, and $I \otimes M = 0$ implies $IM = 0$.
\item If $I$ is an idempotent ideal, then the converse holds: \
$IM = 0$ implies $\Hom_R(I, M) = 0$, and $IM = 0$ implies $I \otimes M = 0$.
\end{enumerate}
\end{lemma}

\begin{remark} \label{remark_Hom(I,R)=E_pres_fp,lim}  
In particular (\ref{lem_idempotent}) implies that $I^\vee \cong E$ in $E-\Mod$. 
Indeed, applying $\Hom_R(I,\, \_)$ to sequence (\ref{ses_I_R_k}) results in the sequence
\begin{equation}
0 \to \Hom_R(I, I) \to \Hom_R(I, R) \to \Hom_R(I, k) \to 0,
\end{equation}
where exactness follows from projectivity of $I$.
Now idempotence and flatness of $I$ imply $I \otimes_R k = 0$ which implies $Ik = 0$ which implies $\Hom_R(I, k) = 0$ which concludes the argument.

So compactness of $I^\vee$ implies that $\Hom_R(I,\, \_): R-\Mod \to E-\Mod$ preserves finitely presented modules.

Clearly $\Hom(I,\, \_)$ preserves colimits as a left adjoint.
\end{remark}

\begin{proof}[Proof of Corollary \ref{cor_B_B+_situation}]
To apply Theorem~(\ref{th_coh_B_straightforward}) to the Diagram~(\ref{def_k_R_E_func}), it remains to check that the conditions on $\res_\psi$ and $\Tor_1^R(k,\,\_)$ are satisfied.

Note that $\res_\psi$ is left adjoint to $\Hom_R(k,\,\_)$, and $k$ is a finitely presented $R$-module by~(\ref{ses_I_R_k}); hence $\res_\psi$ preserves finitely presented objects and commutes with filtered colimits.

For a finitely generated submodule $X \hookrightarrow Y$ of a finitely presented module $Y$, consider the exact sequence
\begin{equation}
\Tor_2^R(k, Y/X) \to \Tor_1^R(k, X) \to \Tor_1^R(k, Y).
\end{equation}

$\Tor_2^R(k, Y/X) = 0$ since, by~(\ref{ses_I_R_k}), $k$ has flat dimension $1$. For a presentation $0 \to Z \to R^n \to Y \to 0$ we have
\begin{equation}
0 \to \Tor_1^R(k,\,Y) \to k \otimes Z \to \dots\,\,\,.
\end{equation}
Hence $\Tor_1^R(k, Y)$ is finitely generated, being a submodule of the finitely generated module $k \otimes Z$ over a Noetherian ring. Therefore $\Tor_1^R(k, X)$ is finitely generated as a submodule of the finitely generated module $\Tor_1^R(k, Y)$, as desired.
\end{proof}

\begin{remark}
Note that in Corollary~(\ref{cor_B_B+_situation}) one can specialise to the setting of {\it homotopes} as developed in~\cite{Bondal_Zhdanovskiy_21}.

Namely, let $A$ be an associative $k$-algebra and let $\Delta \in A$. The {\it homotope} is the associative $k$-algebra $A_\Delta$ (in general without a unit), with the same underlying set and the same $k$-linear structure as $A$, $A_\Delta \cong A$ in $k - \Mod$, but with multiplication defined by $x \times_{A_\Delta} y = x \Delta y$ for all $x, y \in A_\Delta$. We write $B_+ = A_\Delta$.

A unit can be adjoined to any associative algebra functorially, so set $B = \hat A_\Delta := A_\Delta \oplus k$. Then $B_+ \hookrightarrow B$ is the augmentation ideal. We say that $\Delta$ is {\it well-tempered} if $B_+$ is projective and idempotent. One can show that $\End_B(B_+) \cong A$, and as a corollary of Corollary~(\ref{cor_B_B+_situation}) we obtain:
\end{remark}

\begin{theorem}\cite[Theorem 49.]{Bondal_Zhdanovskiy_21} \label{th_bondal_zhdanovskiy}
Let $\Delta \in A$ be well-tempered and let $B = \hat A_\Delta$ . Assume that $A$ is
a coherent algebra over a Noetherian commutative ring $k$. Then $B$ is also coherent.
\end{theorem}

Now we get back to the proof of general result.

\begin{proof}[Proof of Theorem \ref{th_coh_B_straightforward}]
Let $B_1, B_2 \in \BB^{f.p.}$ and let $\varphi \in \Hom_\BB(B_1, B_2)$. It suffices to show that $I = \ker(\varphi) \in \BB^{f.p.}$; by Proposition~(\ref{prop_enough_to_check_on_epi}), we may assume that $\varphi$ is an epimorphism. Thus we have a short exact sequence
\begin{equation}\label{ses_I_B2_B1}
0 \to I \to B_1 \xrightarrow{\varphi} B_2 \to 0.
\end{equation}
By Theorem~(\ref{th_fg_fg_fp}), $I$ is finitely generated; by Theorem~(\ref{th_locpres_fg}), there is an epimorphism $B_3 \xrightarrowdbl{\psi} I$ from some $B_3 \in \BB^{f.p.}$. Setting $K = \ker(\psi)$, we have
\begin{equation}\label{ses_K_B3_I}
0 \to K \to B_3 \xrightarrow{\psi} I \to 0.
\end{equation}
It suffices to show that $K \in \BB^{f.g.}$. Applying the exact sequence of functors $j_! j^* \to \Id_\BB \to i_* i^* \to 0$ to $K$, we obtain
\begin{equation}
j_! j^*(K) \longrightarrow K \longrightarrow i_* i^*(K) \longrightarrow 0.
\end{equation}
By Lemma~(\ref{lem_ab_fp_fg_via_ker_coker}), it suffices to show that $j_! j^*(K), i_* i^*(K) \in \BB^{f.g.}$.

We begin with $j_! j^*(K)$. Applying $j^*$ to~(\ref{ses_I_B2_B1}) and~(\ref{ses_K_B3_I}) we obtain
\begin{equation}
0 \to j^*I \to j^*B_1 \to j^*B_2 \to 0,
\end{equation}
\begin{equation}
0 \to j^*K \to j^*B_3 \to j^*I \to 0.
\end{equation}
Since $j^*$ preserves finitely presented objects, $j^*B_1, j^*B_2, j^*B_3 \in \CC^{f.p.}$. Local coherence of $\CC$ implies $j^*I \in \CC^{f.p.}$ and hence $j^*K \in \CC^{f.p.}$. Since $j_!$ preserves finitely presented objects, $j_! j^*(K) \in \BB^{f.p.}$.

We now study $i_* i^*(K)$. Since $i^*$ is not exact, applying $i^*$ to~(\ref{ses_K_B3_I}) we only obtain
\begin{equation} \label{ses_Li*I_i*K_i*B3_i*I}
\mathbb{L}_1i^*I \to i^*K \to i^*B_3 \xrightarrow{i^*\psi} i^*I \to 0.
\end{equation}
Denote $X = \ker(i^*\psi)$; then~(\ref{ses_Li*I_i*K_i*B3_i*I}) splits into two exact sequences:

\begin{equation} 
0 \to X \to  i^*B_3 \to i^*I \to 0,
\end{equation}
\begin{equation} 
\mathbb{L}_1i^*I \to i^*K \to X \to 0.
\end{equation}
Since $i^*$ preserves finitely presented objects, $i^*B_3 \in \AA^{f.p.}$. Local Noetherianness of $\AA$ implies $X \in \AA^{f.g.}$ as a subobject of a finitely generated object. Since $I$ is a finitely generated subobject of a finitely presented object, the hypothesis on $\mathbb{L}_1 i^*$ gives $\mathbb{L}_1 i^*(I) \in \AA^{f.g.}$. By Lemma~(\ref{lem_ab_fp_fg_via_ker_coker}), $i^*K \in \AA^{f.g.}$, which is equivalent to $i^*K \in \AA^{f.p.}$ in a locally Noetherian category. Since $i_*$ preserves finitely presented objects, $i_* i^*(K) \in \BB^{f.p.}$, as desired.
\end{proof}

Note that in Theorem~(\ref{th_coh_B_straightforward}) the local Noetherianness assumption on $\AA$ is used only because of the non-exactness of $i^*$. If we instead require $i^*$ to be exact, then the local Noetherianness assumption can be weakened.

\begin{theorem} \label{th_coh_B_i*exact}
Let $\AA$, $\BB$, $\CC$ be AB5 locally finitely presentable abelian categories and the following adjoint 
functors between them
\begin{equation}
\begin{tikzcd}
	\AA &&& \BB &&& \CC
	\arrow[""{name=0, anchor=center, inner sep=0}, "{{i_*}}"{description}, from=1-1, to=1-4]
	\arrow[""{name=1, anchor=center, inner sep=0}, "{{i^*}}"{description}, shift right=5, from=1-4, to=1-1]
	\arrow[""{name=2, anchor=center, inner sep=0}, "{{j^*}}"{description}, from=1-4, to=1-7]
	\arrow[""{name=3, anchor=center, inner sep=0}, "{{j_!}}"{description}, shift right=5, from=1-7, to=1-4]
	\arrow["\dashv"{anchor=center, rotate=-90}, shift left=5, draw=none, from=1, to=0]
	\arrow["\dashv"{anchor=center, rotate=-90}, shift left=5, draw=none, from=3, to=2]
\end{tikzcd}.
\end{equation}

Where $i_*$ is a fully faithful embedding of the $\Ker(j^*)$ and $j_!$ is fully faithful.
Moreover let $i_*$, $j^*$ be exact, commute with filtered colimits and preserve finitely presented objects, and $i^*$ is exact.

Then, if $\AA$ and $\CC$ are locally coherent, $\BB$ is locally coherent.
\end{theorem}

\begin{proof}
The proof is much like the first part of the proof of Theorem (\ref{th_coh_B_straightforward}).
\end{proof}




\subsubsection{Local coherency of trivial extension}

The goal of this subsection is to recall the result of Fossum, Griffith, and Reiten~\cite{reiten} on the coherence of the trivial extension (in the sense of Remark~(\ref{rem_triv_ext})) of an abelian category. We first recall the relevant definitions.

 \begin{definition}\cite{reiten}
Let $\mathcal{A}$ be an abelian category endowed with a small full additive subcategory $\mathcal{P}$ consisting of projective objects, and let $A \in \mathcal{A}$.
\begin{enumerate}
\item[(a)] We say $A$ has \textit{finite $\mathcal{P}$-type} when some $P \in \mathcal{P}$ admits an epimorphism $P \to A$.
\item[(b)] We say $A$ has \textit{finite $\mathcal{P}$-presentation} when there is an exact sequence $Q \to P \to A \to 0$ with $P, Q \in \mathcal{P}$.
\item[(c)] We call $A$ \textit{$\mathcal{P}$-coherent} when it has finite $\mathcal{P}$-type and, moreover, each of its subobjects of finite $\mathcal{P}$-type has finite $\mathcal{P}$-presentation.
\item[(d)] We call $A$ \textit{pseudo $\mathcal{P}$-coherent} when every subobject of finite $\mathcal{P}$-type is coherent.
\end{enumerate}
\end{definition}

Write $\Coh_{\mathcal{P}}(\mathcal{A})$ for the full additive subcategory of $\mathcal{P}$-coherent objects of $\mathcal{A}$; it is abelian.

Define a functor $h: \mathcal{A} \to \Fun(\mathcal{P}^{\mathrm{op}}, \mathbf{Ab})$ by $h_A = \mathcal{A}(-, A)$ for $A \in \mathcal{A}$. Projectivity of the objects of $\mathcal{P}$ makes $h$ exact.

Once again Yoneda's lemma lets $h$ match the objects of finite $\mathcal{P}$-presentation in $\mathcal{A}$ with those of finite $\mathcal{P}$-presentation in $\Fun(\mathcal{P}^{\mathrm{op}}, \mathbf{Ab})$. Indeed, if $P \to Q \to A \to 0$ is exact in $\mathcal{A}$, then $h_P \to h_Q \to h_A \to 0$ is exact in $\Fun(\mathcal{P}^{\mathrm{op}}, \mathbf{Ab})$, so $h_A$ has finite $\mathcal{P}$-presentation. In the reverse direction, for $P, Q \in \mathcal{P}$ the identity $[h_P, h_Q] = \mathcal{A}(P, Q)$ shows that the cokernel of a map $h_P \to h_Q$ is represented by the cokernel of the underlying $P \to Q$ in $\mathcal{A}$. Moreover $A$ is $\mathcal{P}$-coherent precisely when $h_A$ is.

It follows that the categories $\Coh_{\mathcal{P}}(\mathcal{A})$ and $\Coh_{\mathcal{P}}(\Fun(\mathcal{P}^{\mathrm{op}}, \mathbf{Ab}))$ are equivalent.

\begin{definition}
We call $\mathcal{P}$ {\it left coherent} when $\mathcal{P} \subset \Coh_{\mathcal{P}}(\Fun(\mathcal{P}^{\mathrm{op}}, \mathbf{Ab}))$; dually, $\mathcal{P}$ is {\it right coherent} when $\mathcal{P}^{\mathrm{op}} \subset \Coh_{\mathcal{P}^{\mathrm{op}}}(\Fun(\mathcal{P}, \mathbf{Ab}))$, and {\it coherent} when it is at once left and right coherent.
\end{definition}

\begin{theorem}\cite[Theorem 2.2]{reiten} \label{th_reiten_TP_coherent}
Let $P$ be a small additive full subcategory of projectives in $\mathcal{A}$ and let $F: \mathcal{A} \to \mathcal{A}$ be right exact. Denote by $T(P)$ the small additive full subcategory obtained by extending $P$ to the comma category $(F \downarrow \Id_{\mathcal{A}})$. Then $T(P)$ is coherent in $(F \downarrow \Id_{\mathcal{A}})$ if and only if:
\begin{enumerate}
\item[(a)] $P$ is coherent in $\mathcal{A}$, and, for every $P$-coherent object $A$ in $\mathcal{A}$,
\item[(b)] the objects $L_i F(A)$ are $P$-coherent in each degree $i > 0$, and
\item[(c)] whenever $B$ is a subobject of $F(A)$ of finite $P$-type, $B$ is coherent and $F(B)$ has finite $P$-type.
\end{enumerate}
\end{theorem}

As a corollary of this theorem one gets a result of 

\begin{corollary} \cite{reiten} \label{cor_reiten_matrix}
The ring $\begin{pmatrix} R & N \\ M & S \end{pmatrix}$ is left coherent precisely when:
\begin{enumerate}
\item[(a)] $R$ and $S$ are both left coherent, and, for every finitely presented left $R$-module $A$ and finitely presented left $S$-module $B$,
\item[(b)] the left modules $\Tor_i^R( {}_S M, A)$ and $\Tor_i^S( {}_R N, B)$ are finitely presented in each degree $i > 0$, and
\item[(c)] whenever $C$ is a finitely generated left $S$-submodule of $M \otimes_R A$ (respectively a finitely generated left $R$-submodule of $N \otimes_S B$), then $C$ is finitely presented and $N \otimes_S C$ (resp.\ $M \otimes_R C$) stays finitely generated.
\end{enumerate}
\end{corollary}

\begin{example} \cite{reiten} \label{ex_reiten_kxi}
To illustrate the corollary, let $k$ be a field and write $A$ for the coherent $k$-algebra $k[x_i \mid i \in S]$ on an infinite index set $S$.
Regard $k$ as an $A$-module through the standard augmentation $A \to k$ that annihilates each $x_i$.
Then the ring $T = \begin{pmatrix} A & 0 \\ k & k \end{pmatrix}$ is left coherent but fails to be right coherent.
\end{example}

Note that in the previous example the commutative ring $A$ is coherent as a polynomial ring in infinitely many variables over the Noetherian ring $k$. The category $\Mod - T$ is still an extension of the categories $\Mod - A$ and $\Mod - k$ as in the recollement of Example~(\ref{ex_psa_matrix_ring}).

\begin{remark} \label{rem_generality}
Note that in Theorem~(\ref{th_coh_B_i*exact}) all the assumptions on the functors $i_*$ and $j^*$
\begin{equation}
\begin{tikzcd}
	\AA &&& \BB &&& \CC
	\arrow[""{name=0, anchor=center, inner sep=0}, "{{i_*}}"{description}, from=1-1, to=1-4]
	\arrow[""{name=1, anchor=center, inner sep=0}, "{{i^*}}"{description}, shift right=5, from=1-4, to=1-1]
	\arrow[""{name=2, anchor=center, inner sep=0}, "{{j^*}}"{description}, from=1-4, to=1-7]
	\arrow[""{name=3, anchor=center, inner sep=0}, "{{j_!}}"{description}, shift right=5, from=1-7, to=1-4]
	\arrow["\dashv"{anchor=center, rotate=-90}, shift left=5, draw=none, from=1, to=0]
	\arrow["\dashv"{anchor=center, rotate=-90}, shift left=5, draw=none, from=3, to=2]
\end{tikzcd}
\end{equation}
are fulfilled automatically in any recollement situation (i.e.\ when $i_*$ and $j^*$ possess suitable right adjoints), except for the preservation of finitely presented objects. The previous example shows that this assumption cannot be omitted.
\end{remark}

\subsection{Reverse course} \label{ssec:reverse}

Let $\BB$, $\CC$ be AB5 locally finitely presentable abelian categories.

As is known from Proposition (\ref{prop_loccoh=cohass}), to establish the local coherence of an AB5 locally finitely presentable category it is enough to show that it is coherently assembled.
For $\CC$, this amounts to checking that $\tYY_\CC$ is left exact.

\begin{theorem}\label{th_revcourse_coh_j^*j_*}
Consider the pair of adjoint functors $(j^* \dashv j_*)$
\begin{equation}
\begin{tikzcd}
	\BB &&& \CC
	\arrow[""{name=0, anchor=center, inner sep=0}, "{j^*}"{description}, from=1-1, to=1-4]
	\arrow[""{name=1, anchor=center, inner sep=0}, "{j_*}"{description}, shift left=5, from=1-4, to=1-1]
	\arrow["\dashv"{anchor=center, rotate=-90}, shift left=5, draw=none, from=0, to=1]
\end{tikzcd},
\end{equation}
such that $j_*$ is fully faithful and preserves filtered colimits, $j^*$ is exact.

Then local coherency of $\BB$ implies local coherency of $\CC$.
\end{theorem}

\begin{proof}
Note that, by full faithfulness,
\begin{equation}
\Id_\CC \cong j^*j_*.
\end{equation}
So now from Lemma (\ref{lem_commutativity_of_ind}) follows
\begin{equation}
\tYY_\CC \cong \tYY_\CC j^* j_* \cong Ij^* \tYY_\BB j_*
\end{equation}
Since the right-hand side is a composition of left exact functors, $\tYY_\CC$ is left exact, as desired.
\end{proof}

\begin{theorem}\label{th_revcourse_coh_j_!j^*}
Consider the pair of adjoint functors $(j_! \dashv j^*)$
\begin{equation}
\begin{tikzcd}
	\BB &&& \CC
	\arrow[""{name=0, anchor=center, inner sep=0}, "{{{j^*}}}"{description}, from=1-1, to=1-4]
	\arrow[""{name=1, anchor=center, inner sep=0}, "{{j_!}}"{description}, shift right=5, from=1-4, to=1-1]
	\arrow["\dashv"{anchor=center, rotate=-90}, shift left=5, draw=none, from=1, to=0]
\end{tikzcd},
\end{equation}
such that $j_!$ is fully faithful and exact, $j^*$ preserves filtered colimits.

Then local coherency of $\BB$ implies local coherency of $\CC$.
\end{theorem}

\begin{proof}
Note that, by full faithfulness,
\begin{equation}
\Id_\CC \cong j^*j_!.
\end{equation}
So now from Lemma (\ref{lem_commutativity_of_ind(spec)}) follows
\begin{equation}
\tYY_\CC \cong \tYY_\CC j^* j_!  \cong Ij^* \tYY_\BB j_!
\end{equation}
Since the right-hand side is a composition of left exact functors, $\tYY_\CC$ is left exact, as desired.
\end{proof}

Restriction of local coherence from $\BB$ to $\AA$ is symmetric to the previous case.

\begin{remark} \label{rem_krause_comparison}
Theorems~\ref{th_revcourse_coh_j^*j_*} and~\ref{th_revcourse_coh_j_!j^*} should be compared with a theorem of Krause \cite[Theorem 2.6]{KRAUSE_spec_loc_coh_cat}: if $\BB$ is locally coherent and $\AA \subseteq \BB$ is a localizing Serre subcategory of finite type, i.e.\ the section functor $s \colon \BB/\AA \to \BB$ (equivalently, by \cite[Lemma 2.4]{KRAUSE_spec_loc_coh_cat}, the right adjoint $t \colon \BB \to \AA$ of the inclusion) commutes with filtered colimits, then $\AA$ and $\BB/\AA$ are locally coherent and $\BB^{f.p.}/\AA^{f.p.} \simeq (\BB/\AA)^{f.p.}$.

In Theorem~\ref{th_revcourse_coh_j^*j_*} the exact functor $j^*$ with fully faithful right adjoint is the Gabriel quotient of $\BB$ by $\Ker(j^*)$, $j_* = s$, and the hypothesis on $j_*$ says that $\Ker(j^*)$ is of finite type; so this theorem is the statement of \cite[Theorem 2.6]{KRAUSE_spec_loc_coh_cat} about $\BB/\AA$, proved through $\tYY$ instead of finitely presented objects, which is why it extends to the compactly assembled setting (Proposition~\ref{prop_efimov_retraction}). Theorem~\ref{th_revcourse_coh_j_!j^*}, applied to $(i_* \dashv i^!)$, recovers the statement about $\AA$ without requiring the image of $j_!$ to be closed under extensions. We do not recover the finer statement on finitely presented objects, nor the converse \cite[Theorem 2.8, Corollaries 2.9, 2.10]{KRAUSE_spec_loc_coh_cat} describing all finite type localizing subcategories of a locally coherent category as closures $\overrightarrow{\mathcal{S}}$ of Serre subcategories $\mathcal{S} \subseteq \BB^{f.p.}$. Note also that for locally coherent $\BB$ the hypotheses ``$i^!$ preserves filtered colimits'' and ``$j_*$ preserves filtered colimits'' are equivalent by \cite[Lemma 2.4]{KRAUSE_spec_loc_coh_cat}.
\end{remark}

Note that the requirement of the functors in the previous theorems to be adjoint is not essential. One simply needs the appropriate retraction of the larger category onto the smaller one. Namely
\begin{proposition}\cite[Proposition 1.46.]{efimov_Th_of_heart} \label{prop_efimov_retraction}
Let $\mathcal{A}$ be a Grothendieck abelian category and let $\mathcal{B}$ be a coherently assembled abelian category. Suppose that we have a retraction $\mathcal{A} \xrightarrow{F} \mathcal{B} \xrightarrow{G} \mathcal{A}$, such that both $F$ and $G$ are left exact and commute with filtered colimits. Then $\mathcal{A}$ is coherently assembled.
\end{proposition}

\subsection{Quasiprojectivity condition} \label{ssec:qproj}

All theorems in the previous subsection give a one-sided implication: from the local coherence of the larger category $\BB$ we deduce the local coherence of the smaller category $\CC$.

These are far from being criteria. The results of~\cite{PARDO_HERNANDEZ} on the coherence of endomorphism rings of quasiprojective modules motivate us to strengthen the assumptions in Theorem~(\ref{th_revcourse_coh_j_!j^*}) so that it becomes a criterion.

We begin with a review of the original results of~\cite{PARDO_HERNANDEZ}.

Let $R$ be a ring and $M,\, N$ be left $R$-modules.

\begin{definition}
A module $N$ is called \textit{$M$-projective} if for every quotient module $L$ of $M$, the canonical homomorphism 
\[
\Hom_R(N, M) \to \Hom_R(N, L)
\] 
is surjective.

A module $M$ is called \textit{quasiprojective} if it is $M$-projective, i.e., for every quotient $L$ of $M$, the map
\[
\Hom_R(M, M) \to \Hom_R(M, L)
\]
is surjective.

A module $M$ is called \textit{$\Sigma$-quasiprojective} if every direct sum $M^{(I)}$ is quasiprojective.
\end{definition}

For the category of modules over the endomorphism ring of a quasiprojective module, the following holds.

Let $M$ be a quasiprojective left $R$-module, and let $S = \End(M)$.

\begin{equation}
\begin{tikzcd}
	{R-\Mod} &&& {S-\Mod}
	\arrow[""{name=0, anchor=center, inner sep=0}, "{j^* = \Hom(M,\,\_)}"', from=1-1, to=1-4]
	\arrow[""{name=1, anchor=center, inner sep=0}, "{j_! = M\bigotimes\limits_ S\_}"', shift right=5, from=1-4, to=1-1]
	\arrow["\dashv"{anchor=center, rotate=-90}, draw=none, from=1, to=0]
\end{tikzcd}
\end{equation}

Let $\eta \colon \Id_S \to j^* j_!$ be the unit of the adjunction.

\begin{lemma} \cite[Lemma 1.1]{PARDO_HERNANDEZ}
 If $X$ is either a finitely presented left $S$-module or a finitely generated submodule of a free left $S$-module, then $\eta_X$ is an isomorphism.
\end{lemma}

In particular, $j_!$ is fully faithful on the full subcategory of finitely presented left $S$-modules. This property, together with the quasiprojectivity of $M$ expressed through the vanishing of $R^1 j^* = \Ext^1_R(M, -)$ on the relevant submodules, is what is retained in the general Definition~\ref{def_qproj_functors} below.

\begin{theorem} \cite[Theorem 1.2]{PARDO_HERNANDEZ}
Let \( M \) be a quasiprojective module and let \( S = \End(M) \), the following are equivalent:
\begin{enumerate}
    \item \( S \) is left coherent
    \item For every homomorphism \( f: M^k \to M^r \), 
there exist \( n \geq 0 \) and a homomorphism \( g: M^n \to \ker f \) such that 
\( \Hom_R(M, \coker g) = 0 \).
\end{enumerate}
\end{theorem}

We now reformulate this setting for arbitrary Grothendieck abelian categories.

\begin{definition} \label{def_qproj_functors}
Consider the pair of adjoint functors $(j_! \dashv j^*)$ 
\begin{equation}
\begin{tikzcd}
	\BB &&& \CC
	\arrow[""{name=0, anchor=center, inner sep=0}, "{j^*}"', from=1-1, to=1-4]
	\arrow[""{name=1, anchor=center, inner sep=0}, "{j_!}"', shift right=5, from=1-4, to=1-1]
	\arrow["\dashv"{anchor=center, rotate=-90}, draw=none, from=1, to=0]
\end{tikzcd}
\end{equation}
between Grothendieck abelian categories.

Denote the image of $\CC^{f.p.}$ under $j_!$ by $\BB_{\CC^{f.p.}}$. We call the pair $(j_! \dashv j^*)$ {\it quasiprojective} if
\begin{enumerate}
    \item $j_!$ is fully faithful on $\CC^{f.p.}$;
    \item the right derived functor satisfies $R^1 j^*(X) = 0$ for every subobject $X$ of an object of $\BB_{\CC^{f.p.}}$.
\end{enumerate}
\end{definition}

\begin{theorem} \label{th_quasiproj_criterion}
Let $(j_! \dashv j^*)$ be a pair of quasiprojective adjoint functors as in Definition~(\ref{def_qproj_functors}). The following are equivalent:
\begin{enumerate}
    \item $\CC$ is locally coherent;
    \item for every morphism $f \colon B_1 \to B_2$ in $\BB_{\CC^{f.p.}}$, there exists a morphism $g \colon B_3 \to \ker f$ with $B_3 \in \BB_{\CC^{f.p.}}$ such that $j^*(\coker g) = 0$ and $j^*(\ker g) = 0$.
\end{enumerate}
\end{theorem}

The core idea underlying this proof is, in large part, indebted to \cite{PARDO_HERNANDEZ}.

\begin{proof}
$(1 \implies 2)$:

Suppose that $\CC$ is locally coherent. Consider the kernel of a morphism as in (2):
\begin{equation}
0 \to K \to j_!X_1 \to j_!X_2
\end{equation}
Applying the left exact functor $j^*$ we obtain an exact sequence in $\CC$:
\begin{equation}
0 \to j^*K \to X_1 \to X_2.
\end{equation}

Hence $j^*K \in \CC^{f.p.}$, by local coherence of $\CC$.

Consider
\begin{equation}
0 \to \ker(g) \to j_!j^*K \xrightarrow{g := \epsilon_K} K \to \coker(g) \to 0.
\end{equation}

It remains to check that $j^*(\coker g) = 0$ and $j^*(\ker g) = 0$. Set $X = \im(g)$:
\begin{equation}
0 \to \ker(g) \to j_!j^*K \to X \to 0
\end{equation}
\begin{equation}
0 \to X \to K \to \coker(g) \to 0.
\end{equation}

Applying $j^*$ to these short exact sequences we obtain

\begin{equation}
0 \to j^*\ker(g) \to j^*K \to j^*X \to R^1j^*\ker(g) \to \dots
\end{equation}
\begin{equation}
0 \to j^*X \to j^*K \to j^*\coker(g) \to R^1j^*X \to \dots.
\end{equation}

Since $\ker(g) \hookrightarrow j_! j^*K$ and $X \hookrightarrow K \hookrightarrow j_! X_1$ are subobjects, and $j^*K, X_1 \in \BB_{\CC^{f.p.}}$, quasiprojectivity implies
that $R^1 j^*\ker(g) = 0$ and $R^1 j^*X = 0$. Hence the above exact sequences take the form

\begin{equation}
0 \to j^*\ker(g) \to j^*K \to j^*X \to 0
\end{equation}
\begin{equation}
0 \to j^*X \to j^*K \to j^*\coker(g) \to 0.
\end{equation}

Note that $j^*(g) = j^*\epsilon_K = (\eta_{j^*K})^{-1}$ is an isomorphism by quasiprojectivity, since $j^*K \in \CC^{f.p.}$. Therefore the maps $j^*K \to j^*X$ and $j^*X \to j^*K$ in the exact sequences above are mono and epi, respectively. This completes the implication $(1 \Rightarrow 2)$.

$(2 \implies 1)$:

Consider the kernel of a morphism
\begin{equation}
0 \to \ker(h) \to X_1 \xrightarrow{h} X_2
\end{equation}
between $X_1, X_2 \in \CC^{f.p.}$. It suffices to show that $\ker(h) \in \CC^{f.p.}$.

Consider the kernel of $j_!(h)$:
\begin{equation}
0 \to N \to j_!X_1 \xrightarrow{j_!(h)} j_!X_2.
\end{equation}
By the assumption in~(2), there exist $X_3 \in \CC^{f.p.}$ and a morphism
\begin{equation}
0 \to \ker(p) \to j_!X_3 \xrightarrow{p} N \to \coker(p) \to 0
\end{equation}
such that $j^*(\ker p) = j^*(\coker p) = 0$.

Let $Y$ be the image of $p$:
\begin{equation}
0 \to \ker(p) \to j_!X_3 \to Y \to 0
\end{equation}
\begin{equation}
0 \to Y \to N \to \coker(p) \to 0
\end{equation}

Applying $j^*$ to these short exact sequences we obtain

\begin{equation}
0 \to j^*\ker(p) \to j^*j_!X_3 \to j^*Y \to R^1j^*(\ker(p))
\end{equation}
\begin{equation}
0 \to j^*Y \to j^*N \to j^*\coker(p) \to \dots.
\end{equation}

Quasiprojectivity implies $R^1 j^*(\ker p) = 0$, since $\ker p$ is a subobject of an object of $\BB_{\CC^{f.p.}}$: $\ker p \hookrightarrow j_! X_3$. Therefore the above exact sequences yield

\begin{equation}
X_3 \cong j^*j_!X_3 \cong j^*Y \cong j^*N.
\end{equation}
Hence $j^*N \in \CC^{f.p.}$. It remains to show that $\ker(h) \cong j^*N$. Indeed:

\begin{equation}
\begin{tikzcd}[sep=small]
	0 &&& {\ker(h)} &&& {X_1} &&& {X_2} \\
	\\
	\\
	0 &&& {j^*N} &&& {j^*j_!X_1} &&& {j^*j_!X_2}
	\arrow[from=1-1, to=1-4]
	\arrow[from=1-4, to=1-7]
	\arrow["h", from=1-7, to=1-10]
	\arrow["\cong", from=1-7, to=4-7]
	\arrow["\cong", from=1-10, to=4-10]
	\arrow[from=4-1, to=4-4]
	\arrow[from=4-4, to=4-7]
	\arrow["{j^*j_!(h)}", from=4-7, to=4-10]
\end{tikzcd}.
\end{equation}
\end{proof}

\subsection{FP-injective objects} \label{ssec:fpinj}

\begin{definition}
Let $\AA$ be an $\omega$-accessible abelian category. An object $M \in \AA$ is called {\it FP-injective} if $\Ext^1(F,M) = 0$ for any finitely presented object $F$.
\end{definition}

Taking $\AA = R - \Mod$ in the previous definition recovers the original definition of~\cite{Stenstrom-1970}.

\begin{proposition}\cite{Stenstrom-1970}
Let $\{X_i\}_{i \in I}$ be a filtered system of FP-injective modules with all transition maps monomorphisms. Then $\indlim X_i$ is FP-injective.
\end{proposition}

\begin{theorem}\cite[Theorem 3.2.]{Stenstrom-1970}
The following assertions are equivalent for the ring $R$:
\begin{itemize}
    \item[(a)] $R$ is right coherent.
    \item[(b)] $\xi_1 : \varinjlim \Ext^1(F, M_i) \to \Ext^1(F, \varinjlim M_i)$ is an isomorphism for every finitely presented module $F$ and direct system $(M_i)$.
    \item[(c)] $\xi_n : \varinjlim \Ext^n(F, M_i) \to \Ext^n(F, \varinjlim M_i)$ are isomorphisms for all $n \geqslant 0$, for every finitely presented right module $F$ and direct system $(M_i)_I$.
    \item[(d)] Every direct limit of FP-injective right modules is FP-injective.
\end{itemize}
\end{theorem}

The preceding result admits an immediate generalisation to AB5 locally finitely presentable abelian categories.

\begin{theorem}\label{th_fpinj_ab}
Let $\AA$ be an AB5 locally finitely presentable abelian category. The following are equivalent:
\begin{itemize}
    \item[(a)] $\AA$ is locally coherent.
    \item[(b)] $\xi_1 \colon \varinjlim_i \Ext^1(F, M_i) \to \Ext^1(F, \varinjlim_i M_i)$ is an isomorphism for every finitely presented object $F$ and every direct system $(M_i)$.
    \item[(c)] $\xi_n \colon \varinjlim_i \Ext^n(F, M_i) \to \Ext^n(F, \varinjlim_i M_i)$ is an isomorphism for every $n \geqslant 0$, every finitely presented object $F$, and every direct system $(M_i)$.
    \item[(d)] Every filtered colimit of FP-injective objects is FP-injective.
\end{itemize}
\end{theorem}

\begin{remark}
Every Grothendieck abelian category, and hence every AB5 locally finitely presentable abelian category, has enough injective objects.
\end{remark}

\begin{theorem}
Let $\AA$, $\BB$, $\CC$ be AB5 locally finitely presentable abelian categories and the following adjoint 
functors between them
\begin{equation}
\begin{tikzcd}
	\AA &&& \BB &&& \CC
	\arrow[""{name=0, anchor=center, inner sep=0}, "{i_*}"{description}, from=1-1, to=1-4]
	\arrow[""{name=1, anchor=center, inner sep=0}, "{i^!}"{description}, shift left=5, from=1-4, to=1-1]
	\arrow[""{name=2, anchor=center, inner sep=0}, "{j^*}"{description}, from=1-4, to=1-7]
	\arrow[""{name=3, anchor=center, inner sep=0}, "{j_*}"{description}, shift left=5, from=1-7, to=1-4]
	\arrow["\dashv"{anchor=center, rotate=-90}, shift left=5, draw=none, from=0, to=1]
	\arrow["\dashv"{anchor=center, rotate=-90}, shift left=5, draw=none, from=2, to=3]
\end{tikzcd}
\end{equation}

Suppose that $i_*$ is the fully faithful embedding of $\Ker(j^*)$, that $j_*$ is fully faithful, that $i_*$ is exact, that $j_*$ preserves filtered colimits, that $R^1 i^!(X) = 0$ for every FP-injective $X \in \BB$, and that all four functors preserve FP-injective objects.

Then, if $\AA$ and $\CC$ are locally coherent, $\BB$ is locally coherent.
\end{theorem}

\begin{proof}
Let $\{E_i\}_{i \in I}$ be a filtered system of FP-injective objects in $\BB$. By Proposition~(\ref{prop_pirash_recol_sprava_iRi}), for every $i \in I$ we have a decomposition
\begin{equation}
0 \to i_*i^! E_i \to E_i \to j_*j^*E_i \to 0.
\end{equation}
Since $j_*$ and $i_*$ preserve filtered colimits, we obtain
\begin{equation}
0 \to i_* \indlim i^! E_i \to \indlim E_i \to j_* \indlim j^*E_i \to 0.
\end{equation}

The left and right terms are FP-injective by the local coherence assumption; the middle term is then FP-injective as an extension. By Theorem~(\ref{th_fpinj_ab}), $\BB$ is locally coherent, as desired.
\end{proof}

\subsection{CM-injective objects} \label{ssec:cminj}

Theorem~\ref{th_fpinj_ab} characterises local coherence of an AB5 locally finitely presentable abelian category by the behaviour of FP-injective objects under filtered colimits. Both sides of this characterisation have natural counterparts in the compactly assembled world: local coherence is replaced by coherent assembly (Definition~\ref{def_compass_cohass}, Proposition~\ref{prop_loccoh=cohass}), and the role of finitely presented objects is taken over by compact morphisms (Definition~\ref{def_compact_morphism}). The purpose of this subsection is to show that, once FP-injectivity is generalised accordingly, the equivalence (a)$\iff$(d) of Theorem~\ref{th_fpinj_ab} survives for arbitrary compactly assembled Grothendieck abelian categories (Theorem~\ref{th_cminj}). The naive generalisation does not work: a compactly assembled category may have no non-zero compact objects at all (for instance sheaves of vector spaces on the real line, \cite[\S 8]{efimov_Th_of_heart}), so that the condition ``$\Ext^1(F, M) = 0$ for every compact object $F$'' is vacuous. Instead we ask that extensions be killed by compact \emph{morphisms}; objects with this property are called \emph{CM-injective} (Definition~\ref{def_cminj}), and in the locally finitely presentable case they are exactly the FP-injective objects (Proposition~\ref{prop_cminj_vs_fpinj}). The direct implication is then proved by transporting the problem to $\Ind(\AA)$, which is always ``formally locally coherent'' (Lemma~\ref{lem_cm_ext_ind}); the converse is a categorical version of the argument of Stenstr{\"o}m for the implication \textbf{9 $\Rightarrow$ 1} of Theorem~\ref{th_MEGA}. The authors learned about the possibility of such a generalisation from Alexander Efimov in a private conversation. Thus the main goal of this subsection is to prove the following criterion.

\begin{theorem} \label{th_cminj}
Let $\AA$ be a compactly assembled Grothendieck abelian category. The following are equivalent:
\begin{itemize}
    \item[(a)] $\AA$ is coherently assembled.
    \item[(b)] For every compact morphism $\psi \colon w \to z$, every filtered system $\{x_i\}_{i \in I}$ in $\AA$ with colimit $x$, and every $\xi \in \Ext^1(z, x)$, the class $\psi^* \xi$ lies in the image of $\varinjlim_i \Ext^1(w, x_i) \to \Ext^1(w, x)$.
    \item[(c)] Every filtered colimit of CM-injective objects is CM-injective.
\end{itemize}
\end{theorem}

Throughout this subsection $\AA$ is a compactly assembled Grothendieck abelian category; by \cite[Remark 1.42]{efimov_Th_of_heart} this is the same as a Grothendieck abelian category satisfying (AB6). We freely use the notation and the standard facts collected in Remarks~\ref{rem_ind_basic}, \ref{rem_tYY_basic} and \ref{rem_compact_lifts}; in particular $\colim \colon \Ind(\AA) \to \AA$ is exact, $\tYY$ is right exact and preserves filtered colimits, and $\AA$ is coherently assembled if and only if $\tYY$ is exact, if and only if $\tYY$ preserves monomorphisms.

\begin{definition} \label{def_cminj}
An object $x \in \AA$ is called \emph{CM-injective} (injective with respect to compact morphisms) if for every compact morphism $\phi \colon y \to z$ in $\AA$ the induced map
\[
\phi^* \colon \Ext^1_{\AA}(z, x) \longrightarrow \Ext^1_{\AA}(y, x)
\]
is zero. Equivalently, every short exact sequence $0 \to x \to e \to z \to 0$ splits after pull-back along any compact morphism $y \to z$.
\end{definition}

Injective objects are CM-injective, and CM-injective objects are closed under products and retracts. In a locally finitely presentable category the notion reduces to the familiar one.

\begin{proposition} \label{prop_cminj_vs_fpinj}
Let $\AA$ be an AB5 locally finitely presentable abelian category, compactly assembled by Proposition~\ref{prop_omega_acc_compass}. A morphism $\phi \colon y \to z$ of $\AA$ is compact if and only if it factors through a finitely presented object. Consequently, an object of $\AA$ is CM-injective if and only if it is FP-injective.
\end{proposition}

\begin{proof}
By Proposition~\ref{prop_omega_acc_compass}, $\tYY z = \{z_j\}_{j \in J}$ is a filtered system of finitely presented objects with $\indlim z_j = z$; we use the description of compact morphisms of Proposition~\ref{prop_compact_via_tYY}. Let $\phi$ be compact with lift $\hat{\phi} \colon \YY y \to \tYY z = \colim[j] \YY z_j$. Since $\YY y$ is compact in $\Ind(\AA)$, $\hat{\phi} = \mathrm{in}_{j_0} \circ \YY a$ for some $j_0$ and $a \colon y \to z_{j_0}$; applying $\rho_z$ we get $\YY \phi = \YY(\sigma_{j_0} a)$, so $\phi = \sigma_{j_0} \circ a$ factors through the finitely presented object $z_{j_0}$. Conversely, let $\phi = b \circ a$ with $a \colon y \to w$, $b \colon w \to z$ and $w$ finitely presented. As $z = \indlim z_j$, the morphism $b$ factors as $b = \sigma_{j_0} \circ c$ for some $c \colon w \to z_{j_0}$, and $\mathrm{in}_{j_0} \circ \YY(c a)$ is a lift of $\phi$.

If $x$ is FP-injective and $\phi = b \circ a$ as above, then $\phi^* = a^* \circ b^*$ factors through $\Ext^1(w, x) = 0$. If $x$ is CM-injective and $w$ is finitely presented, then $\Id_w$ is compact, so $\Id_w^* = \Id$ is zero on $\Ext^1(w, x)$, i.e.\ $\Ext^1(w, x) = 0$.
\end{proof}

\begin{lemma}[{\cite[Corollary 1.41]{efimov_Kth_large_cat}, \cite[\S 1.6]{efimov_Th_of_heart}}] \label{lem_cm_factor}
Let $\CC$ be a compactly assembled category (not necessarily abelian). Every compact morphism $\phi \colon y \to z$ in $\CC$ factors as $\phi = \psi \circ \chi$, where $\chi \colon y \to w$ and $\psi \colon w \to z$ are compact. Consequently a compact morphism is a composition of any finite number of compact morphisms.
\end{lemma}

The proof in \cite[Corollary 1.41]{efimov_Kth_large_cat} is written for dualizable categories, but, as remarked in \cite[\S 1.6]{efimov_Th_of_heart}, it uses only compact assembly: one writes $\tYY z = \colim[j] \tYY z_j$ and factors a lift of $\phi$ through a stage using the compactness of $\YY y$ in $\Ind(\CC)$.

\begin{remark} \label{rem_cminj_ext_closed}
It follows that CM-injective objects are closed under extensions. Indeed, let $0 \to x' \to x \to x'' \to 0$ be exact with $x'$, $x''$ CM-injective, let $\xi \in \Ext^1(z, x)$ and let $\phi = \psi \chi$ be a compact morphism factored as in Lemma~\ref{lem_cm_factor}. The image of $\psi^* \xi$ in $\Ext^1(w, x'')$ vanishes, so $\psi^* \xi$ comes from $\Ext^1(w, x')$, and hence $\phi^* \xi = \chi^* \psi^* \xi$ lies in the image of $\chi^* \colon \Ext^1(w, x') \to \Ext^1(y, x')$, which is zero.
\end{remark}

\medskip
The finitely presented objects of $\Ind(\AA)$ are the retracts of the objects $\YY a$, $a \in \AA$; since $\AA$ is abelian, kernels of morphisms between them stay finitely presented. This is the reason why the analogue of condition (b) of Theorem~\ref{th_fpinj_ab} holds in $\Ind(\AA)$ \emph{unconditionally}. We give a self-contained proof of the two facts we need.

\begin{lemma} \label{lem_cm_ext_yoneda}
For $a, b \in \AA$ the functor $\YY$ induces an isomorphism $\Ext^1_{\AA}(a, b) \xrightarrow{\sim} \Ext^1_{\Ind(\AA)}(\YY a, \YY b)$, natural in both arguments.
\end{lemma}

\begin{proof}
The map is well defined since $\YY$ is exact. It is injective: if the image of an extension $0 \to b \to e \to a \to 0$ splits by a section $s \colon \YY a \to \YY e$, then $s = \YY s_0$ by full faithfulness, and $s_0$ splits the original extension. It is surjective: let $0 \to \YY b \to M \to \YY a \to 0$ be exact in $\Ind(\AA)$. The functor $\YY \circ \colim$ is exact, so the row $0 \to \YY b \to \YY \colim M \to \YY a \to 0$ is exact, and the unit $\eta$ gives a morphism of short exact sequences from the former row to the latter which is an isomorphism at both ends ($\eta_{\YY a}$ is invertible for every $a$). By the five lemma $\eta_M \colon M \xrightarrow{\sim} \YY \colim M$, i.e.\ the given extension is the $\YY$-image of $0 \to b \to \colim M \to a \to 0$.
\end{proof}

\begin{lemma} \label{lem_cm_ext_ind}
Let $w \in \AA$ and let $\{M_t\}_{t \in T}$ be a filtered system in $\Ind(\AA)$ with colimit $M = \colim[t] M_t$. Then the canonical map
\[
\varinjlim_{t} \Ext^1_{\Ind(\AA)}(\YY w, M_t) \longrightarrow \Ext^1_{\Ind(\AA)}(\YY w, M)
\]
is an isomorphism.
\end{lemma}

\begin{proof}
\emph{Surjectivity.} Let $0 \to M \xrightarrow{m} E \xrightarrow{p} \YY w \to 0$ be exact in $\Ind(\AA)$. Write $E = \{e_k\}_{k \in K}$, so $E = \colim[k] \YY e_k$, and define $f_k \colon e_k \to w$ by $\YY f_k = p \circ \mathrm{in}_k$.

\emph{(i) Some $f_{k_0}$ is an epimorphism.} The morphism $p$ is the filtered colimit of the morphisms $p \circ \mathrm{in}_k \colon \YY e_k \to \YY w$ with constant target. Images commute with filtered colimits in $\Ind(\AA)$, and $\YY$ is exact, hence
\[
\YY w = \im p = \colim[k] \im(p \circ \mathrm{in}_k) = \colim[k] \YY(\im f_k).
\]
Since $\YY w$ is compact in $\Ind(\AA)$, the identity of $\YY w$ factors through some $\YY(\im f_{k_0}) \hookrightarrow \YY w$; a monomorphism admitting a right inverse is an isomorphism, so $\im f_{k_0} = w$.

\emph{(ii) $E$ is a push-out.} Put $\kappa := \ker f_{k_0} \in \AA$; by exactness of $\YY$ the row $0 \to \YY \kappa \to \YY e_{k_0} \to \YY w \to 0$ is exact. The composite $\mathrm{in}_{k_0} \circ \YY(\kappa \hookrightarrow e_{k_0})$ is killed by $p$, hence factors uniquely through $m$: we obtain $\gamma \colon \YY \kappa \to M$. Consider $v := (m, \mathrm{in}_{k_0}) \colon M \oplus \YY e_{k_0} \to E$. Its image contains $M = \ker p$ and maps onto $\YY w$ (because $p \circ \mathrm{in}_{k_0} = \YY f_{k_0}$ is an epimorphism), so $v$ is an epimorphism. Its kernel is identified with $\YY \kappa$, embedded via $(-\gamma, \YY(\kappa \hookrightarrow e_{k_0}))$: indeed, the second component of $\ker v \to \YY e_{k_0}$ is a monomorphism since $m$ is, it factors through $\ker(\YY f_{k_0}) = \YY \kappa$, and on $\YY \kappa$ the first component is forced to be $-\gamma$. Hence $E$ is the push-out of $\YY e_{k_0}$ and $M$ along $\YY \kappa$, and
\[
[E] = \gamma_* \bigl[\, 0 \to \YY \kappa \to \YY e_{k_0} \to \YY w \to 0 \,\bigr] \in \Ext^1_{\Ind(\AA)}(\YY w, M).
\]

\emph{(iii) Descent to a stage.} The object $\YY \kappa$ is compact in $\Ind(\AA)$ (here we use that the ``syzygy'' $\kappa$ lies again in $\AA$), hence $\gamma$ factors as $\gamma = (M_{t_0} \to M) \circ \gamma'$ for some $t_0 \in T$ and $\gamma' \colon \YY \kappa \to M_{t_0}$. Therefore $[E] = (M_{t_0} \to M)_* \, \gamma'_* [\YY e_{k_0}]$ lies in the image of $\Ext^1(\YY w, M_{t_0})$.

\emph{Injectivity.} Let $\zeta \in \Ext^1(\YY w, M_t)$ be represented by $E_t$ and suppose that its image in $\Ext^1(\YY w, M)$ vanishes. The push-outs $E_{t'} := M_{t'} \sqcup_{M_t} E_t$, $t' \geq t$, form a filtered system whose colimit is $E_M := M \sqcup_{M_t} E_t$, because filtered colimits commute with push-outs. A section $s \colon \YY w \to E_M$ factors through some $E_{t'}$ by compactness of $\YY w$, as $s = (E_{t'} \to E_M) \circ s'$; then $s'$ is a section of $E_{t'} \to \YY w$, so the image of $\zeta$ in $\Ext^1(\YY w, M_{t'})$ vanishes.
\end{proof}

\begin{corollary} \label{cor_cm_ext_ind}
Let $w \in \AA$ and let $\{x_i\}_{i \in I}$ be a filtered system in $\AA$ with formal colimit $\tilde X := \{x_i\}_{i \in I} \in \Ind(\AA)$. Then the canonical map
\[
\varinjlim_{i} \Ext^1_{\AA}(w, x_i) \xrightarrow{\ \sim\ } \Ext^1_{\Ind(\AA)}(\YY w, \tilde X), \qquad [\zeta] \mapsto (\mathrm{in}_i)_* \YY(\zeta),
\]
is an isomorphism.
\end{corollary}

\begin{proof}
Apply Lemma~\ref{lem_cm_ext_ind} to $M_t = \YY x_i$, using $\tilde X = \colim[i] \YY x_i$, and Lemma~\ref{lem_cm_ext_yoneda}.
\end{proof}

We now turn to the ingredients of the converse implication. The first lemma is the point where CM-injectivity is converted into a statement about $\tYY$.

\begin{lemma} \label{lem_cm_ses_cminj}
Let $x \in \AA$ be CM-injective and let $0 \to x \to e \to z \to 0$ be exact in $\AA$. Then $0 \to \tYY x \to \tYY e \to \tYY z \to 0$ is exact in $\Ind(\AA)$.
\end{lemma}

\begin{proof}
Let $\tYY z = \{z_j\}_{j \in J}$ and $e_j := e \times_z z_j$. Each row $0 \to x \to e_j \to z_j \to 0$ is exact, being the pull-back of the given sequence along $\sigma_j$, and it splits because $\sigma_j$ is compact and $x$ is CM-injective. Since $\AA$ is AB5, filtered colimits commute with pull-backs, so $\indlim[j] e_j = e \times_z \indlim[j] z_j = e$. The functor $\tYY$ is additive and is a left adjoint, so it preserves split exact sequences and filtered colimits; hence $0 \to \tYY x \to \tYY e \to \tYY z \to 0$ is the filtered colimit over $j$ of the split exact rows $0 \to \tYY x \to \tYY e_j \to \tYY z_j \to 0$, and filtered colimits in $\Ind(\AA)$ are exact.
\end{proof}

\begin{lemma} \label{lem_cm_ext_compare}
Let $M \in \Ind(\AA)$ and $m := \colim M$.
\begin{enumerate}
    \item[(i)] For every $x \in \AA$ the map $\gamma_x \colon \Ext^1_{\Ind(\AA)}(\tYY x, M) \to \Ext^1_{\AA}(x, m)$ induced by $\colim$ is injective.
    \item[(ii)] Let $\Sigma \colon 0 \to x \xrightarrow{\iota} e \xrightarrow{\pi} z \to 0$ be exact in $\AA$, let $K := \Ker(\tYY \pi)$ and let $q \colon \tYY x \to K$ be the epimorphism induced by $\tYY \iota$ (right exactness of $\tYY$). If $g \colon x \to m$ is such that $g_* [\Sigma] \in \Ext^1_{\AA}(z, m)$ lies in the image of $\gamma_z$, then there exists $h \colon K \to M$ with $\colim(h \circ q) = g$.
\end{enumerate}
\end{lemma}

\begin{proof}
(i) Let $0 \to M \to \mathcal{E} \xrightarrow{p} \tYY x \to 0$ represent a class with $\gamma_x$-image zero, i.e.\ the row $0 \to m \to \colim \mathcal{E} \xrightarrow{\colim p} x \to 0$ splits, by a section $s$ say. Let $\hat{s} \colon \tYY x \to \mathcal{E}$ correspond to $s$ under the $\tYY \dashv \colim$ adjunction. Then $\colim(p \hat{s}) = \colim(p) \circ s = \Id_x = \colim(\Id_{\tYY x})$, hence $p \hat{s} = \Id_{\tYY x}$, and the class is zero.

(ii) Since $\colim$ is exact, it takes the exact row $0 \to K \to \tYY e \to \tYY z \to 0$ to $\Sigma$, with $\colim q = \Id_x$. Consider the diagram with exact rows
\begin{equation}
\begin{tikzcd}[column sep=small]
	{\Hom(\tYY e, M)} & {\Hom(K, M)} & {\Ext^1(\tYY z, M)} & {\Ext^1(\tYY e, M)} \\
	{\Hom(e, m)} & {\Hom(x, m)} & {\Ext^1(z, m)} & {\Ext^1(e, m)}
	\arrow[from=1-1, to=1-2]
	\arrow["\cong"', from=1-1, to=2-1]
	\arrow["\partial", from=1-2, to=1-3]
	\arrow["r"', from=1-2, to=2-2]
	\arrow["{(\tYY \pi)^*}", from=1-3, to=1-4]
	\arrow["{\gamma_z}"', from=1-3, to=2-3]
	\arrow["{\gamma_e}", from=1-4, to=2-4]
	\arrow["{\iota^*}", from=2-1, to=2-2]
	\arrow["\delta", from=2-2, to=2-3]
	\arrow["{\pi^*}", from=2-3, to=2-4]
\end{tikzcd}
\end{equation}
in which the vertical maps are induced by $\colim$; explicitly $r(h) = \colim(h q)$. The diagram commutes because connecting homomorphisms are compatible with exact functors. The first vertical map is a bijection by the $\tYY \dashv \colim$ adjunction, $r$ is injective (if $\colim(hq) = 0$ then $hq = 0$, so $h = 0$ as $q$ is an epimorphism), and $\gamma_z$, $\gamma_e$ are injective by (i).

Let $\delta g = g_* [\Sigma] = \gamma_z \theta$. Then $\gamma_e ((\tYY \pi)^* \theta) = \pi^* \delta g = 0$, so $(\tYY \pi)^* \theta = 0$ and $\theta = \partial h_1$ for some $h_1 \colon K \to M$. Now $\delta (g - r(h_1)) = \gamma_z \theta - \gamma_z \partial h_1 = 0$, hence $g - r(h_1) = f \circ \iota$ for some $f \colon e \to m$. Let $\hat{f} \colon \tYY e \to M$ correspond to $f$ under the $\tYY \dashv \colim$ adjunction and let $h_2$ be its restriction to $K$; then $h_2 q = \hat{f} \circ \tYY \iota$ and $r(h_2) = f \iota$. Thus $h := h_1 + h_2$ satisfies $r(h) = g$.
\end{proof}

The following proposition is the only place where coherent assembly is used in the direct implication.

\begin{proposition} \label{prop_cm_descent}
Let $\AA$ be coherently assembled, let $\psi \colon w \to z$ be a compact morphism, let $\{x_i\}_{i \in I}$ be a filtered system in $\AA$ and $x = \indlim[i] x_i$. Then for every $\xi \in \Ext^1_{\AA}(z, x)$
\[
\psi^* \xi \in \Im \Bigl( \varinjlim_i \Ext^1_{\AA}(w, x_i) \longrightarrow \Ext^1_{\AA}(w, x) \Bigr),
\]
the map being the push-forward along the canonical morphisms $x_i \to x$.
\end{proposition}

\begin{proof}
Let $\xi$ be represented by $\Sigma \colon 0 \to x \to e \to z \to 0$.

\emph{Step 1.} Since $\tYY$ is exact, the row $\tYY \Sigma \colon 0 \to \tYY x \to \tYY e \to \tYY z \to 0$ is exact in $\Ind(\AA)$.

\emph{Step 2.} Choose a lift $\hat{\psi} \colon \YY w \to \tYY z$ and pull $\tYY \Sigma$ back along it; we obtain an exact row $0 \to \tYY x \to \tYY e \times_{\tYY z} \YY w \to \YY w \to 0$ representing $\hat{\psi}^* [\tYY \Sigma]$.

\emph{Step 3.} Let $\tilde X = \{x_i\}_{i \in I} \in \Ind(\AA)$, so $\colim \tilde X = x$, and let $\alpha \colon \tYY x \to \tilde X$ correspond to $\Id_x$ under the $\tYY \dashv \colim$ adjunction; thus $\colim \alpha = \Id_x$. Pushing out along $\alpha$ we get
\[
[E'] := \alpha_* \hat{\psi}^* [\tYY \Sigma] \in \Ext^1_{\Ind(\AA)}(\YY w, \tilde X).
\]

\emph{Step 4.} By Corollary~\ref{cor_cm_ext_ind} there are $i_0 \in I$ and $\zeta \in \Ext^1_{\AA}(w, x_{i_0})$ with $(\mathrm{in}_{i_0})_* \YY(\zeta) = [E']$.

\emph{Step 5.} Apply the exact functor $\colim$. It preserves extensions and commutes with pull-backs and push-outs, and under our identifications
\[
\colim \tYY \Sigma = \Sigma, \qquad \colim \hat{\psi} = \psi, \qquad \colim \alpha = \Id_x, \qquad \colim(\mathrm{in}_{i_0}) = (x_{i_0} \to x), \qquad \colim \YY(\zeta) = \zeta .
\]
Applying $\colim$ to the equality of Step 4 gives $(x_{i_0} \to x)_* \zeta = \colim [E'] = \psi^* \xi$.
\end{proof}

\begin{proposition} \label{prop_cm_reverse}
Let $\AA$ be a compactly assembled Grothendieck abelian category in which every filtered colimit of CM-injective objects is CM-injective. Then $\AA$ is coherently assembled.
\end{proposition}

\begin{proof}
Since $\tYY$ is a left adjoint, it is right exact, so it suffices to show that it preserves kernels of epimorphisms. So let $\Sigma \colon 0 \to x \xrightarrow{\iota} e \xrightarrow{\pi} z \to 0$ be exact in $\AA$, let $K = \Ker(\tYY \pi)$ and let $q \colon \tYY x \to K$ be the induced epimorphism; we must show that $q$ is an isomorphism. For this it is enough to prove that for every $N \in \Ind(\AA)$ the injective map $q^* \colon \Hom(K, N) \to \Hom(\tYY x, N)$ is surjective: taking $N = \tYY x$ we then find $h$ with $h q = \Id$, so that the epimorphism $q$ is also a split monomorphism. Writing $N = \{n_i\}_{i \in I}$ and $n = \colim N$, we have to show that every $g \colon x \to n$ is of the form $\colim(h q)$ for some $h \colon K \to N$.

Since $\AA$ is a Grothendieck category, it admits a functorial injective embedding $\varepsilon \colon \Id_{\AA} \to J$. Put $E := \{J(n_i)\}_{i \in I} \in \Ind(\AA)$ and let $\varepsilon \colon N \to E$ be the induced morphism; it is the filtered colimit of the monomorphisms $\YY(\varepsilon_{n_i})$, hence a monomorphism in $\Ind(\AA)$, and we write $\beta \colon E \to E/N$ for its cokernel. Let $E^{\infty} := \colim E = \indlim[i] J(n_i)$. By assumption, $E^{\infty}$ is CM-injective, being a filtered colimit of injective objects; moreover $\colim \varepsilon \colon n \to E^{\infty}$ is a monomorphism and $\colim(E/N) = E^{\infty} / n$, by exactness of $\colim$.

Consider $g' := \colim(\varepsilon) \circ g \colon x \to E^{\infty}$ and the push-out $\Sigma' := g'_* \Sigma \colon 0 \to E^{\infty} \to e' \to z \to 0$. By Lemma~\ref{lem_cm_ses_cminj} the row $\tYY \Sigma'$ is exact in $\Ind(\AA)$. Let $\alpha \colon \tYY E^{\infty} \to E$ correspond to $\Id_{E^{\infty}}$ under the $\tYY \dashv \colim$ adjunction, so that $\colim \alpha = \Id_{E^{\infty}}$, and put $\theta := \alpha_* [\tYY \Sigma'] \in \Ext^1_{\Ind(\AA)}(\tYY z, E)$. As $\colim$ is exact and commutes with push-outs,
\[
\gamma_z (\theta) = (\colim \alpha)_* [\colim \tYY \Sigma'] = [\Sigma'] = g'_* [\Sigma] .
\]
By Lemma~\ref{lem_cm_ext_compare}(ii), applied to $M = E$, there is $h' \colon K \to E$ with $\colim(h' q) = g'$.

It remains to see that $h'$ takes values in $N$. We have
\[
\colim(\beta \circ h' \circ q) = \colim(\beta) \circ g' = (E^{\infty} \to E^{\infty}/n) \circ \colim(\varepsilon) \circ g = 0,
\]
so $\beta h' q = 0$ by the $\tYY \dashv \colim$ adjunction, and $\beta h' = 0$ because $q$ is an epimorphism. Hence $h' = \varepsilon \circ h$ for a unique $h \colon K \to N = \Ker \beta$. Finally $\colim(\varepsilon) \circ \colim(hq) = g' = \colim(\varepsilon) \circ g$, and $\colim \varepsilon$ is a monomorphism, so $\colim(hq) = g$.
\end{proof}

\begin{proof}[Proof of Theorem~\ref{th_cminj}]
\textbf{(a) $\Rightarrow$ (b)} is Proposition~\ref{prop_cm_descent}.

\textbf{(b) $\Rightarrow$ (c).} Let $\{x_i\}_{i \in I}$ be a filtered system of CM-injective objects with colimit $x$, let $\phi \colon y \to z$ be compact and $\xi \in \Ext^1(z, x)$. Factor $\phi = \psi \circ \chi$ with $\chi \colon y \to w$ and $\psi \colon w \to z$ compact (Lemma~\ref{lem_cm_factor}). By (b) there are $i_0 \in I$ and $\zeta \in \Ext^1(w, x_{i_0})$ with $\psi^* \xi = (x_{i_0} \to x)_* \zeta$. Since $x_{i_0}$ is CM-injective and $\chi$ is compact, $\chi^* \zeta = 0$, and by bifunctoriality of $\Ext^1$
\[
\phi^* \xi = \chi^* \psi^* \xi = \chi^* (x_{i_0} \to x)_* \zeta = (x_{i_0} \to x)_* \chi^* \zeta = 0 .
\]

\textbf{(c) $\Rightarrow$ (a)} is Proposition~\ref{prop_cm_reverse}.
\end{proof}

\begin{remark} \label{rem_cminj_final}
\begin{enumerate}
    \item If $\AA$ is locally finitely presentable, then coherent assembly is local coherence (Proposition~\ref{prop_loccoh=cohass}) and CM-injectivity is FP-injectivity (Proposition~\ref{prop_cminj_vs_fpinj}), so Theorem~\ref{th_cminj} recovers the equivalence (a)$\iff$(d) of Theorem~\ref{th_fpinj_ab}. Note that the proof of (a)$\Rightarrow$(c) does not use projective objects: the classical argument of Stenstr{\"o}m resolves a finitely presented module by finitely generated free modules and uses coherence to keep the syzygies finitely presented, whereas here this finiteness of syzygies is built into $\Ind(\AA)$ (Lemma~\ref{lem_cm_ext_ind}(iii)), and coherent assembly enters only through the exactness of $\tYY$ in Step~1 of Proposition~\ref{prop_cm_descent}. Condition (b) is a version of the surjectivity of the map $\xi_1$ of Theorem~\ref{th_fpinj_ab}(b), relative to compact morphisms; for $\AA$ locally finitely presentable, $w$ finitely presented and $\psi = \Id_w$ it is exactly the surjectivity of $\xi_1$.

    \item In view of Remark~\ref{rem_cminj_ext_closed}, the coherence transfer theorem at the end of Subsection~\ref{ssec:fpinj} admits a verbatim analogue for compactly assembled Grothendieck abelian categories $\AA$, $\BB$, $\CC$: if the four functors preserve CM-injective objects, $R^1 i^!$ vanishes on CM-injective objects and the remaining hypotheses are kept, then $\BB$ is coherently assembled as soon as $\AA$ and $\CC$ are. The proof is the same, with Theorem~\ref{th_cminj} in place of Theorem~\ref{th_fpinj_ab}.

    \item One expects an analogue of condition (c) of Theorem~\ref{th_fpinj_ab} for higher $\Ext^n$: a compact morphism factors into a composition of $n+1$ compact morphisms (Lemma~\ref{lem_cm_factor}), and one would like to kill $\phi^* \xi$, $\xi \in \Ext^n(z, \indlim x_i)$, one step at a time. This is closely related to the higher derived functors $L^i \tYY$ and the notion of $n$-coherent assembly; we do not pursue it here.
\end{enumerate}
\end{remark}

\subsubsection{Kernels of compact morphisms} \label{sssec:kernels_compact}

One could expect coherent assembly to admit a description of a more elementary kind, parallel to the definition of local coherence through kernels of morphisms between finitely presented objects: since compact morphisms replace finitely presented objects in the compactly assembled world, the natural candidate is the requirement that kernels of morphisms between compact morphisms be compact. In this subsection we show that this requirement is strictly stronger than coherent assembly, and that in the locally finitely presentable case it characterises local Noetherianity rather than local coherence. Throughout, $\AA$ is a compactly assembled Grothendieck abelian category, and we keep the notation of Remarks~\ref{rem_ind_basic}, \ref{rem_tYY_basic} and \ref{rem_compact_lifts}.

\begin{definition} \label{def_NA}
Consider the following two conditions on $\AA$.
\begin{itemize}
    \item[(NA)] For every morphism of arrows $(g_1, g_3) \colon (\beta \colon B_1 \to B_3) \to (\gamma \colon C_1 \to C_3)$ in $\Fun([1], \AA)$ with $\beta$ and $\gamma$ compact, the induced morphism $\ker g_1 \to \ker g_3$ is compact.
    \item[(NA$'$)] Compact morphisms are stable under base change along monomorphisms: if $\psi \colon w \to e$ is compact and $\kappa \hookrightarrow e$ is a monomorphism, then the projection $w \times_e \kappa \to \kappa$ is compact.
\end{itemize}
\end{definition}

\begin{lemma} \label{lem_NA_equiv}
Conditions (NA) and (NA$'$) are equivalent.
\end{lemma}

\begin{proof}
(NA$'$)$\Rightarrow$(NA). Let $(g_1, g_3) \colon \beta \to \gamma$ be as in (NA). Since $g_3 \beta = \gamma g_1$, the morphism $\beta$ restricts to $\ker g_1 \to \ker g_3$, and this restriction factors as $\ker g_1 \hookrightarrow B_1 \times_{B_3} \ker g_3 \to \ker g_3$, where the second morphism is the base change of $\beta$ along $\ker g_3 \hookrightarrow B_3$, compact by (NA$'$). Precomposition with a monomorphism preserves compactness (Remark~\ref{rem_compact_morphism_basic}).

(NA)$\Rightarrow$(NA$'$). Let $\psi \colon w \to e$ be compact, $\kappa \hookrightarrow e$ a monomorphism with quotient $\pi \colon e \to e/\kappa$, and $p_1, p_2$ the projections of $w \times_e \kappa$. The morphism $\psi' := \psi p_1 \colon w \times_e \kappa \to e$ is compact and $\pi \psi' = 0$. Hence the pair $(w \times_e \kappa \to 0,\ \pi)$ is a morphism of compact arrows $(\psi' \colon w \times_e \kappa \to e) \to (0 \colon 0 \to e/\kappa)$, and its kernel arrow is $\ker(w \times_e \kappa \to 0) = w \times_e \kappa \to \ker \pi = \kappa$, i.e.\ $p_2$. By (NA) it is compact.
\end{proof}

To relate (NA) to coherent assembly we need two complements to Lemma~\ref{lem_tYY_compact_transitions}: a presentation of $\tYY x$ with compact transition morphisms is unique, and a morphism $\tYY g$ can be presented levelwise.

\begin{lemma} \label{lem_compact_transitions_are_tYY}
Let $\CC$ be a compactly assembled category and let $\{x_i\}_{i \in I}$ be a filtered system in $\CC$ with compact transition morphisms $f_{ii'}$ and colimit $x$, with canonical morphisms $\sigma_i \colon x_i \to x$. Then the morphism $\alpha \colon \tYY x \to \{x_i\}_{i \in I}$ corresponding to $\Id_x$ under the $\tYY \dashv \colim$ adjunction is an isomorphism in $\Ind(\CC)$.
\end{lemma}

\begin{proof}
By the convention of Remark~\ref{rem_filtered_convention} we may assume that $I$ is a directed poset. We construct an inverse $\beta \colon \{x_i\} = \colim[i] \YY x_i \to \tYY x = \colim[i] \tYY x_i$. For $i \in I$ choose $i' > i$ and a lift $\hat{f}_{ii'} \colon \YY x_i \to \tYY x_{i'}$ of $f_{ii'}$, and put $\beta_i := \tYY(\sigma_{i'}) \circ \hat{f}_{ii'} \colon \YY x_i \to \tYY x$.

This does not depend on the choices. Two lifts of $f_{ii'}$ differ by a morphism $\theta \colon \YY x_i \to \Ker \rho_{x_{i'}}$. Choose $i'' > i'$; by Remark~\ref{rem_compact_lifts}, $\tYY f_{i'i''} = \hat{f}_{i'i''} \circ \rho_{x_{i'}}$ kills $\Ker \rho_{x_{i'}}$, hence so does $\tYY \sigma_{i'} = \tYY \sigma_{i''} \circ \tYY f_{i'i''}$, and $\beta_i$ is independent of the lift; independence of $i'$ follows by composing lifts. Since $\hat{f}_{jj'} \circ \YY f_{ij}$ is a lift of $f_{ij'}$, we have $\beta_j \circ \YY f_{ij} = \beta_i$, so the $\beta_i$ define $\beta$.

By the adjunction, $\alpha \circ \tYY \sigma_{i'} = \mathrm{in}_{i'} \circ \rho_{x_{i'}}$, both sides having colimit $\sigma_{i'}$. Hence $\alpha \beta \, \mathrm{in}_i = \mathrm{in}_{i'} \rho_{x_{i'}} \hat{f}_{ii'} = \mathrm{in}_{i'} \YY f_{ii'} = \mathrm{in}_i$, i.e.\ $\alpha \beta = \Id$. Finally $\colim(\beta \, \mathrm{in}_i) = \sigma_{i'} f_{ii'} = \sigma_i$, so $\colim \beta = \Id_x = \colim \alpha$, and $\beta \alpha = \Id$ by the adjunction, both being endomorphisms of $\tYY x$ with colimit $\Id_x$.
\end{proof}

\begin{lemma} \label{lem_tYY_common_presentation}
Let $\CC$ be a compactly assembled category and $g \colon y \to z$ a morphism in $\CC$. There exist a directed poset $K$ and filtered systems $\{y_k\}_{k \in K}$, $\{z_k\}_{k \in K}$ with compact transition morphisms, isomorphisms $\{y_k\} \cong \tYY y$ and $\{z_k\} \cong \tYY z$ in $\Ind(\CC)$, and morphisms $g_k \colon y_k \to z_k$ commuting with the transition morphisms, such that $\tYY g$ is represented by $\{g_k\}_{k \in K}$.
\end{lemma}

\begin{proof}
By Lemma~\ref{lem_tYY_compact_transitions} choose presentations $\tYY y = \{y_i\}_{i \in I}$ and $\tYY z = \{z_j\}_{j \in J}$ over directed posets with compact transition morphisms $y_{ii'}$, $z_{jj'}$. The morphism $\tYY g$ is an element of $\varprojlim_i \varinjlim_j \Hom(y_i, z_j)$; let $K$ be the set of triples $(i, j, \phi)$, $\phi \colon y_i \to z_j$ representing the $i$-th component of $\tYY g$, ordered by $(i, j, \phi) \leq (i', j', \phi')$ iff $i \leq i'$, $j \leq j'$ and $\phi' y_{ii'} = z_{jj'} \phi$. This is the usual reindexing of a morphism of ind-objects over a common filtered index category (\cite[Expos\'e I, \S 8.1]{SGA4}, \cite[\S\S 6.1, 6.4]{Kashiwara_Schapira}): $K$ is directed and the projections $K \to I$, $K \to J$ are order preserving and cofinal. Setting $y_k := y_i$, $z_k := z_j$ and $g_k := \phi$ for $k = (i, j, \phi)$ gives the required data; the transition morphisms of the new systems are transition morphisms of the old ones, hence compact.
\end{proof}

\begin{proposition} \label{prop_NA_cohass}
If $\AA$ satisfies (NA), then $\AA$ is coherently assembled.
\end{proposition}

\begin{proof}
Let $g \colon y \to z$ be a morphism in $\AA$; we show that the canonical morphism $\tYY(\ker g) \to \Ker(\tYY g)$ is an isomorphism. Present $\tYY g$ levelwise by $g_k \colon y_k \to z_k$ as in Lemma~\ref{lem_tYY_common_presentation}. Kernels in $\Ind(\AA)$ of levelwise morphisms are computed levelwise, since $\YY$ is exact and filtered colimits in $\Ind(\AA)$ are exact; thus $\Ker(\tYY g) = \{\ker g_k\}_{k \in K}$. For $k \leq k'$ the pair $(g_k, g_{k'})$ is a morphism of compact arrows $(y_k \to y_{k'}) \to (z_k \to z_{k'})$, so by (NA) the transition morphism $\ker g_k \to \ker g_{k'}$ is compact. By (AB5), $\colim \{\ker g_k\} = \indlim[k] \ker g_k = \ker g$. Lemma~\ref{lem_compact_transitions_are_tYY} now identifies $\{\ker g_k\}$ with $\tYY(\ker g)$, and the identification is the canonical comparison morphism, both having colimit $\Id_{\ker g}$. Hence $\tYY$ preserves kernels.
\end{proof}

Note that (AB5) enters only through the identification $\indlim[k] \ker g_k = \ker g$; the remaining steps use compact assembly alone. The converse to Proposition~\ref{prop_NA_cohass} fails, and in the locally finitely presentable case (NA) is a Noetherian condition.

\begin{proposition} \label{prop_NA_lfp}
Let $\AA$ be an AB5 locally finitely presentable abelian category. The following are equivalent:
\begin{itemize}
    \item[(i)] $\AA$ satisfies (NA);
    \item[(ii)] $\AA$ satisfies (NA$'$);
    \item[(iii)] every subobject of a finitely presented object of $\AA$ is finitely presented;
    \item[(iv)] $\AA$ is locally Noetherian.
\end{itemize}
\end{proposition}

\begin{proof}
(i)$\Leftrightarrow$(ii) is Lemma~\ref{lem_NA_equiv}, and (iii)$\Leftrightarrow$(iv) follows from Theorem~\ref{th_fg_fg_fp}: a finitely presented object is Noetherian iff all its subobjects are finitely generated, and a finitely generated subobject $U$ of a finitely presented object $X$ is finitely presented iff $X/U$ is, by Lemma~\ref{lem_ab_fp_fg_via_ker_coker}. Recall from Proposition~\ref{prop_cminj_vs_fpinj} that a morphism of $\AA$ is compact iff it factors through a finitely presented object.

(ii)$\Rightarrow$(iii). Let $X$ be finitely presented and $\iota \colon \kappa \hookrightarrow X$ any subobject. The morphism $\iota$ is compact, its target being finitely presented, and its base change along the monomorphism $\iota$ itself is $\Id_\kappa$. By (NA$'$), $\Id_\kappa$ is compact, so $\kappa$ is a retract of a finitely presented object and hence finitely presented.

(iii)$\Rightarrow$(ii). Let $\psi \colon w \to e$ be compact, factored as $w \to v \to e$ with $v$ finitely presented, and let $\kappa \hookrightarrow e$. Then $w \times_e \kappa \to \kappa$ factors through $v \times_e \kappa \to \kappa$, and $v \times_e \kappa$ is a subobject of $v$, finitely presented by (iii). Hence $w \times_e \kappa \to \kappa$ factors through a finitely presented object and is compact.
\end{proof}

\begin{example} \label{ex_NA_fails}
Let $R = k[t_1, t_2, \dots]$ be the polynomial ring in countably many variables over a field, and $\mathfrak{m} = (t_1, t_2, \dots)$. The ring $R$ is coherent but not Noetherian, so $R - \Mod$ is coherently assembled (Proposition~\ref{prop_loccoh=cohass}) but, by Proposition~\ref{prop_NA_lfp}, does not satisfy (NA). Explicitly, consider the morphism of arrows
\[
(\beta \colon \mathfrak{m} \xrightarrow{\ \iota\ } R) \xrightarrow{\ (u,\, q)\ } (\gamma \colon \mathfrak{m}/\mathfrak{m}^2 \xrightarrow{\ 0\ } k),
\]
where $\iota$ is the inclusion, $u$ and $q$ are the quotient maps and $k = R/\mathfrak{m}$. The square commutes, $\beta$ is compact because its target is finitely presented, and $\gamma = 0$ is compact. The kernel arrow is the inclusion $\mathfrak{m}^2 \hookrightarrow \mathfrak{m}$, and it is not compact: a factorisation through a finitely presented module would produce a finitely generated submodule $L$ with $\mathfrak{m}^2 \subseteq L \subseteq \mathfrak{m}$, but any finitely generated $L \subseteq \mathfrak{m}$ is contained in $(t_1, \dots, t_n)$ for some $n$, whereas $t_{n+1}^2 \in \mathfrak{m}^2$. In the form (NA$'$): $\Id_{\mathfrak{m}}$ is the base change of the compact morphism $\iota$ along $\mathfrak{m} \subseteq R$, and $\mathfrak{m}$ is not finitely presented.
\end{example}

\begin{remark}
Proposition~\ref{prop_NA_lfp} is consistent with \cite[Proposition 1.44]{efimov_Th_of_heart}: a Grothendieck abelian category is locally Noetherian iff it is compactly assembled and $\rho_x \colon \tYY x \to \YY x$ is a monomorphism for every $x$, i.e.\ iff lifts of compact morphisms are unique. Uniqueness of lifts is exactly what one needs to pass from compactness of the components of a morphism of arrows to compactness of the morphism in $\Fun([1], \AA)$, and this is the mechanism behind (NA). Thus (NA) is best viewed as a Noetherian-type finiteness condition, and Theorem~\ref{th_cminj} shows that the correct generalisation of coherence to compactly assembled categories is the injective-side condition (c) rather than a kernel condition on compact morphisms.
\end{remark}

\bibliography{references}

\bigskip
\footnotesize
\noindent\textsc{Mikhail Dusman}\par
\noindent\textit{Email address}: \href{mailto:michail.dusman@gmail.com}{dusman.ma@phystech.edu}

\medskip
\noindent\textsc{Alexander Pavlov}\par
\noindent\textit{Email address}: \href{mailto:sasha.pavlov@gmail.com}{abpavlov@hse.ru}

\end{document}